\documentclass[reqno]{amsart}
\usepackage{amssymb}
\usepackage{amsmath}
\usepackage{amsthm}
\usepackage{pgfplots} 
\usepackage{graphicx}
\usepackage[utf8x]{inputenc}
\usepackage{tikz-cd}
\usepackage{bbm}
\usepackage{bm}
\usepackage{subcaption}
\usepackage[boxruled]{algorithm2e}
\usepackage{mathtools}
\usepackage{lipsum}
\usepackage{tabularx}
\usepackage[title,titletoc]{appendix}
\usepackage{booktabs}
\usepackage{calc}
\usepackage{array}
\usepackage{tikz}
\usepackage{boites}
\usepackage{here}
\usepackage[]{hyperref}
\usepackage{natbib,enumerate}
\usepackage{enumitem}
\usepackage{cleveref}
\usepackage{adjustbox}
\setlist[itemize]{itemsep=2pt, topsep=2pt, leftmargin=*}
\setlist[enumerate]{itemsep=2pt, topsep=2pt, leftmargin =*}
\usetikzlibrary{arrows}
\usepackage{extarrows}
\usepackage{ stmaryrd }
\usepackage{ dsfont }
\usepackage{ longtable }
\usepackage{comment}
\usepackage{nameref}
\usepackage{hyperref}
\pgfplotsset{compat=1.17} %%suggested by the compiler

\usepackage{etoolbox}
\usepackage[margin=1in]{geometry}
\definecolor{BrickRed}{rgb}{0.6, 0.1875, 0.2475}
\definecolor{Cerulean}{rgb}{0.0, 0.36, 0.4875}
\definecolor{MyGreen}{rgb}{0.1946, 0.3887, 0.0551}

\newcounter{introcounter}

\newtheoremstyle{introtheorem}
  {1\baselineskip\@plus.2\baselineskip\@minus.2\baselineskip}% Space above
  {1\baselineskip\@plus.2\baselineskip\@minus.2\baselineskip}% Space below
  {\slshape}% Body font
  {}%Indent amount (empty = no indent, \parindent = para indent)
  {\bfseries}%  Thm head font
  {.}%       Punctuation after thm head
  { }%      Space after thm head: " " = normal interword space;
  {}%       Thm head spec

\theoremstyle{introtheorem}
\newtheorem{introthm}{Theorem}

\newtheorem{introprop}[introthm]{Proposition}

\theoremstyle{plain}
\newtheorem{thm}{Theorem}[section]
\newtheorem{Prop}[thm]{Proposition}
\newtheorem{Lemma}[thm]{Lemma}

\theoremstyle{definition}
\newtheorem{Def}[thm]{Definition}

\theoremstyle{remark}

\newtheorem{Rmk}[thm]{Remark}
\newtheorem{Question}[thm]{Question}

\crefname{introthm}{Theorem}{Theorems}
\crefname{introprop}{Proposition}{Propositions}
\crefname{Lemma}{Lemma}{Lemmas}
\crefname{Def}{Definition}{Definitions}
\crefname{Question}{Question}{Questions}
\crefname{Prop}{Proposition}{Propositions}
\crefname{thm}{Theorem}{Theorems}
\crefname{Construction}{Construction}{Constructions}
\crefname{Question}{Question}{Questions}
\crefname{Rmk}{Remark}{Remarks}
\crefname{Notation}{Notation}{Notations}
\crefname{section}{Section}{Sections}
\crefname{subsection}{Subsection}{Subsections}
\crefname{Ex}{Example}{Examples}

\makeatletter
\newlength{\temp@boite}
\newlength{\saveparindent}
\def\encadrement#1#2{%
  \def\bkvz@before@breakbox{\ifhmode\par\fi\vskip5pt\vskip\breakboxskip\relax}%
  \fboxrule=0.4pt
  \fboxsep=5pt
  \def\bkvz@set@linewidth{\advance\linewidth -2\fboxrule
                          \advance\linewidth -2\fboxsep}%
  \def\bkvz@left{\color{#2}\vrule \@width\fboxrule\hskip\fboxsep\color{black}}%
  \def\bkvz@right{\color{#2}\hskip\fboxsep\vrule \@width\fboxrule\color{black}}%
  \def\bkvz@top{\hbox to \hsize{%
      \color{#2}%
      \setlength{\temp@boite}{\fboxrule+0.5ex}%
      \vrule\@width\fboxrule\@height \temp@boite %
      \rule[0.5ex]{2em}{\fboxrule}%
      {#1}%
      \setlength{\temp@boite}{\textwidth-2em-\widthof{#1}-2\fboxrule}%
      \rule[0.5ex]{\temp@boite}{\fboxrule}%
      \setlength{\temp@boite}{\fboxrule+0.5ex}%
      \vrule\@width\fboxrule\@height \temp@boite}}%
  \def\bkvz@bottom{\color{#2}\hrule\@height\fboxrule}%
  \breakbox\vspace{0pt}}

\makeatother

\makeatletter
\def\encadrementombre#1#2{%
  \fboxsep=5pt
  \fboxrule=0.4pt
  \def\bkvz@before@breakbox{%
    \ifhmode\par\fi\vskip5pt\vskip\breakboxskip\relax}%
  \def\bkvz@set@linewidth{%
    \advance\linewidth -2\fboxrule 
    \setlength{\fboxsep}{5pt}\advance\linewidth -2\fboxsep
  }%
  \def\bk@line{\hbox to \linewidth{%
      \ifbkcount\smash{\llap{\the\bk@lcnt\ }}\fi
      \setlength{\fboxsep}{0pt}\colorbox{#2!15}{%
        \setlength{\fboxsep}{5pt}%
        {\color{#2}\vrule width \fboxrule}\hskip\fboxsep
        \box\bk@bxa
        \hskip\fboxsep{\color{#2}\vrule width\fboxrule}%
        }%
      }}%
  \def\bkvz@top{\hbox to \hsize{%
      \setlength{\temp@boite}{\fboxrule+0.5ex}%
      \color{#2!15}\vrule\@width\textwidth\@height\temp@boite %
      \hspace{-\textwidth}%
      \color{#2}%
      \vrule\@width\fboxrule\@height \temp@boite %
      \rule[0.5ex]{2em}{\fboxrule}%
      {\fboxsep=2pt\colorbox{white}{#1}}%
      \setlength{\temp@boite}{\textwidth-2em-\widthof{\fboxsep=2pt\colorbox{white}{#1}}-2\fboxrule}%
      \rule[0.5ex]{\temp@boite}{\fboxrule}%
      \setlength{\temp@boite}{\fboxrule+0.5ex}%
      \vrule\@width\fboxrule\@height \temp@boite %
}}%
  \def\bkvz@bottom{{\color{teal}\hrule\@height\fboxrule}}%
  \color{black}\breakbox}%

\makeatother

\makeatletter
\def\encadrementombrevar#1#2{%
  \fboxsep=5pt
  \fboxrule=0.4pt
  \def\bkvz@before@breakbox{%
    \ifhmode\par\fi\vskip5pt\vskip\breakboxskip\relax}%
  \def\bkvz@set@linewidth{%
    \advance\linewidth -2\fboxrule 
    \setlength{\fboxsep}{5pt}\advance\linewidth -2\fboxsep
  }%
  \def\bk@line{\hbox to \linewidth{%
      \ifbkcount\smash{\llap{\the\bk@lcnt\ }}\fi
      \setlength{\fboxsep}{0pt}\colorbox{#2!15}{%
        \setlength{\fboxsep}{5pt}%
        {\color{#2}\vrule width \fboxrule}\hskip\fboxsep
        \box\bk@bxa
        \hskip\fboxsep{\color{#2}\vrule width\fboxrule}%
        }%
      }}%
  \def\bkvz@top{\hbox to \hsize{%
      \setlength{\temp@boite}{\fboxrule+0.5ex}%
      \color{#2!15}\vrule\@width\textwidth\@height\temp@boite %
      \hspace{-\textwidth}%
      \color{#2}%
      \vrule\@width\fboxrule\@height \temp@boite %
      \rule[0.5ex]{2em}{\fboxrule}%
      {\fboxsep=2pt\fcolorbox{#2}{white}{#1}}%
      \setlength{\temp@boite}{\textwidth-2em-\widthof{\fboxsep=2pt\fcolorbox{#2}{#2!15}{#1}}-2\fboxrule}%
      \rule[0.5ex]{\temp@boite}{\fboxrule}%
      \setlength{\temp@boite}{\fboxrule+0.5ex}%
      \vrule\@width\fboxrule\@height \temp@boite %
}}%
  \def\bkvz@bottom{{\color{#2}\hrule\@height\fboxrule}}%
  \color{black}\breakbox}%

\makeatother

\newcommand{\N}{\mathbb{N}}
\newcommand{\Z}{\mathbb{Z}}
\newcommand{\R}{\mathbb{R}}
\newcommand{\C}{\mathbb{C}}
\newcommand{\Q}{\mathbb{Q}}
\newcommand{\id}[1]{\ensuremath{\mathsf{id}_{#1}}}

\newcommand{\colim}[0]{\ensuremath{\mathsf{colim}}}

\newcommand{\unit}[0]{\ensuremath{\mathds{1}}}
\newcommand{\E}[1]{\ensuremath{\mathcal{E}_{#1}}}
\newcommand{\Einfty}[0]{\ensuremath{\E{\infty}}}

\newcommand{\factor}[2]{\left. \raise 2pt\hbox{\ensuremath{#1}} \right/
 	  \hskip -2pt\raise -2pt\hbox{\ensuremath{#2}}}

\newcommand{\Affl}{{\mathbb{A}^1} }%%affine line
\newcommand{\Gm}{\mathbb{G}_m}
\newcommand{\Sph}{\mathbb{S}}
\newcommand{\calS}{\mathcal{S}}
\newcommand{\Spec}[0]{ \mathsf{Spec} }

\newcommand{\Smash}{\ensuremath{\wedge}}
\newcommand{\ko}[0]{\ensuremath{\mathsf{ko}}}

\newcommand{\kw}[0]{\ensuremath{\mathsf{kw}}}
\newcommand{\KO}[0]{\ensuremath{\mathsf{KO}}}
\newcommand{\kgl}[0]{\ensuremath{\mathsf{kgl}}}
\newcommand{\KGL}[0]{\ensuremath{\mathsf{KGL}}}
\newcommand{\KU}[0]{\ensuremath{\mathsf{KU}}}
\newcommand{\KW}[0]{\ensuremath{\mathsf{KW}}}
\newcommand{\MSL}[0]{\ensuremath{\mathsf{MSL}}}
\newcommand{\MO}[0]{\ensuremath{\mathsf{MO}}}
\newcommand{\MU}[0]{\ensuremath{\mathsf{MU}}}
\newcommand{\MSp}[0]{\ensuremath{\mathsf{MSp}}}
\newcommand{\MGL}[0]{\ensuremath{\mathsf{MGL}}}
\newcommand{\MSO}[0]{\ensuremath{\mathsf{MSO}}}

\newcommand{\kotop}[0]{\ensuremath{\ko^\mathsf{top}}}
\newcommand{\Mtop}[0]{\ensuremath{M_\mathsf{top}}}
\newcommand{\KOtop}[0]{\ensuremath{\KO^\mathsf{top}}}
\newcommand{\rR}[0]{\ensuremath{r_{\R}}}
\newcommand{\rC}[0]{\ensuremath{r_{\mathbb{C}}}}
\newcommand{\HZ}[0]{\ensuremath{\mathsf{H}\Z}}
\newcommand{\HQ}[0]{\ensuremath{\mathsf{H}\Q}}
\newcommand{\HA}[0]{\ensuremath{\mathsf{HA}}}
\newcommand{\Zmod}[0]{\ensuremath{\Z /2}}
\newcommand{\HZmod}[0]{\ensuremath{\mathsf{H}\Z/2}}
\newcommand{\HZtilde}[0]{\ensuremath{\widetilde{\HZ}}}
\newcommand{\homsh}[3]{\ensuremath{\underline{\pi}_{#1}(#3)_{#2}}}
\newcommand{\calR}[0]{\ensuremath{\mathcal{R}}}

\newcommand{\calC}[0]{\ensuremath{\mathcal{C}}}

\newcommand{\calF}[0]{\ensuremath{\mathcal{F}}}

\newcommand{\Sq}[0]{\ensuremath{\mathsf{Sq}}}

\newcommand{\Sp}{\mathsf{Sp}}

\newcommand{\map}[0]{\ensuremath{\mathsf{Map}}}
\newcommand{\Sm}{\ensuremath{\mathsf{Sm}}}
\newcommand{\Spc}[0]{\ensuremath{ \mathsf{Spc}}}
\newcommand{\spaces}[0]{\Spc}
\newcommand{\Pre}{\ensuremath{\mathcal{P}}}
\newcommand{\SHR}[0]{\ensuremath{\mathsf{SH}(\R)}}
\newcommand{\SH}[0]{\ensuremath{\mathsf{SH}}}
\newcommand{\SHk}[0]{\ensuremath{\mathsf{SH}(k)}}
\newcommand{\eff}[0]{\ensuremath{\mathsf{eff}}}
\newcommand{\veff}[0]{\ensuremath{\mathsf{veff}}}

\newcommand{\Alg}[0]{\ensuremath{\mathsf{Alg}}}
\newcommand{\CAlg}[0]{\ensuremath{\mathsf{CAlg}}}
\newcommand{\Mod}[0]{\ensuremath{\mathsf{Mod}}}

\newcommand{\Cat}[0]{\ensuremath{\mathsf{Cat}_\infty}}
\newcommand{\wgpd}[0]{\ensuremath{{\omega,\simeq}}}
\newcommand{\kgpd}[0]{\ensuremath{{\kappa,\simeq}}}

\newcommand{\lcofib}[1]{\ensuremath{\xhookrightarrow{\ #1\ }}}

\newcommand{\lsimeq}[1]{\ensuremath{\xrightarrow[\ #1\ ]{\raisebox{-0.4ex}{\tiny $\simeq$}}}}

\tikzset{
    weq/.style={anchor=south, rotate=90, inner sep=.5mm}
}

\usetikzlibrary{calc}
\usetikzlibrary{decorations.pathmorphing}

\tikzset{curve/.style={settings={#1},to path={(\tikztostart)
    .. controls ($(\tikztostart)!\pv{pos}!(\tikztotarget)!\pv{height}!270:(\tikztotarget)$)
    and ($(\tikztostart)!1-\pv{pos}!(\tikztotarget)!\pv{height}!270:(\tikztotarget)$)
    .. (\tikztotarget)\tikztonodes}},
    settings/.code={\tikzset{quiver/.cd,#1}
        \def\pv##1{\pgfkeysvalueof{/tikz/quiver/##1}}},
    quiver/.cd,pos/.initial=0.35,height/.initial=0}

\tikzset{tail reversed/.code={\pgfsetarrowsstart{tikzcd to}}}
\tikzset{2tail/.code={\pgfsetarrowsstart{Implies[reversed]}}}
\tikzset{2tail reversed/.code={\pgfsetarrowsstart{Implies}}}
\tikzset{no body/.style={/tikz/dash pattern=on 0 off 1mm}}

\usepackage{subfiles}

\makeatletter
\def\@tocline#1#2#3#4#5#6#7{\relax
  \ifnum #1>\c@tocdepth % then omit
  \else
    \par \addpenalty\@secpenalty\addvspace{#2}%
    \begingroup \hyphenpenalty\@M
    \@ifempty{#4}{%
      \@tempdima\csname r@tocindent\number#1\endcsname\relax
    }{%
      \@tempdima#4\relax
    }%
    \parindent\z@ \leftskip#3\relax \advance\leftskip\@tempdima\relax
    \rightskip\@pnumwidth plus4em \parfillskip-\@pnumwidth
    #5\leavevmode\hskip-\@tempdima
      \ifcase #1
       \or\or \hskip 2.5em \or \hskip 5em \else \hskip 7.5em \fi%
      #6\nobreak\relax
    \dotfill\hbox to\@pnumwidth{\@tocpagenum{#7}}\par% <---- \dotfill -> \hfill
    \nobreak
    \endgroup
  \fi}
\makeatother

\begin{document}

\title[The real Betti realization of very effective Hermitian K-theory]{The real Betti realization of very effective Hermitian K-theory}
\date{\today}

\author{Julie Bannwart}
\address{Institut für Mathematik, JGU Mainz, Germany}
\email{\href{bannwart.julie@gmail.com}{bannwart.julie@gmail.com}}

\begin{abstract}
    Real Betti realization is a symmetric monoidal functor from the category of motivic spectra to that of topological spectra, extending the functor that associates to a scheme over $\R$ the space of its real points. In this article, we identify the $\E{1}$-ring structure on the real realization of $\ko$, the very effective cover of Hermitian K-theory $\KO$, by an explicit 2-local fracture square, as being equivalent to $\mathsf{L}(\R)_{\geq 0}$, the connective L-theory spectrum of $\R$. Along the way, we also compute the real realization of other motivic ring spectra: motivic cohomology $\HZ$ and $\HZmod$, the variant $\widetilde{\HZ}$, $\KO$, the very effective cover $\kgl$ of algebraic K-theory $\KGL$.
\end{abstract}
\keywords{Real Betti realization, very effective Hermitian K-theory, $\KO$ (or $\mathsf{KQ}$), $\ko$ (or $\mathsf{kq}$).}

\maketitle
\tableofcontents

%\newpage
%\makeatletter
%\providecommand\@dotsep{5}
%\makeatother
%\listoftodos\relax

\pagenumbering{arabic}

\newpage
\section*{Introduction}\label{Sect:introduction}

Motivic homotopy theory, mainly introduced by Morel and Voevodsky \cite{MV}, provides a setting to study algebraic geometry and cohomology theories for schemes using the methods of homotopy theory. We work in the ($\infty$-)categories of motivic spaces $\Spc(S)$ and of motivic spectra $\SH(S)$ over a base scheme $S$. The former is obtained by considering the subcategory of presheaves of spaces on $\Sm_S$ (smooth schemes of finite type over $S$) consisting of those presheaves that are $\Affl$-invariant and are sheaves in the Nisnevich topology. Namely, we impose Nisnevich descent and contractibility of the affine line via a Bousfield localization. In a formula, $\Spc(S)  = \mathsf{L}_{\Affl, \mathsf{Nis}}(\Pre(\Sm_S))$. The category of motivic spectra $\SH(S)  = \Spc(S)_\ast[(\mathbb{P}^1)^{-1}]$ is then obtained from that of pointed motivic spaces by $\mathbb{P}^1$-stabilization, i.e., by making the projective line invertible with respect to the smash product in the category of pointed motivic spaces. We obtain a presentably symmetric monoidal, stable category.\\

Several classical examples of cohomology theories for schemes become representable in the category of motivic spectra: for instance, motivic cohomology with integral coefficients, represented by $\HZ\in\SH(S)$, motivic cohomology with $\Zmod$-coefficients $\HZmod$, ($\Affl$-invariant) algebraic K-theory $\KGL$, ($\Affl$-invariant) Hermitian K-theory $\KO$, algebraic cobordism $\MGL$, and so on. All these examples are not only motivic spectra, but they are also endowed with multiplicative structures, which the corresponding cohomology theories inherit. More precisely, they all admit the structure of commutative algebra objects in $\SH(S)$, i.e., they can be made into motivic $\Einfty$-rings. \\

To investigate motivic phenomena, it is sometimes useful to bring them back into the topological setting. The real Betti realization functor $\rR$ is one of the tools to do so. This functor from the category of motivic spaces over $\R$ (or any subfield $k\lcofib{} \R$) to that of spaces extends the functor of real points (from the category of smooth schemes over $\R$ to that of topological spaces). It stabilizes to a functor between the category of motivic spectra and that of topological spectra. Similarly, one can define a complex Betti realization functor $\rC$ from motivic spaces over $\mathbb{C}$ (or any subfield) to spaces, which has a stable analog from motivic spectra over $\mathbb{C}$ to topological spectra. For example, the complex realization of $\KGL$ is equivalent to the topological spectrum $\KU$ representing complex topological K-theory \cite[Lem.\ 2.2]{PPR}, while its real realization vanishes \cite[Lem.\ 3.9]{BH-eta}. The motivic spectrum $\KO$ complex realizes to $\KOtop$ \cite[Lem.\ 13]{ARO}, representing real topological K-theory, but its real realization is $\KOtop[1/2]$ \cite[Lem.\ 3.9]{BH-eta}. The motivic spectra $\MGL$, $\MSL$, and $\MSp$, representing variants of algebraic cobordism, have complex realizations the spectra $\MU$, $\mathsf{MSU}$, and $\MSp^\mathsf{top}$, respectively, and real realizations $\MO$, $\MSO$, and $\MU$, respectively, all of which represent variants of topological cobordism \cite[Cor.\ 4.7]{BH-eta}. \\

The latter examples of topological spectra also admit the structure of commutative algebra objects, namely they are topological $\Einfty$-rings. Since the Betti realization functor is symmetric monoidal with respect to the smash products on motivic and topological spectra, the realization of any motivic $\Einfty$-ring inherits a structure of topological $\Einfty$-ring. It is therefore natural to wonder whether one can obtain some multiplicative generalizations of the aforementioned results, i.e., identify the induced $\E{n}$-structures on their real realizations, for some $1\leq n \leq \infty$. \\

This paper aims at answering this question for a refinement of the Hermitian $K$-theory spectrum $\KO$, namely its very effective cover $\ko$. Very effectiveness is a certain notion of ``connectivity'' for motivic spectra. The very effective cover functor being lax symmetric monoidal, $\ko$ inherits an $\Einfty$-structure from that of $\KO$, and so does its real realization. Our main result identifies the $\E{1}$-structure on the real realization of $\ko$. Along the way, we will compute the real realization of some other classical examples of motivic spectra, as we discuss below.

We will consider the $\Einfty$-structure on $\KO$ constructed in \cite[\S7, p.\ 129]{HJNY} or \cite[Def.\ 8.1.1]{CHN}, see \Cref{Def:KO} and \Cref{Rmk:Einfty-stc-KO}.

\begin{introthm}[\Cref{Prop:rRko}]\label{Prop:introthmfracsq} There is a Cartesian square of $\E{1}$-rings (and thus also of spectra)
    \[\begin{tikzcd}[row sep = 3em, column sep = 3em]
        {\rR\ko} & {\KOtop_{\geq 0}[1/2]} \\
        {\HZ_{(2)}[t^4]} & {\mathsf{H}\Q[t^4].}
        \arrow["x\mapsto \beta_4/2", from=1-1, to=1-2]
        \arrow["x\mapsto t^4"', from=1-1, to=2-1]
        \arrow["\lrcorner"{anchor=center, pos=0.125}, draw=none, from=1-1, to=2-2]
        \arrow["\mathsf{ch}\,"', "\,\beta_4/2\mapsto t^4", from=1-2, to=2-2]
        \arrow["{t^4\mapsto t^4}"', from=2-1, to=2-2]
    \end{tikzcd}\]
    
    In particular, there is an isomorphism of graded rings $\pi_\ast(\rR\ko) \cong \Z[x]$ where $x$ has degree 4.
    \end{introthm}

    Here, the assignments of elements labeling the arrows describe the action of the maps on the fourth homotopy groups; $\KOtop$ is the topological spectrum representing real topological K-theory; and $\mathsf{H}\mathbb{Q}[t^4]$ is the free $\mathcal{E}_1$-$\mathsf{H}\mathbb{Q}$-algebra on one generator $t^4$ in degree 4. It is an $\HQ$-module, with homotopy ring a polynomial ring over $\Q$ with a single generator in degree 4. The definition of $\HZ_{(2)}[t^4]$ is similar.\\

Hebestreit, Land, and Nikolaus exhibited an identical fracture square for the connective L-theory of the real numbers \cite[p.\ 3]{HLN}. Comparing the two results proves:

\begin{introthm}[\Cref{Prop:compwithLtheory}]\label{Prop:introthmcompLth} There is an equivalence of $\E{1}$-rings
    $$\rR\ko \simeq \mathsf{L}(\R)_{\geq 0}$$
between the real realization of the very effective Hermitian K-theory spectrum and the connective cover of the L-theory spectrum of $\R$.
\end{introthm}

As mentioned above, the proof of \Cref{Prop:introthmfracsq} uses the relation of $\ko$ to other motivic spectra whose real realization is easier to compute (in particular to determine the homotopy ring of $\rR\ko$, and to analyze the maps appearing in the fracture square). Namely, $\ko$ is related to $\KO$, of course, but also to $\kgl$, the very effective cover of algebraic K-theory $\KGL$, via a cofiber sequence $ \Sigma^{1,1}\ko \xrightarrow{\ \eta\ } \ko \longrightarrow \kgl$ (see \cite[Prop.\ 11]{ARO}), where $\eta$ is the motivic Hopf map. We will also use the decomposition of $\ko$ into easier pieces, that is, its very effective slices, which have been computed in \cite{Bachmann-genslices}, and involve in particular the motivic cohomology spectra $\HZ$ and $\HZmod$, and their variant $\HZtilde$. We state the relevant results below. 

\begin{introprop}[\protect{\Cref{Prop:rRKO,Prop:rRko1/2}}]\label{Prop:intromini3} Let $\KW := \KO[\eta^{-1}]$ be the motivic spectrum representing Balmer--Witt K-theory (where $\eta$ is 
    the motivic Hopf map). There is an $\Einfty$-map of spectra
    $$\gamma: \mathsf{L}(\R) \longrightarrow \rR\KW$$
    inducing an equivalence $\mathsf{L}(\R)[1/2] \simeq \rR\KW$. In particular, there are equivalences of $\Einfty$-rings 
    $$\KOtop[1/2] \simeq \rR\KW \simeq \rR\KO.$$
\end{introprop}

The proof of \Cref{Prop:intromini3} relies crucially on the description of the real realization functor as a localization away from $\rho : \calS^0 \to \Gm$ (\cite[Thm.\ 35]{Bachmann-rho}, see \Cref{Prop:rRinvertingrho}), and the relation $\rho\eta \simeq -2$ in $\SH(\R)[\rho^{-1}]$.\\

Although the next results may be familiar to specialists (particularly in the case of $\HZ$), we are unaware of any explicit reference in the literature.

\begin{introprop}[\Cref{Prop:rRHZHZmod}]\label{Prop:intromini1} Let $\HZ$ and $\HZmod$ be the motivic spectra representing motivic cohomology with $\Z$-coefficients and $\Zmod$-coefficients, respectively. Then, there is an equivalence of $\E{1}$-maps
    $$\rR(\HZ \to \HZmod) \simeq (\HZmod[t^2] \to \HZmod[t])$$ 
    where $\HZ \to \HZmod$ is the projection, and $\HZmod[t^2] \to \HZmod[t]$ is the map of free $\E{1}$-$\HZmod$-algebras (see \Cref{Def:freeE1HA}) inducing the inclusion $\Zmod[t^2]\subseteq \Zmod[t]$ in homotopy. 
\end{introprop}

\begin{introprop}[\Cref{Prop:rRkgl}]\label{Prop:intromini2} Let $\KGL$ be the motivic spectrum representing algebraic K-theory, and $\kgl$ be its very effective cover. Then, there is an equivalence of $\E{1}$-maps
    $$\rR(\kgl \to \HZ) \simeq (\HZmod[t^4] \to \HZmod[t^2])$$
where $\kgl \to \kgl/\beta_\KGL \simeq \HZ$ is the projection (with $\beta_\KGL$ the periodicity generator for $\KGL$), and $\HZmod[t^4] \to \HZmod[t^2]$ is the map of free $\E{1}$-$\HZmod$-algebras (see \Cref{Def:freeE1HA}) inducing the inclusion $\Zmod[t^4]\subseteq \Zmod[t^2]$ in homotopy. 
\end{introprop}

\begin{introprop}[\Cref{Prop:rRHZtilde}]\label{Prop:intromini4} On the homotopy rings, the real realization of the quotient map $$\HZtilde := f_0\underline{K}^{MW}_\ast \longrightarrow f_0(\underline{K}^{MW}_\ast/\eta) = f_0\underline{K}^{M}_\ast = \HZ$$ identifies with the quotient map $\Z[t^2]/(2t^2) \to \Z[t^2]/2$.
\end{introprop}

\textbf{Strategy for the proof of \cref{Prop:introthmfracsq}. } We first compute the homotopy ring of $\rR\ko$, by relating $\ko$ to other motivic spectra whose real realizations can be computed more easily, namely the ones in Propositions \ref{Prop:intromini1} to \ref{Prop:intromini4}. This involves spectral sequence computations, for the filtration on $\rR\ko$ induced by its very effective slices. This homotopy ring is a polynomial ring over $\Z$ with a single generator $x$ in degree 4, $\Z[x]$.

We will then identify the underlying $\E{1}$-ring, by describing it as an explicit pullback square. We will use a 2-local fracture square. Indeed, the situation is simpler when $2$ is inverted, and we will then rely on \Cref{Prop:intromini3}. The localization of $\rR\ko$ at $(2)$ will be identified with a free $\E{1}$-algebra in the category of $\HZ_{(2)}$-modules. In particular, this will require us to produce an $\E{1}$-map of $\HZ_{(2)}$-modules $\HZ_{(2)} \to (\rR\ko)_{(2)}$, and we will also use our computation of the homotopy ring of $\rR\ko$. To produce such a map, we will use the Thom spectrum $\MSO$ as an intermediate object. Indeed, the existence of $\E{2}$-maps $\HZ_{(2)} \to \MSO_{(2)}$ in $\Sp$ and $\MSL \to \ko$ in $\SH(\R)$ has been shown in the past (\cite[Cor.\ 3.7]{HLN} and \cite[Thm.\ 10.1]{HJNY} respectively). Moreover, there is an equivalence of $\Einfty$-rings $\rR(\MSL) \simeq \MSO$, relying on a general comparison of the topological and motivic Thom spectrum functors, which we conducted in \cite{rRThom-arXiv}. The latter article originated from here and can be viewed as a second half to the present article. 

Combining these results with a careful analysis of the equivalences we have chosen and the maps in the fracture square yields the pullback square of \cref{Prop:introthmfracsq}.\\

\textbf{Organization of the article.} In \Cref{Sect:preliminaries}, we make some recollections about the notions of (very) effectiveness for motivic spectra, the real realization functor, and free $\E{1}$-$\HA$-algebras in topological spectra. It serves as a toolbox whose results will be used throughout the paper. Readers already familiar with the material may safely skip this section and return to it as needed. In \Cref{Sect:examplesrealization}, we compute the real realizations of several classical motivic spectra, in particular we prove Propositions \ref{Prop:intromini3} to \ref{Prop:intromini4}. These computations will come in handy in \Cref{Sect:rRko}, where we compute $\rR\ko$ and prove \Cref{Prop:introthmfracsq,Prop:introthmcompLth}. At the very end is an \nameref{Sect:indexofnotation}, to which the reader may refer for any unexplained notation.\\

\textbf{Conventions and prerequisites}\hfill
\begin{itemize}
    \item For introductory material regarding the topics discussed in this article, the reader may refer to \cite{Mastersthesis}.
    \item We work in the $\infty$-categorical setting, following \cite{Lurie-HTT,Lurie-HA}. We simply say ``category'' for ``$\infty$-category''. Let $\Cat$ be the $\infty$-category of small $\infty$-categories.
    \item We denote by $\Spc$ the category of spaces and by $\Sp$ the category of spectra.
    \item Unless mentioned otherwise, $k$ denotes a fixed \emph{perfect} base field, and $S$ a fixed qcqs base scheme. When working with Hermitian K-theory, we assume $\mathsf{char}(k)\neq 2$. Since we will eventually only work with subfields of $\R$ or $\C$, these assumptions will be automatically satisfied. Let $\Sm_S$ be the category of smooth schemes of finite type over $S$. We denote by $\Spc(S)$ and $\SH(S)$ the categories of motivic spaces and motivic spectra, respectively, over $S$ (constructed in the setting of model categories in \cite{MV}, and with $\infty$-categories in \cite{Robalo} for example). For a smooth scheme $X$ over $S$ or a space $Y$, we will often also denote by $X$ and $Y$ the corresponding motivic spaces or motivic spectra. Let $T$ be the infinite suspension spectrum of the pointed motivic space $(\mathbb{P}^1,\infty)$.
   \item We denote by $\rR : \SH(\R) \to \Sp$ the real realization functor (see \cite[\S 10]{Bachmann-rho} for example), and also the corresponding functors $\Pre(\Sm_\R) \to \Spc$ and $\Spc(\R) \to \Spc$, or their pointed analogs, by $\rR$.
   \item We denote by $\HZ$, $\HZmod$ both the motivic and topological Eilenberg--MacLane spectra (representing motivic and singular cohomology, respectively); the context should suffice to distinguish them.
    \item For $p$ a prime, \emph{localization at $(p)$} means inverting every prime different from $p$, whereas \emph{localization away from $p$} means inverting $p$.
\end{itemize}

\vspace{0.5cm}
\section{A few preliminaries}\label{Sect:preliminaries}

This introductory section contains some preliminaries about the notions of effectiveness and very effectiveness for motivic spectra (\Cref{Subsect:eff}), the real realization functor (\Cref{Subsect:rR}), and free $\E{1}$-$\HA$-algebras (in topological spectra), for $\mathsf{A}$ a discrete commutative ring and $\HA$ its Eilenberg-Mac Lane spectrum (\Cref{Subsect:freeE1HA}). The latter will appear in our computations in \Cref{Sect:examplesrealization,Sect:rRko}.

\vspace{0.2cm}
\subsection{Recollections about (very) effectiveness of motivic spectra}\label{Subsect:eff}\hfill\vspace{0.2cm}

In this subsection, we recall the definitions of the notions of effectiveness \cite[\S 2]{Voevodsky-openproblems} and very effectiveness \cite[Def.\ 5.5]{SO}, essentially to fix notation. The upshot is that these are two different notions of connectivity for motivic spectra. For some justification of these concepts, we refer to the introduction of \cite{Bachmann-genslices}.

\begin{Def}[\protect{\cite[\S 13]{BH-norms}}]\label{Def:SHkeff} The \emph{category of effective motivic spectra} $\SHk^\eff$ is the subcategory of $\SHk$ generated under colimits by all $\calS^n \Smash \Sigma^\infty_+ X$ for $X\in\Sm_k$ and $n\in\Z$. The \emph{category of very effective motivic spectra} $\SHk^\veff$ is the subcategory of $\SHk$ generated under colimits by all $\calS^n \Smash \Sigma^\infty_+ X$ for $X\in\Sm_k$ and $n\geq 0$. For $n\geq 0$, let $\SHk^\eff(n) := T^{\Smash n} \Smash \SHk^\eff$ be the category of \emph{$n$-effective motivic spectra} and $\SHk^\veff(n) := T^{\Smash n} \Smash \SHk^\veff$ be the category of \emph{$n$-very effective motivic spectra}.
\end{Def}

\begin{Def}[\protect{\cite[Thm. 5.2.3]{Morel-A1intro}}, \protect{\cite[App.\ B]{BH-norms}}]\label{Def:homotopytstc} The \emph{homotopy t-structure} on $\SHk$ is defined by
    \begin{align*}
        \SHk_{\geq 0} &= \{E\in\SHk \mid \homsh{i}{\ast}{E} = 0 \ \forall i<0\},\\
        \SHk_{\leq 0} &= \{E\in\SHk \mid \homsh{i}{\ast}{E} = 0 \ \forall i>0\}.
    \end{align*}
    Equivalently, $\SHk_{\geq 0}$ is generated under colimits and extensions by $\{\Sigma^{p,q} \Sigma^\infty_+ X \mid X\in\Sm_k, p-q \geq 0\}$.
    \end{Def}

    \begin{Rmk}\label{Rmk:smstconheart} The smash product symmetric monoidal structure on $\SHk$ is compatible with the homotopy t-structure in the sense of \cite[Def.\ A.10]{AN}. To see this, we use the equivalent characterization of the connective part in \Cref{Def:homotopytstc}. Then the unit is one of the generators of the connective part. Moreover, the tensor product functor preserves colimits in each variable; since $\SHk$ is stable, this functor is in particular exact in both variables. Therefore, to check that $\SHk_{\geq 0}$ is stable under tensor products we only have to show that the aforementioned collection of generators also is. For any $X,Y\in\Sm_k$ and $p,q,p',q'\in\Z$ with $p-q\geq 0$ and $p'-q'\geq 0$, we have $\Sigma^{p,q} \Sigma^\infty_+ X \Smash \Sigma^{p',q'} \Sigma^\infty_+ Y \simeq \Sigma^{p+p',q+q'}\Sigma^\infty_+ (X\times Y)$ with $p+p' - (q+q') \geq 0$, which is again in the connective part. In particular, $\SHk_{\geq 0}$ is a symmetric monoidal subcategory of $\SHk$, and there is a unique symmetric monoidal structure on the heart $\SHk^\heartsuit$ such that the truncation $\SHk_{\geq 0} \to \SHk^\heartsuit$ is symmetric monoidal \cite[Lem.\ A.12]{AN}.
    \end{Rmk}

\begin{Def}[e.g. \protect{\cite[before Prop.\ 4]{Bachmann-genslices}}]\label{Def:effectivehomotopytstc} The \emph{effective homotopy t-structure} on $\SHk^\eff$ is defined by
\begin{align*}
    \SHk^\eff_{\geq 0} &= \{E\in\SHk^\eff \mid \homsh{i}{0}{E} = 0 \ \forall i<0\},\\
    \SHk^\eff_{\leq 0} &= \{E\in\SHk^\eff \mid \homsh{i}{0}{E} = 0 \ \forall i>0\}.
\end{align*}
For $n\in\Z$, let $\tau_{\geq_e n}$ and $\tau_{\leq_e n}$ be the truncation functors with respect to this t-structure.
\end{Def}

\begin{Prop}\label{Prop:propertiesofSHkeff} \phantom{j}
\begin{enumerate}[label = (\roman*)]
    \item We have $\SHk^\eff_{\geq 0} = \SHk^\veff$.
    \item The smash product symmetric monoidal structure on $\SHk$ restricts to symmetric monoidal structures on both $\SHk^\eff$ and $\SHk^\veff$, such that the inclusion functors are symmetric monoidal.
\end{enumerate}  
\end{Prop}
\begin{proof}
Part $(i)$ is \cite[Remark after Prop.\ 4]{Bachmann-genslices}. For part $(ii)$, note that since these two subcategories are full and contain the unit, by \cite[Prop.\ 2.2.1.1 and Rmk.\ 2.2.1.2]{Lurie-HA} we only have to show that they are stable under the smash product. The latter preserves colimits in both variables, so we only have to check the statement for generators. For all $m,n\in\Z$ and $X,Y\in\Sm_k$, we have
$$(\calS^n \Smash \Sigma^\infty_+ X) \Smash (\calS^m \Smash \Sigma^\infty_+ Y) \simeq \calS^{m+n} \Smash \Sigma^{\infty}(X_+ \Smash Y_+) \simeq \calS^{m+n} \Smash \Sigma^{\infty}_+(X\times Y),$$
which is by definition effective, and even very effective when $n,m\in\N$.
\end{proof}

\begin{Lemma}\label{Prop:rn}
    The inclusion functors $\iota_n : \SHk^\eff(n) \xhookrightarrow{} \SHk$ and  $\tilde{\iota}_n : \SHk^\veff(n) \xhookrightarrow{} \SHk$ admit right adjoints $r_n$ and $\tilde{r}_n$ respectively, for all $n\geq 0$. 
\end{Lemma}
\begin{proof} This follows from the adjoint functor theorem \cite[Cor.\ 5.5.2.9]{Lurie-HTT} and presentability of the categories involved \cite[Prop.\ C.12.(2)]{Hoyois-traceformula}.
\end{proof}

\begin{Def}\label{Def:effcover} We call the compositions $\iota_n\circ r_n =: f_n$ and $\tilde{\iota}_n\circ \tilde{r}_n =: \tilde{f}_n$ the \emph{$n$-effective cover} and \emph{very $n$-effective cover} functors respectively. For all $E\in \SHk$, we obtain by adjunction natural maps $f_{n+1}E \to f_nE$ and $\tilde{f}_{n+1}E \to \tilde{f}_nE$ for all $n\geq 0$. Their cofibers $s_nE := \mathsf{cof}(f_{n+1}E \to f_nE)$ and $\tilde{s}_nE := \tilde{f}_{n+1} \to \tilde{f}_nE$ are called the \emph{$n$-effective slice} and \emph{very $n$-effective slice} of $E$, respectively. By \emph{(very) effective cover}, we mean the (very) $0$-effective cover.
\end{Def}

\begin{Prop}[\protect{\cite{Pelaez}}]\label{Prop:smheart} The inclusion functor $\iota_0: \SHk^\eff \to \SHk$ is right t-exact. The functor $r_0: \SHk \to \SHk^\eff$ is t-exact, lax symmetric monoidal, and restricts to a lax symmetric monoidal functor $r_0^\heartsuit: \SHk^\heartsuit \to \SHk^{\eff,\heartsuit}$.
\end{Prop}
\begin{proof} The claims about t-exactness are proven in \cite[Prop.\ 4.(3)]{Bachmann-genslices}. As a right adjoint to the inclusion, which is symmetric monoidal (\Cref{Prop:propertiesofSHkeff}$(ii)$), $r_0$ is also lax symmetric monoidal by the doctrinal adjunction principle (follows from of \cite[Cor.\ 7.3.2.7]{Lurie-HA}). As a t-exact functor, $r_0$ carries the non-negative part to the non-negative part of the t-structures, and the heart to the heart. Its (co)restriction to the symmetric monoidal subcategories $\SH(k)_{\geq 0}$ and $\SHk^\eff_{\geq 0}$ is then also lax symmetric monoidal. Its (co)restriction to the hearts $r_0^\heartsuit$ can be written as the composition 
$$\SHk^\heartsuit \xhookrightarrow{} \SHk_{\geq 0} \xrightarrow{r_0} \SHk^\eff_{\geq 0} \xrightarrow{\pi_0} \SHk^{\eff,\heartsuit}.$$
Here, $\pi_0$ is symmetric monoidal by definition of the tensor product on the heart \cite[Lem.\ A.12]{AN}. The first inclusion is lax symmetric monoidal as a right adjoint to the truncation, which is symmetric monoidal by \Cref{Rmk:smstconheart}. Thus, $r_0^\heartsuit$ is lax symmetric monoidal.
\end{proof}
\vspace{0.2cm}

\subsection{Recollections about the real realization functor}\label{Subsect:rR}\hfill\vspace{0.2cm}

We now state several results about the real realization functor, recalling its construction in the process (one original reference in the $C_2$-equivariant case is \cite{DI}; we give $\infty$-categorical references below). In this subsection, we assume that the field $k$ comes with a fixed embedding $k\xhookrightarrow{} \R$. Considering the case $k=\R$ is not really a loss of generality; the real realization functor over the base $k$ is equivalent to the precomposition of the real realization functor over the base $\R$ by the pullback along $\Spec(\R) \to \Spec(k)$. 

\begin{Prop}\label{Prop:rR} The real realization functors $\rR: \Spc(k)_\ast \to \Spc_\ast$ and $\rR:\SH(k) \to \Sp$ are colimit-preserving, symmetric monoidal functors with respect to the smash products.
\end{Prop}
\begin{proof} This follows from the construction of $\rR$, as we now recall (see e.g.\ \cite[\S 10]{Bachmann-rho} or \cite[\S 7.2]{ret2}). The functor of real points $\Sm_k \to \Spc$ preserves finite products and thus left Kan extends to colimit-preserving, symmetric monoidal functors $\Pre(\Sm_k) \to \Spc$ for the Cartesian structures, and $\Pre(\Sm_k)_\ast \to \Spc_\ast$ for the smash products. Passing to the localization, we obtain a colimit preserving, symmetric monoidal functor $\rR: \Spc(k)_\ast \to \Spc_\ast$. The postcomposition $\Spc(k)_\ast \to \Spc_\ast \to \Sp$ by the infinite suspension functor inverts $\mathbb{P}^1$, and thus lifts to a colimit preserving, symmetric monoidal functor $\rR: \SH(k) \to \Sp$.
\end{proof}

Here is one example of why very effectiveness is a reasonable notion of connectivity in $\SH(k)$.

\begin{Lemma}\label{Prop:rRconnective}
    For all $m\geq 0$, the restriction of $\rR$ to $\SHk^\veff(m)$ takes values in $m$-connective spectra. 
\end{Lemma}
\begin{proof}
    Since $\SHk^\veff(m) = T^{\Smash m} \Smash \SHk^\veff$ by definition and $\rR(T^{\Smash m }) = \Sph^m$, it suffices to prove the case $m=0$. By definition, $\SHk^\veff$ is generated under colimits by all objects of the form $\calS^n \Smash \Sigma^\infty_+ X$ for $n\geq 0$ and $X\in \Sm_k$. These all realize to connective spectra because $\rR(\calS^n \Smash \Sigma^\infty_+ X) = \Sph^n \Smash \Sigma^\infty_+ \rR(X)$ is connective for all $n\geq 0$ and $X\in\Sm_k$. This concludes because $\rR$ commutes with colimits, and a colimit of connective spectra is connective (this holds for the non-negative part of any t-structure).  
\end{proof}

The category of spectra $\Sp$ is obtained as a localization of a category of \emph{prespectra}: if one views a spectrum as a sequence of pointed spaces $(c_n)_{n\in\N}$ with equivalences $c_n \to \Omega c_{n+1}$, then prespectra are defined as similar sequences for which the structure maps are not required to be equivalences. More precisely, the category of prespectra is that of reduced functors $\Spc_\ast^\mathsf{fin} \to \Spc$, and $\Sp$ is the full subcategory of reduced excisive functors \cite[Prop.\ 1.4.2.13.]{Lurie-HA}. The latter is in particular a localization of the category of prespectra \cite[Rmk.\ 1.4.2.4.]{Lurie-HA}, and the localization functor is called \emph{spectrification}.

\begin{Lemma}\label{Prop:levelwiserR} Let $E = (E_0,E_1, \dots) \in \SHk$. Then, the real realization of $E$ is computed as the spectrification of the prespectrum $(\rR(E_n))_{n\in\N}$. More explicitly, for all $n\in\Z$, $\rR(E)_n \simeq \colim_m \Omega^m \rR(E_{n+m})$.
    
    In particular, $\pi_n(\rR(E)) \cong \colim_m \pi_{n+m}(\rR(E_{m}))$.
    \end{Lemma}
    \begin{proof}
    First note that $(\rR(E_n))_{n\in\N}$ is a prespectrum via the structure maps obtained by real realization of those of $E$. Indeed, a map $T \Smash E_n \to E_{n+1}$ realizes to a map $\Sph^1 \Smash \rR(E_n) \to \rR(E_{n+1})$. The spectrification functor for topological prespectra sends a sequence $(E_n)_{n\in\N}$ to $\colim_n \Sigma^{-n}\Sigma^\infty E_n$ where the maps in the colimit are induced by the structure maps $\Sigma E_n \to E_{n+1}$ \cite[Lem.\ 2.1.3]{Schwede-spectra}. Similarly, for motivic prespectra, spectrification sends a sequence $(E_n)_{n\in\N}$ to $\colim_n \Sigma_T^{-n}\Sigma_T^\infty E_n$ with maps induced by $\Sigma_T E_n \to E_{n+1}$. If $E$ as in the statement is already a spectrum, it is its own spectrification, and therefore we have
    $$\rR(E) \simeq \rR(\colim_n (T^{\Smash (-n)}\Smash \Sigma^\infty_T E_n)) \simeq \colim_n (\Sph^{\Smash (-n)} \Smash \Sigma^\infty_{\Sph^1} \rR(E_n)).$$
    The right-hand side is exactly the spectrification of the sequence $(\rR(E_n))_{n\in\N}$ with structure maps given by the real realizations of those of $E$. For the last claim, we have
    \begin{align*}
        \pi_n(\rR(E)) &\cong \colim_m \pi_{n+m} (\rR(E)_m) \cong \colim_m \pi_{n+m} (\colim_k \Omega^k \rR(E_{m+k})) \\
        &\cong \colim_m\colim_k \pi_{n+m+k} (\rR(E_{m+k})) \cong \colim_\ell \pi_{\ell+n} (\rR(E_{\ell})).
    \end{align*}
    \end{proof}
    
\begin{thm}[\protect{\cite[Thm.\ 35]{Bachmann-rho}}]\label{Prop:rRinvertingrho} Let $S$ be a qcqs scheme, and $\rho: \calS^0 \to \Gm$ in $\SH(S)$ be the inclusion on the points $1$ and $-1$. There is an equivalence of categories
    $$\mathsf{SH}(S)[\rho^{-1}] \lsimeq{} \mathsf{SH}(\mathsf{Shv}(RS))$$
    between the category of $\rho$-local motivic spectra over $S$, and the category of spectra associated with the topos of sheaves of spaces on $RS$, the real spectrum of $S$ \cite[\S 0.4.2]{Scheiderer-ret}. In particular, there is an equivalence $\SHR[\rho^{-1}] \lsimeq{} \Sp$ such that the composite 
    $\SHR \longrightarrow  \SHR[\rho^{-1}] \lsimeq{} \Sp$
    is naturally equivalent to the real realization functor $\rR$. 
    \end{thm}
    
 \begin{Lemma}\label{Prop:homotopyofrR}
        For any $E\in \SHR$, the homotopy groups of its real realization can be computed as
        $$\forall m\in\Z, \ \ \pi_m(\rR(E)) = \colim_n (\cdots \longrightarrow \underline{\pi}_m(E)_n(\Spec(\R)) \xrightarrow{\,\ \rho\,\ } \underline{\pi}_m(E)_{n+1}(\Spec(\R)) \longrightarrow \cdots).$$
    \end{Lemma}
    \begin{proof}
        Call $F$ the equivalence $\SHR[\rho^{-1}] \to \Sp$ above, and $L_\rho: \SHR \to \SHR[\rho^{-1}]$ the canonical localization functor. Then, for all $m \in\Z$ and $E\in \SHR$, we have
        \begin{align*}
        \pi_m(\rR E) &\cong [\Sph^m,\rR E] \cong [\rR(\calS^m),\rR E] \cong [F(L_\rho(\calS^m)),F(L_\rho(E))] \cong [L_\rho(\calS^m),L_\rho(E)] \\
            & \cong [\calS^m,\colim_n\, E\Smash \Gm^{\Smash n}] \tag{\text{since $L_\rho$ is a left-adjoint}}\\
            & \cong \colim_n\, [\calS^m, E\Smash \Gm^{\Smash n}] \tag{by compactness of $\calS^m$}\\
            & \cong \colim_n\, \homsh{m}{n}{E}(\Spec(\R))
        \end{align*}
        where the last isomorphism follows from the fact that the sections of the homotopy sheaves and presheaves of $E$ agree on $\Spec(\R)$. This holds because sheafification preserves stalks, and the functor of sections over $\Spec(\R)$ is a stalk functor of the Nisnevich site on $\Sm_\R$ (by the description of the points of the Nisnevich site in \cite[p99]{MV}).
    \end{proof}

\begin{Lemma}\label{Prop:rRcommuteswithloc} Let $p$ be a prime. The real realization functor commutes with localization away from $p$, localization at $(p)$, and rationalization.
\end{Lemma}
\begin{proof} These operations are described by colimits of maps given by multiplication by an integer; both colimits and such maps are preserved by $\rR$. Indeed, multiplication by an integer on some object $X$ of a stable category is defined using the diagonal and codiagonal maps $X \to X^{\oplus p} \to X$. Since real realization is an exact functor, it preserves the latter.
\end{proof}

\vspace{0.2cm}

\subsection{Free \texorpdfstring{$\E{1}$-$\HA$-}{E1-HA-}algebras}\label{Subsect:freeE1HA}\hfill\vspace{0.2cm}

In \Cref{Sect:examplesrealization,Sect:rRko}, we will encounter $\Einfty$-rings whose homotopy rings are polynomial in one variable. We will identify the underlying $\E{1}$-structures as free $\E{1}$-algebras in the category of $\HA$-modules, where $\HA$ is the Eilenberg-Mac Lane spectrum associated with a discrete commutative ring $\mathsf{A}$.

\begin{Def} Let $R \in \CAlg(\Sp)$. The \emph{category of $\E{k}$-$R$-algebras in $\Sp$} is $\Alg_{\E{k}}(\Mod_R(\Sp))$, the category of $\E{k}$-algebras in $R$-modules.
\end{Def}

\begin{Prop}\label{Prop:freeE1Ralgebra} The forgetful functor $U: \Alg_{\E{1}}(\Mod_R(\Sp)) \to \Sp$ admits a left adjoint $F_{\E{1},R}$, the \emph{free $\E{1}$-$R$-algebra functor}, given on the underlying spectra by $E\mapsto \bigvee_{n\geq 0} E^{\Smash n} \Smash R$. 
\end{Prop}
    \begin{proof} The functor $U$ is the composition of two forgetful functors $\mathsf{Alg}_{\E{1}}(\mathsf{Mod}_R(\Sp)) \to \mathsf{Mod}_R(\Sp)$ and $\mathsf{Mod}_R(\Sp) \to \Sp$. Both admit left adjoints by \cite[Cor.\ 3.1.3.4]{Lurie-HA} and \cite[Prop.\ 4.2.4.2]{Lurie-HA}, respectively; explicit expressions are given in \cite[Cor.\ 4.1.1.18]{Lurie-HA} and \cite[Cor.\ 4.2.4.8]{Lurie-HA}, respectively.
    \end{proof}
    
    \begin{Def}\label{Def:freeE1HA} Let $n\geq 0$, and let $\mathsf{A}$ be a discrete commutative ring. Then $\HA \in \CAlg(\Sp)$ and thus we may define $\HA[t^n] := F_{\E{1},\HA}(\Sigma^n\Sph)$ the \emph{free $\E{1}$-$\HA$-algebra over one generator in degree $n$}.
    \end{Def}
    
    \begin{Prop}\label{Prop:freeHA} Let $n\geq 0$, and let $\mathsf{A}$ be a discrete commutative ring. The free $\E{1}$-$\HA$-algebra $\HA[t^n]$ has the following properties:
    \begin{enumerate}[label = (\roman*)]
        \item The underlying spectrum is given by $\bigvee_{j\geq 0} \Sigma^{nj}\HA$.
        \item The homotopy ring is a polynomial ring on one generator in degree $n$, i.e., $\pi_\ast(\HA[t^n]) \cong A[t^n]$.
        \item Let $E$ be an $\E{2}$-ring spectrum with an $\E{2}$-map $\HA \to E$. Then $E$ can be viewed as an {$\E{1}$-$\HA$-algebra} in a canonical way, and any element $\alpha\in \pi_n(E)$ determines a map $\HA[t^n] \to E$ in $\Alg_{\E{1}}(\Mod_R(\Sp))$, sending $t^n$ to $\alpha$ in homotopy. 
    \end{enumerate}
    \end{Prop}
    \begin{proof} Item $(i)$ is given by \Cref{Prop:freeE1Ralgebra}. In particular, the computation of the homotopy groups follows. To finish the proof of item $(ii)$, we must consider the ring structure. We have to show that multiplication by a generator of the $n$-th homotopy group induces isomorphisms in all degrees, i.e., that the composition
    $$\Sph^n \Smash \HA[t^n] \xrightarrow{t^n \Smash \id{}} \HA[t^n] \Smash \HA[t^n] \longrightarrow \HA[t^n]$$ 
    induces isomorphisms in homotopy in positive degrees. This map is given by the identification 
    $$\Sph^n \Smash \bigvee_{j\geq 0} \Sigma^{nj}\HA \simeq \bigvee_{j\geq 0} (\Sigma^n\Sph \Smash \Sigma^{nj}\HA) \longrightarrow \bigvee_{j\geq 0} \Sigma^{nj}\HA$$
    mapping $\Sigma^n\Sph \Smash \Sigma^{nj}\HA$ to $\Sigma^{n(j+1)}\HA$ identically, so our claim follows. Item $(iii)$, using that there is a forgetful functor $\Alg_{\E{2}}(\Sp)_{\HA/} \to \Alg_{\E{1}}(\Mod_{\HA}(\Sp))$ \cite[Rmk.\ 3.7]{ACB}, is simply rephrasing the fact that $F_{\E{1},\HA}$ is left adjoint to the forgetful functor. 
\end{proof}

\begin{Rmk} \Cref{Prop:freeHA}$(iii)$ may seem a bit artificial. The point is that $\HA[t^n]$ is \emph{not} a free object over $\Sph^n \in \Sp$ in the category $\Alg_{\E{1}}(\Sp)_{\HA/}$; the latter category is a priori \emph{not} equivalent to $\Alg_{\E{1}}(\Mod_{\HA}(\Sp))$ (see \cite[Warning 7.1.3.9]{Lurie-HA} and \cite[Remark 3.7]{ACB}). There is indeed a forgetful functor $\Alg_{\E{1}}(\Sp)_{\HA/} \to \Mod_{\HA}(\Sp)$, but the $\E{1}$-structure of an object on the left-hand side is not compatible with the $\HA$-module structure obtained on the right-hand side. A higher commutativity (for instance $E$ belonging to $\Alg_{\E{2}}(\Sp)_{\HA/}$ as in item $(iii)$) ensures this compatibility. 
\end{Rmk}

\vspace{0.5em}
\section{The real realizations of some classical motivic spectra}\label{Sect:examplesrealization}

In this section, we compute the real realizations of several classical motivic spectra, either at the level of spectra, $\E{1}$- or $\Einfty$-rings : algebraic K-theory $\KGL$ and its very effective cover $\kgl$, motivic cohomology $\HZ$ and $\HZmod$, their variant $\HZtilde$, and Hermitian K-theory $\KO$. All of them will appear in our computations with $\ko$ in \Cref{Subsect:rRKO}. The complex realizations of these motivic spectra behave ``as expected'': for example, $\rC(\KGL) \simeq \KU$ and $\rC(\kgl) \simeq \KU_{\geq 0}$ \cite[Lem.\ 3]{ARO}, $\rC(\HZ) \simeq \HZ$ \cite[in particular Thm.\ 5.5]{Levine-comparison}, and $\rC(\KO) \simeq \KOtop$ \cite[Lem.\ 13]{ARO}. As we will see, real realization holds more surprises for us. In the whole section, let $k \xhookrightarrow{} \R$ be a fixed subfield (in order to define a real realization functor on $\SH(k)$).

\vspace{0.2cm}
\subsection{The real realizations of \texorpdfstring{$\HZmod$}{HZ/2} and \texorpdfstring{$\HZ$}{HZ}}\label{Subsect:rRHZHZmod}\hfill\vspace{0.2cm}

The motivic spectra $\HZ$ and $\HZmod$ appear as very effective slices of $\ko$, whose real realizations we will need in \Cref{Sect:rRko}.

\begin{Def}\label{Def:HZ} The \emph{motivic Eilenberg-Mac Lane spectrum} $\HZ\in\SH(k)$, is equivalently defined as:
\begin{enumerate}[label = (\roman*)]
    \item The motivic spectrum representing the motivic cohomology theory in $\SHk$, i.e., for all $X\in\Sm_k$ and $p,q\in\Z$, we have $H^{p,q}(X,\Z) \cong [\Sigma^{\infty}_+X,\Sigma^{p,q}\HZ]$.
    \item The zeroth effective slice of the sphere spectrum in $\SHk$.
    \item The effective cover of \emph{Milnor K-theory}: $\HZ = f_0\underline{K}^M_\ast$ \cite[\S 6.3]{Morel-A1intro}.
\end{enumerate}
Definitions $(i)$ and $(ii)$ are equivalent by \cite{Levine-coniveau}, and definition $(iii)$ is equivalent to $(i)$ and $(ii)$ by \cite[Lem.\ 12]{Bachmann-genslices}. The \emph{motivic Eilenberg-Mac Lane spectrum} $\HZmod$ is defined as the cofiber of the multiplication by 2 map on $\HZ$ (it represents mod 2 motivic cohomology and is equivalent to $f_0(\underline{K}^M_\ast/2)$).
\end{Def}

\begin{Rmk}\label{Rmk:HZHZmodEinfty} Definition $(iii)$ allows us to define an $\Einfty$-structure on $\HZ$ and $\HZmod$. Indeed, since the functor $f_0 : \SH(k)^\heartsuit \to \SH(k)^\eff$ is lax symmetric monoidal by \Cref{Prop:smheart}, we only have to show $\underline{K}^M_\ast, \underline{K}^M_\ast/2 \in \mathsf{CAlg}(\SHk^\heartsuit)$. By \cite[Rmk.\ 1.2.1.12]{Lurie-HA}, the heart of a $t$-structure is a 1-category, and therefore its commutative algebras are the commutative algebras in the 1-categorical sense. By \cite[after Thm.\ 6.4.7]{Morel-A1intro}, we have $\underline{K}^{MW}_\ast \simeq \homsh{0}{\ast}{\calS}\in\SHk^\heartsuit$. Since the truncation functor $\SHk_{\geq 0} \to \SHk^\heartsuit$ is symmetric monoidal \cite[Lem.\ A.12]{AN}, the image $\homsh{0}{\ast}{\calS}$ of $\calS \in \CAlg(\SHk_{\geq 0})$ inherits the structure of a commutative algebra object. Because $\underline{K}^{M}_\ast$ and $\underline{K}^{M}_\ast/2$ are quotients of $\underline{K}^{MW}_\ast$, they also define commutative algebras.

	Definition $(ii)$ also provides an $\Einfty$-structure on $\HZ$, since the zeroth effective slice functor is symmetric monoidal \cite[Thm.\ 3.6.13]{Pelaez}. It is proved in \cite[Thm.\ B.4]{BH-norms} that this $\Einfty$-structure agrees with the one constructed by Spitzweck in \cite{Spitzweck-HZ}. The proof of the theorem also shows the uniqueness of such an $\Einfty$-structure when the unit is prescribed. 
\end{Rmk}

\begin{Prop}\label{Prop:HZHZmodveryeff}
    The motivic spectra $\HZ$, $\HZmod$, and $\HZtilde := f_0\underline{K}_\ast^{MW}$ are very effective. 
\end{Prop}
\begin{proof} Since the functor $r_0 : \SHk \to \SHk^\eff$ is t-exact (\Cref{Prop:smheart}), and $\SHk^\eff_{\geq 0} = \SHk^\veff$ by \Cref{Prop:propertiesofSHkeff}$(i)$, it suffices to show that $\underline{K}_\ast^{MW}$, $\underline{K}^{M}_\ast$, and $\underline{K}^{M}_\ast/2$ belong to the connective part in the homotopy t-structure on $\SHk$. Actually, they even belong to the heart: by \cite[Lem.\ 6 and below]{Bachmann-genslices}, they are effective homotopy modules (i.e., they belong to the image of $\iota_0^\heartsuit: \SH(k)^{\eff,\heartsuit} \to \SH(k)^\heartsuit$).
\end{proof}

When trying to compute real realizations, the case of $\HZmod$ is easier to deal with than that of $\HZ$, because the global sections of the homotopy sheaves of $\HZmod$ are known.

\begin{Def} Let $\tau \in H^{0,1}(k,\Zmod) \cong \mu_2(k) \cong \Zmod$ be a generator ($\mu_2$ denotes the square roots of unity), i.e., corresponding to $-1\in\mu_2(k)$.
\end{Def}

\begin{thm}[\cite{Voevodsky-Milnor,Voevodsky-Milnor-published}]\label{Prop:Milnorsconjecture} There are isomorphisms of graded rings $$\underline{K}^M_\ast(k)/2 \cong H^\ast_\text{ét} (k, \Zmod) \cong \homsh{0}{\ast}{\HZmod}(k).$$ 
Moreover, $\underline{\pi}_{\ast,\ast}(\HZmod)(\R) \simeq \Zmod[\tau,\rho]$, where $\tau$ has bidegree $(0,-1)$ and $\rho$ has bidegree $(-1,-1)$.
\end{thm}

\begin{proof} Milnor's conjecture, proven by Voevodsky \cite{Voevodsky-Milnor,Voevodsky-Milnor-published}, states in particular that for a field $k'$ of characteristic not 2, we have $\underline{K}^M_\ast(k')/2 \cong H^\ast_\text{ét} (k', \Zmod)$. In words, the Milnor K-theory $\mathsf{mod}\, 2$ of $k'$ is equal to the étale cohomology $\mathsf{mod}\, 2$ of $k'$, or equivalently the group cohomology of the absolute Galois group of $k'$ with coefficients in $\Zmod$. By Voevodsky's proof we also have $\underline{K}^M_n(k')/2 \cong H^{n,n}(k',\Zmod)$ (where the right hand-side is motivic cohomology mod 2, so it is $\underline{\pi}_{-n,-n}(\HZmod)(k)$ if $k=k'$), and that cup product with $\tau\in H^{0,1}(k,\Zmod)$, induces isomorphisms $H^{p,q}(k',\Zmod) \cong H^{p,q+1}(k',\Zmod)$ when $0\leq p \leq q$ (see for example \cite[Lem.\ 6.1 and the end of the proof of Lemma 6.9]{RO}, or \cite[(7.1)]{KRO}). 
	
In particular, if $k=k'$, we get $\underline{K}^M_n(k)/2 \cong [\Spec(k),\Sigma^{n,n}\HZmod] \cong \underline{\pi}_{-n,-n}(\HZmod)(k)$. If $k =\R$, by the above $\underline{K}^M_\ast(\R)/2$ is the group cohomology $\mathsf{mod}\, 2$ of $\Zmod$ (for the trivial action). By \cite[\S III.1, Ex.\ 2]{Brown}, it is thus equal to $\Zmod$ in every non-negative degree (and trivial otherwise), and as a ring it is polynomial on one generator, namely $\rho$. Finally, the last claim follows from the statement about multiplication by $\tau \in \underline{\pi}_{0,-1}(\HZmod)(\R)$.
\end{proof}

\begin{Prop}\label{Prop:rRHZHZmod} There is an equivalence of $\E{1}$-maps
    $$\rR(\HZ \to \HZmod) \simeq (\HZmod[t^2] \to \HZmod[t])$$ 
    where $\HZ \to \HZmod$ is the projection, and $\HZmod[t^2] \to \HZmod[t]$ is the map of free $\E{1}$-$\HZmod$-algebras (see \Cref{Def:freeE1HA}) inducing the inclusion $\Zmod[t^2]\subseteq \Zmod[t]$ in homotopy.
\end{Prop}
\begin{proof}
    We proceed in three steps. We first identify the homotopy groups of the real realization of $\HZmod$ and their structure as a graded ring. Then, using this and a theorem of Hopkins and Mahowald, we identify $\rR(\HZmod)$ as a free $\E{1}$-$\HZmod$-algebra. Finally, we deduce the homotopy ring and the $\E{1}$-structure of $\rR(\HZ)$ from the cofiber sequence $\HZ \xrightarrow{\cdot 2} \HZ \to \HZmod$.\\

    \textbf{Step 1:} \emph{Computation of the homotopy ring of $\rR(\HZmod)$.} By \Cref{Prop:homotopyofrR}, $\pi_k(\rR\HZmod)$ is obtained as the colimit $\colim_n\ \homsh{k}{n}{\HZmod}(\R) = \colim_n\ \pi_{k-n,-n}(\HZmod)(\R)$. By \Cref{Prop:Milnorsconjecture}, for $k\geq 0$ and $n\geq k$ this can be rewritten as
    $$\colim_n H^{n-k,n}(\R,\Zmod) \cong \colim_n \tau^k H^{n-k,n-k}(\R,\Zmod) \cong \colim_n \tau^k\underline{K}^M_{n-k}(\R) \cong \colim_n \Zmod\{\tau^k\rho^{n-k}\}.$$
    
    This colimit is $\Zmod$ with generator $(\tau\rho^{-1})^k \in \homsh{k}{0}{\HZmod}(\R)$. It follows that $\pi_\ast(\rR\HZmod) \cong \Zmod[t]$ is a polynomial ring on a single generator $t:=\rR(\tau\rho^{-1})$ in degree 1.\\
    
     \textbf{Step 2:} \emph{Structure of $\rR(\HZmod)$ as an $\E{1}$-ring.} We will show that there exists a map of $\E{2}$-rings $\HZmod \to \rR(\HZmod)$. Here, on the left-hand side, $\HZmod$ is the topological Eilenberg-Mac Lane spectrum, whereas on the right-hand side it is the motivic Eilenberg-Mac Lane spectrum from \Cref{Def:HZ}. By \Cref{Prop:freeHA}$(iii)$, we obtain a map $\HZmod[t] \to \rR(\HZmod)$ sending $t$ on the left hand-side to the element $t\in \pi_1(\rR(\HZmod))$ we had above. Then, since by \Cref{Prop:freeHA}$(ii)$, $\pi_\ast(\HZmod[t]) \cong \Zmod[t]$ as rings, and $\pi_\ast(\HZmod[t]) \to \pi_\ast(\rR(\HZ))$ is a ring homomorphism, we deduce that the latter is an isomorphism, so $\HZmod[t] \to \rR(\HZmod)$ is an equivalence of $\E{1}$-rings. 
    
    To produce the aforementioned $\E{2}$-map $\HZmod \to \rR(\HZmod)$, we appeal to the so-called Hopkins--Mahowald theorem (\cite{Mahowald}, reformulated in \cite[Thm.\ 4.18]{MNN}): the free $\E{2}$-ring with $2=0$ is $\HZmod$, i.e., there is a pushout diagram of $\E{2}$-rings
    % https://q.uiver.app/#q=WzAsNCxbMCwwLCJGX3tcXEV7Mn19KFxcU3BoXjApIl0sWzEsMCwiXFxTcGgiXSxbMCwxLCJcXFNwaCJdLFsxLDEsIlxcSFptb2QiXSxbMCwyLCIwIiwyXSxbMCwxLCJcXGNkb3QgMiJdLFsyLDNdLFsxLDNdLFszLDAsIiIsMSx7InN0eWxlIjp7Im5hbWUiOiJjb3JuZXIifX1dXQ==
\[\begin{tikzcd}
	{F_{\E{2}}(\Sph^0)} & \Sph \\
	\Sph & \HZmod .
	\arrow["{2}", from=1-1, to=1-2]
	\arrow["0"', from=1-1, to=2-1]
	\arrow[from=1-2, to=2-2]
	\arrow[from=2-1, to=2-2]
	\arrow["\lrcorner"{anchor=center, pos=0.125, rotate=180}, draw=none, from=2-2, to=1-1]
\end{tikzcd}\]

Since $2 = 0 \in \pi_\ast(\rR\HZmod)$ by the previous computation, we obtain an $\E{2}$-map $\HZmod \to \rR(\HZmod)$.\\

    \textbf{Step 3:} \emph{Realization of $\HZ$, and of the quotient map to $\HZmod$.} Consider once more the cofiber sequence $\HZ \xrightarrow{\cdot 2} \HZ \to \HZmod$. We claim that $\rR(\cdot 2) = \cdot 2 = 0$ on $\rR\HZ$. By definition, we have $\HZ = f_0\underline{K}^M_\ast$. However, in Milnor K-theory, $2\rho = 0$. Indeed, $\rho \in [\calS^0,\Gm] \cong \homsh{0}{1}{\calS}(k) \cong K^{MW}_1(k)$ corresponds to $[-1]$ in the Milnor-Witt K-theory of $k$ \cite[Thm.\ 6.3.3 and 6.4.1]{Morel-A1intro}. But then in Milnor K-theory, we obtain $2\rho = [-1] + [-1] = [(-1)^2] = [1] = 0$. Thus, after inverting $\rho$, the multiplication by 2 map becomes the zero map. We obtain a long exact sequence (here $n\geq 0$)
% https://q.uiver.app/#q=WzAsOCxbMywxLCJcXHBpX24oXFxyUlxcSFptb2QpIFxcY29uZyBcXFptb2QiXSxbNCwxLCJcXGRvdHMiXSxbMiwxLCJcXHBpX24oXFxyUlxcSFopIl0sWzEsMSwiXFxwaV9uKFxcclJcXEhaKSJdLFswLDAsIlxcZG90cyJdLFsxLDAsIlxccGlfe24rMX0oXFxyUlxcSFopIl0sWzIsMCwiXFxwaV97bisxfShcXHJSXFxIWikiXSxbMywwLCJcXHBpX3tuKzF9KFxcclJcXEhabW9kKSBcXGNvbmcgXFxabW9kIl0sWzAsMV0sWzMsMiwiXFxjZG90IDIgPSAwIiwyXSxbMiwwLCJxX24iLDIseyJzdHlsZSI6eyJ0YWlsIjp7Im5hbWUiOiJob29rIiwic2lkZSI6InRvcCJ9fX1dLFs0LDVdLFs1LDYsIlxcY2RvdCAyID0gMCJdLFs2LDcsInFfe24rMX0iLDAseyJzdHlsZSI6eyJ0YWlsIjp7Im5hbWUiOiJob29rIiwic2lkZSI6InRvcCJ9fX1dLFs3LDMsIlxccGFydGlhbF9uIiwxLHsic3R5bGUiOnsiaGVhZCI6eyJuYW1lIjoiZXBpIn19fV1d
\[\begin{tikzcd}
	\cdots & {\pi_{n+1}(\rR\HZ)} & {\pi_{n+1}(\rR\HZ)} & {\pi_{n+1}(\rR\HZmod) \cong \Zmod} \\
	& {\pi_n(\rR\HZ)} & {\pi_n(\rR\HZ)} & {\pi_n(\rR\HZmod) \cong \Zmod} & \cdots .
	\arrow[from=1-1, to=1-2]
	\arrow["{\cdot 2 = 0}", from=1-2, to=1-3]
	\arrow["{q_{n+1}}", hook, from=1-3, to=1-4]
	\arrow["{\partial_n}"{description}, two heads, from=1-4, to=2-2]
	\arrow["{\cdot 2 = 0}"', from=2-2, to=2-3]
	\arrow["{q_n}"', hook, from=2-3, to=2-4]
	\arrow[from=2-4, to=2-5]
\end{tikzcd}\]
By \Cref{Prop:rRconnective}, all spectra involved are connective (since $\HZ$ and $\HZmod$ are very effective by \Cref{Prop:HZHZmodveryeff}). Therefore, $q_0$ is also surjective; it is an isomorphism. Then $\pi_0(\rR \HZ) \cong \Zmod$ and $\partial_0$ is a surjection $\Zmod \to \Zmod$, namely an isomorphism. Thus, the injective map $q_1$ has trivial image, and so $\pi_1(\rR \HZ) = 0$. We are back to the original situation: the long exact sequence repeats in exactly the same way. Therefore, $\pi_n(\rR \HZ)$ is $\Zmod$ if $n$ is even and non-negative and $0$ else. Let $u$ be a generator of $\pi_2(\rR \HZ)\cong \Zmod$. Since $q_2$ is an isomorphism, the real realization of the quotient map sends $u$ to the generator $t^2$ of $\pi_2(\rR \HZmod)$. Since this quotient map induces after real realization a ring homomorphism on homotopy, and an isomorphism in even degrees, we obtain $q_{2n}(u^n)=t^{2n}$, so $u^n$ generates $\pi_{2n}(\rR \HZ)$. This proves that $\pi_\ast(\rR\HZ) \cong \Zmod[t^2]$ as rings. The identification with $\HZmod[t^2]$ as $\E{1}$-rings follows as in Step 2 from the Hopkins--Mahowald theorem.
\end{proof}
\vspace{0.2cm}
\subsection{The real realizations of \texorpdfstring{$\KGL$}{KGL} and \texorpdfstring{$\kgl$}{kgl}}\label{Subsect:rRKGLkgl}\hfill\vspace{0.2cm}

The algebraic K-theory spectrum $\KGL$ (as defined in \cite[\S 6.2]{Voevodsky-ICM}, also see \cite[\S 3.2]{BH-eta} for example) appears in our computations with $\KO$ because the two of them are related by the Wood cofiber sequence ${\Sigma^{1,1}\KO \xrightarrow{\eta}\KO \xrightarrow{c} \KGL}$ where $\eta: \Gm \to \unit$ is the motivic Hopf map and $c$ is the forgetful map (Hermitian K-theory is defined using bundles with a symmetric form, and algebraic K-theory all algebraic bundles, the forgetful map ignores the form) \cite[Thm.\ 3.4]{RO}.

\begin{Lemma}[Morel; \protect{\cite[Lem.\ 3.9]{BH-eta}}]\label{Prop:rRKGL} We have $\rR\KGL\simeq 0$.
\end{Lemma}
\begin{proof}
 Let us show that $\pi_n(\rR \KGL) = 0$ for all $n\in \Z$. We want to use \Cref{Prop:levelwiserR} about computing real realization of spectra levelwise. Since $\KGL = (K,K,\dots)$ where $K$ is the motivic localization of $\Z\times BGL$ \cite{MV}, we first compute
 \begin{align*}
     \rR(K) &= \rR(\mathsf{L_{mot}}(\Z\times BGL)) \simeq \rR\left(\bigsqcup_{\Z} BGL\right) \simeq \bigsqcup_{\Z} \rR(BGL) \underset{(\star)}{\simeq} \Z \times B\left(\rR\left(\bigcup_{n\geq 1} GL_n\right)\right) \\
     &\simeq \Z \times B\left(\bigcup_{n\geq 1} GL_n(\R)\right) \underset{(\star\star)}{\simeq} \Z \times B\left(\bigcup_{n\geq 1} O_n\right) \simeq \Z \times BO.
 \end{align*}
Here $(\star\star)$ follows from the fact that $GL_n(\R)$ deformation retracts onto the orthogonal group $O_n(\R)$ by Gram-Schmidt orthogonalization, and $(\star)$ from the fact that $\rR$ commutes with colimits and finite products, in particular with geometric realizations and the bar construction. Then, by \Cref{Prop:levelwiserR}, $\pi_n(\rR\KGL) \cong \colim_m\, \pi_{n+m}(\Z\times BO)$ for all $n\in\Z$. The homotopy groups of $\Z\times BO$ are given by $\Z,\Zmod,\Zmod,0,\Z,0,0,0,\dots$ repeating with period 8 \cite{Bott}. In particular, there is a cofinal subsequence of zeroes appearing in the colimit, so all homotopy groups of $\rR\KGL$ vanish.
\end{proof}

\begin{Prop}\label{Prop:rRkgl} There is an equivalence of $\E{1}$-maps
    $$\rR(\kgl \to \HZ) \simeq (\HZmod[t^4] \to \HZmod[t^2])$$
where $\kgl \to \HZ$ is the cofiber of the map $\beta_\kgl: \mathbb{P}^1\Smash \kgl \to \kgl$ induced by the periodicity generator $\beta_\KGL$ for $\KGL$, and $\HZmod[t^4] \to \HZmod[t^2]$ is the map of free $\E{1}$-$\HZmod$-algebras (see \Cref{Def:freeE1HA}) inducing the inclusion $\Zmod[t^4]\subseteq \Zmod[t^2]$ in homotopy.
\end{Prop}
\begin{proof}
We consider the cofiber sequence $T\Smash \kgl \to \kgl \to \HZ$ of the statement \cite[Prop.\ 7]{ARO}. By \Cref{Prop:rRHZHZmod}, $\rR(\HZ) \simeq \HZmod[t^2]$. We defer to \Cref{Prop:rRbetakgl} below the proof that $\rR(\beta_\kgl)=0$. Using this, the long exact sequence induced in homotopy after real realization splits as short exact sequences
% https://q.uiver.app/#q=WzAsNSxbMywwLCJcXHBpX3tuLTJ9KFxcclJcXGtnbCkiXSxbMSwwLCJcXHBpX3tufShcXHJSXFxrZ2wpIl0sWzIsMCwiXFxwaV9uKFxcclJcXEhaKSJdLFswLDAsIjAiXSxbNCwwLCIwIl0sWzEsMl0sWzIsMF0sWzMsMV0sWzAsNF1d
\[\begin{tikzcd}
	0 & {\pi_{n}(\rR\kgl)} & {\pi_n(\rR\HZ)} & {\pi_{n-2}(\rR\kgl)} & 0
	\arrow[from=1-1, to=1-2]
	\arrow[from=1-2, to=1-3]
	\arrow[from=1-3, to=1-4]
	\arrow[from=1-4, to=1-5]
\end{tikzcd}\]
for all $n\in \Z$. For $n$ odd, the middle term vanishes, so all odd homotopy groups of $\rR\kgl$ are trivial. By \Cref{Prop:rRconnective}, $\rR\kgl$ is connective, so for $n=0$, we obtain an isomorphism $\pi_0(\rR\kgl) \to \pi_0(\rR\HZ) \cong \Zmod$. For $n=2$, the map on the right of the sequence is therefore a surjection $\Zmod \to \Zmod$, so it is an isomorphism, and $\pi_2(\rR\kgl)=0$. For $n=4$ we obtain an isomorphism $\pi_n(\rR\kgl) \to \pi_n(\rR\HZ) \cong \Zmod$. Continuing this argument inductively, $\kgl \to \HZ$ induces after real realization isomorphisms in homotopy in degrees divisible by 4, and the other homotopy groups of $\rR\kgl$ vanish. The fact that $\rR\kgl \simeq \HZmod[t^4]$ as $\E{1}$-rings is then proven exactly as in \Cref{Prop:rRHZHZmod} for $\rR\HZ$. In this construction, using the Hopkins--Mahowald theorem, the realization of the $\Einfty$-map $\kgl \to \HZ$ becomes a map of $\E{2}$-rings under $\HZmod$, in particular a map of $\E{1}$-$\HZmod$-algebras (by the proof of \Cref{Prop:freeE1Ralgebra}$(iii)$). Then, $\HZmod[t^4]$ being a free object, the real realization of $\kgl \to \HZ$ is exactly the map of $\E{1}$-$\HZmod$-algebras in the statement. 
\end{proof}

\begin{Lemma}\label{Prop:rRbetakgl} In the notation of \Cref{Prop:rRkgl}, we have $\rR(\beta_\kgl) = 0$.
\end{Lemma}
\begin{proof}
Let us abbreviate $\beta:= \beta_\kgl$. After real realization, the map $T\Smash \kgl \to \kgl$ induces multiplication by $\rR(\beta)$ in homotopy, viewing $\rR(\beta)$ as an element in $\pi_1(\rR\kgl)$ (because $\beta : T \to \kgl$ and $\rR(T) = \Sph^1$). We want to show that $\rR(\beta) \cdot 1 = 0$ where $1\in \pi_0(\rR\kgl)$, i.e., that the map $\pi_0(\rR\kgl) \to \pi_1(\rR\kgl)$ induced by $\beta$ is the zero map. Consider the cofiber sequence appearing as the last row in the following diagram (obtained by the octahedral axiom)
% https://q.uiver.app/#q=WzAsOSxbMCwwLCJcXFNpZ21hXns0LDJ9XFxrZ2wiXSxbMCwxLCJcXFNpZ21hXnsyLDF9XFxrZ2wiXSxbMCwyLCJcXFNpZ21hXnsyLDF9XFxIWiJdLFsxLDAsIlxcU2lnbWFeezQsMn1cXGtnbCJdLFsxLDEsIlxca2dsIl0sWzIsMCwiXFxTaWdtYV57MiwxfVxca2dsIl0sWzIsMSwiXFxrZ2wiXSxbMSwyLCJcXGtnbC9cXGJldGFeMiJdLFsyLDIsIlxcSFouIl0sWzAsMywiIiwwLHsibGV2ZWwiOjIsInN0eWxlIjp7ImhlYWQiOnsibmFtZSI6Im5vbmUifX19XSxbMyw0LCJcXGJldGFeMiIsMl0sWzAsMSwiXFxTaWdtYV57MiwxfVxcYmV0YSIsMl0sWzUsNiwiXFxiZXRhIiwyXSxbMyw1LCJcXFNpZ21hXnsyLDF9XFxiZXRhIl0sWzQsNiwiIiwwLHsibGV2ZWwiOjIsInN0eWxlIjp7ImhlYWQiOnsibmFtZSI6Im5vbmUifX19XSxbMSw0LCJcXGJldGEiXSxbMSwyLCJcXFNpZ21hXnsyLDF9XFxnYW1tYSIsMl0sWzQsN10sWzYsOCwiXFxnYW1tYSIsMl0sWzIsNywiIiwxLHsic3R5bGUiOnsiYm9keSI6eyJuYW1lIjoiZGFzaGVkIn19fV0sWzcsOCwiIiwxLHsic3R5bGUiOnsiYm9keSI6eyJuYW1lIjoiZGFzaGVkIn19fV1d
\[\begin{tikzcd}[column sep = 2em, row sep = 2em]
	{\Sigma^{4,2}\kgl} & {\Sigma^{4,2}\kgl} & {\Sigma^{2,1}\kgl} \\
	{\Sigma^{2,1}\kgl} & \kgl & \kgl \\
	{\Sigma^{2,1}\HZ} & {\kgl/\beta^2} & {\HZ.}
	\arrow[equals, from=1-1, to=1-2]
	\arrow["{\Sigma^{2,1}\beta}"', from=1-1, to=2-1]
	\arrow["{\Sigma^{2,1}\beta}", from=1-2, to=1-3]
	\arrow["{\beta^2}"', from=1-2, to=2-2]
	\arrow["\beta"', from=1-3, to=2-3]
	\arrow["\beta", from=2-1, to=2-2]
	\arrow["{\Sigma^{2,1}\gamma}"', from=2-1, to=3-1]
	\arrow[equals, from=2-2, to=2-3]
	\arrow[from=2-2, to=3-2]
	\arrow["\gamma"', from=2-3, to=3-3]
	\arrow[dashed, from=3-1, to=3-2]
	\arrow[dashed, from=3-2, to=3-3]
\end{tikzcd}\]
Then we have a commutative diagram
% https://q.uiver.app/#q=WzAsNyxbMSwxLCJcXHBpXzIoXFxyUihcXGtnbC9cXGJldGEpKSJdLFsyLDEsIlxccGlfMChcXHJSKFxca2dsL1xcYmV0YSkpIl0sWzMsMSwiXFxwaV8xKFxcclIoXFxrZ2wvXFxiZXRhXjIpKSJdLFs0LDEsIlxcZG90cyJdLFsyLDAsIlxccGlfMChcXHJSXFxrZ2wpIl0sWzMsMCwiXFxwaV8xKFxcclJcXGtnbCkiXSxbMCwxLCJcXGRvdHMiXSxbMCwxLCJcXHBhcnRpYWwiXSxbMSwyXSxbMiwzXSxbNCw1LCJcXGJldGEiXSxbNSwyLCJcXGNvbmciXSxbNCwxLCJcXGNvbmciXSxbNiwwXV0=
\begin{equation}\label{Diag:rRbetakgl}\begin{tikzcd}[column sep = 2em, row sep = 2em]
	&& {\pi_0(\rR\kgl)} & {\pi_1(\rR\kgl)} \\
	\cdots & {\pi_2(\rR(\kgl/\beta))} & {\pi_0(\rR(\kgl/\beta))} & {\pi_1(\rR(\kgl/\beta^2))} & \cdots
	\arrow["\beta", from=1-3, to=1-4]
	\arrow[from=1-3, to=2-3]
	\arrow[from=1-4, to=2-4]
	\arrow[from=2-1, to=2-2]
	\arrow["\partial", from=2-2, to=2-3]
	\arrow[from=2-3, to=2-4]
	\arrow[from=2-4, to=2-5]
\end{tikzcd}\end{equation}
where the bottom row is the long exact sequence induced by this cofiber sequence after realization, and the square comes from the left bottom square in the previous diagram.\\

\textbf{Step 1:} \emph{It suffices to show that $\partial$ is surjective; in other terms, that it hits $1 \in \pi_0(\rR(\kgl/\beta)) \cong \Zmod$.} Indeed, note that the proof of \Cref{Prop:rRkgl} did not use that $\rR(\beta)=0$ to obtain the isomorphisms $\pi_{0}(\rR\kgl) \cong \pi_{0}(\rR(\kgl/\beta)) \cong \pi_{0}(\rR(\HZ))\cong \Zmod$. This fact implies that the vertical map on the left-hand side in Diagram (\ref{Diag:rRbetakgl}) is an isomorphism. The vertical map on the right-hand side is an isomorphism, because in the long exact sequence induced by the real realization of the cofiber sequence $\Sigma^{4,2}\kgl \xrightarrow{\beta^2} \kgl \to \kgl/\beta^2$, the groups $\pi_i(\rR(\Sigma^{4,2}\kgl)) \cong \pi_{i-2}(\rR\kgl)$ vanish for $i=0,1$ by connectivity of $\rR\kgl$. Thus, in Diagram (\ref{Diag:rRbetakgl}), by exactness, if $\partial$ is surjective then $\beta$ is the zero map.\\
    
\textbf{Step 2:} \emph{Reduction to a computation in the homotopy of $\HZmod \in\SHk$.} By Step 1, it suffices to show that $\partial(t^2)=1$, where $t^2$ is viewed as an element in $\pi_\ast(\rR(\kgl/\beta)) \cong \pi_\ast(\rR(\HZ)) \cong \Zmod[t^2]$. By \Cref{Prop:rRHZHZmod}, the quotient map $\HZ \to \HZmod$ induces isomorphisms in even homotopy groups after real realization. Let $d: \HZ \to \Sigma^{3,1}\HZ$ be the boundary map representing $\partial$ before real realization. We claim: $(1)$ that the diagram
    % https://q.uiver.app/#q=WzAsNCxbMCwwLCJcXEhaIl0sWzEsMCwiXFxTaWdtYV57MywxfVxcSFoiXSxbMCwxLCJcXEhabW9kIl0sWzEsMSwiXFxTaWdtYV57MywxfVxcSFptb2QiXSxbMCwxLCJkIl0sWzAsMiwiXFxtYXRoc2Z7bW9kfVxcLCAyIiwyXSxbMSwzLCJcXG1hdGhzZnttb2R9XFwsIDIiXSxbMiwzLCJcXFNxXjMiXV0=
\[\begin{tikzcd}
	\HZ & {\Sigma^{3,1}\HZ} \\
	\HZmod & {\Sigma^{3,1}\HZmod}
	\arrow["d", from=1-1, to=1-2]
	\arrow["{\mathsf{mod}\, 2}"', from=1-1, to=2-1]
	\arrow["{\mathsf{mod}\, 2}", from=1-2, to=2-2]
	\arrow["{\Sq^3}", from=2-1, to=2-2]
\end{tikzcd}\]
is commutative, where $\Sq^3$ is the third motivic Steenrod square (see \cite{Voevodksy-reducedpowers}); and $(2)$ that $\Sq^3(\tau^2) = \rho^3$. Claims $(1)$ and $(2)$ imply that $\partial(t^2)=1$. Indeed, $t^2 \in \pi_2(\rR\HZmod) \cong \colim_n\ \homsh{2}{n}{\HZmod}(\R)$ corresponds to $\rho^{-2}\tau^2 \in \homsh{2}{0}{\HZmod}(\R)$ (see \Cref{Prop:Milnorsconjecture}), or equivalently, in the colimit, to $\tau^2 \in \homsh{2}{2}{\HZmod}(\R)$. It is therefore mapped to $\rho^3 \in \homsh{0}{3}{\HZmod}(\R)$, corresponding in $\colim_n\ \homsh{0}{n}{\HZmod}(\R) \cong \pi_0(\rR\HZmod)$ to $1\in \homsh{0}{0}{\HZmod}(\R)$ and thus a generator of $\Zmod \cong \pi_0(\rR\HZmod)$. This suffices because in the commutative square of Claim $(1)$, the vertical maps induce isomorphisms on the even homotopy groups after real realization.\\

    \textbf{Step 3:} \emph{We prove Claim $(1)$.} A similar statement for $d \Smash \HZmod$ instead of $(\mathsf{mod}\,2) \circ d$
	can be found in \cite[Lemma 9]{ARO}. To prove the result in our situation, by \cite[Lem.\ A.4]{RO}, we have 
    $$[\HZ,\Sigma^{3,1}\HZmod] = H^{0,0}(\Spec(\R),\Zmod)\{\Sq^3 \circ (\HZ\to\HZmod)\}= \Zmod\{\Sq^3 \circ (\HZ\to\HZmod)\}.$$
    Indeed, $H^{0,0}(\Spec(\R),\Zmod) \cong \underline{K}_0^M/2(\R) \cong \Zmod$ by \Cref{Prop:Milnorsconjecture}. Then, $(\mathsf{mod}\,2) \circ d$ is either $\Sq^3 \circ (\mathsf{mod}\,2)$ or $0$. To distinguish between them, we take complex realizations. The complex realization of the sequence $\Sigma^{2,1}\kgl \to \kgl \xrightarrow{\gamma} \HZ$ is the sequence $\Sigma^2\mathsf{ku} \to \mathsf{ku} \to \HZ$ (by \cite[in particular Thm.\ 5.5]{Levine-comparison} for $\HZ$ and \cite[Lem.\ 3]{ARO} for $\kgl$). Therefore, the complex realization of $d$ is exactly the suspension of the map $\Sigma^{-1}\HZ \to \Sigma^2\HZ$ called the first $k$-invariant of $\mathsf{ku}$. This map is the non-zero operation $Q_1$ (whose reduction $\mathsf{mod}\, 2$ is non-zero as well), see \cite[Section 2]{Bruner} (based on \cite[\S III.16]{Adams-blue-book}) for a proof and more details.\\
    
    \textbf{Step 4:} \emph{We prove Claim $(2)$.} By \cite{Voevodksy-reducedpowers}, we have the following properties of the motivic Steenrod squares: firstly, $\Sq^2(\tau)=0$ since $\Sq^{2n}(u) = 0$ for any $u$ of bidegree $(p,q)$ with $n>p-q$ and $n\geq q$, and then $\Sq^3(\tau) = \Sq^1\Sq^2(\tau) =0$; secondly, $\Sq^1=\beta$ and $\beta \tau = \rho$, and finally, we have the motivic Cartan formula
    $$ \Sq^{2i+1}(u\smile v) = \sum_{r=0}^{2i+1} \Sq^r(u) \smile \Sq^{2i+1-r}(v) + \rho\sum_{r=0}^{i-1} \Sq^{2r+1}(u) \smile \Sq^{2i-2r-1}(v). $$
    Using these facts, we compute
    \begin{align*}
        \Sq^3(\tau^2) &= \Sq^0(\tau)\Sq^3(\tau) + \Sq^1(\tau)\Sq^2(\tau) + \Sq^2(\tau)\Sq^1(\tau) + \Sq^3(\tau)\Sq^0(\tau) + \rho \Sq^1(\tau)\Sq^1(\tau) \\
        &= \rho (\beta\tau)^2 = \rho^3,
    \end{align*}
    which proves Claim $(2)$.
\end{proof}

\begin{Rmk} The long exact sequence argument in the proof of \Cref{Prop:rRbetakgl} could have been replaced by a spectral sequence argument using the real realization of the very effective slice filtration on $\kgl$ and its multiplicative structure, as we will do later in the proof of \Cref{Prop:rRbetako} for $\ko$ instead. Then, computing the differential $d^1: E^1_{2,0} = \pi_2(\rR\tilde{s}_0\kgl) \to \pi_1(\rR\tilde{s}_1\kgl) = E^1_{1,1}$ boils down to our proof of \Cref{Prop:rRbetakgl} (actually $d^1=(\pi_2\circ\rR)(d)$ where $d:\HZ \to \Sigma^{3,1}\HZ$ is the connecting map in the proof of the latter). 
\end{Rmk}
\vspace{0.2cm}
\subsection{The real realization of \texorpdfstring{$\HZtilde$}{HZ tilde}}\label{Subsect:rRHZtilde}\hfill\vspace{0.2cm}

The motivic spectrum $\HZtilde := f_0\underline{K}^{MW}_\ast$ appears in the computation of the very effective slices of $\ko$, which we will need for the proof of our main result about $\rR(\ko)$ in \Cref{Sect:rRko}. It is endowed with the structure of a motivic $\Einfty$-ring in the same way as $\HZ$ and $\HZmod$ in \Cref{Rmk:HZHZmodEinfty}.

\begin{Prop}\label{Prop:rRHZtilde} On the homotopy rings, the real realization of the quotient map $$\HZtilde := f_0\underline{K}^{MW}_\ast \longrightarrow f_0(\underline{K}^{MW}_\ast/\eta) = f_0\underline{K}^{M}_\ast = \HZ$$ identifies with the quotient map $\Z[t^2]/(2t^2) \to \Z[t^2]/2$.
\end{Prop}
\begin{proof}
   The homotopy ring of $\rR\HZ$ was computed in \Cref{Prop:rRHZHZmod}. By \cite[Lem.\ 17]{Bachmann-genslices}, the quotient map in the statement induces isomorphisms on homotopy sheaves $\underline{\pi}_n(-)_\ast$ for all $n\neq 0$. In particular, for all $n\neq 0$, $\pi_n(\rR \HZtilde) \cong \colim_m \underline{\pi}_n(\HZtilde)_m(\R) \cong \colim_m \underline{\pi}_n(\HZ)_m(\R) \cong \pi_n(\rR\HZ)$. Since $\HZtilde$ is very effective by \Cref{Prop:HZHZmodveryeff}, its real realization is connective (\Cref{Prop:rRconnective}), and so we only have to deal with the zeroth homotopy group. Consider the unit map $\unit \to \HZtilde$. By \cite[Lem.\ 17]{Bachmann-genslices} and \cite[after Thm.\ 6.4.7]{Morel-A1intro}, this map induces isomorphisms on the homotopy sheaves $\underline{\pi}_0(-)_\ast$, which agree with those of the sphere spectrum. In particular, arguing as above, $\pi_0(\rR\HZtilde) \cong \pi_0(\rR\underline{K}^{MW}_\ast) \cong \pi_0(\rR\calS) = \pi_0(\Sph) = \Z$. We therefore have in degree 0 a ring map $\Z \to \Zmod$, which must be the quotient by 2 map. This suffices to conclude our proof, because we have determined $\pi_\ast(\rR \HZtilde) \to \pi_\ast(\rR\HZ)$ in every degree, we know the ring structure on the right hand-side, and we know that the map is a ring homomorphism.
\end{proof}
\vspace{0.2cm}
\subsection{The real realization of \texorpdfstring{$\KO$}{KO}}\label{Subsect:rRKO}\hfill\vspace{0.2cm}

We now compute the real realization of $\KO$ as an $\Einfty$-ring. The result was already proven at the level of spectra (see \cite[Lem.\ 3.9]{BH-eta}); the proof here is very similar, with more care to the multiplicative structure. In the proof, we will encounter the motivic $\Einfty$-ring $\KW := \KO[\eta^{-1}]$ representing Balmer--Witt K-theory, see for example \cite[Appendix A]{Hornbostel}. Let us fix once and for all the $\Einfty$-structures under consideration. Recall that we assume $\mathsf{char}(k) \neq 2$.

\begin{Def}\label{Def:KO}
	Let $\KO \in \SH(k)$ be the Hermitian $K$-theory motivic $\Einfty$-ring and $\KW := \KO[\eta^{-1}]\in \SH(k)$ be the Balmer--Witt K-theory motivic $\Einfty$-ring, both with $\Einfty$-structure as constructed in \cite[Def.\ 8.1.1]{CHN} (where they are denoted by $\mathsf{KQ}_{\Spec(k)}$ and $\mathsf{KW}_{\Spec(k)}$).
\end{Def}
\begin{Rmk}\label{Rmk:Einfty-stc-KO}
We will later need the following fact (for \Cref{Prop:rRko(2)}) : this $\Einfty$-structure on $\KO$ agrees with the one considered in \cite[\S7, p.\ 129]{HJNY}. Indeed, by \cite[Lem.\ 7.4]{HJNY}, this $\Einfty$-structure is unique with the property of lifting that of a certain presheaf of $\Einfty$-rings, namely $\mathcal{GW}$. The $\Einfty$-structure from \cite[Def.\ 8.1.1]{CHN} satisfies this condition (the different $\mathcal{GW}$'s used in the two references agree by \cite[\S 4.6, Prop.\ 4.6.1]{CHN}).
\end{Rmk}
\begin{Rmk}\label{Rmk:map-KO-KW}
	With the $\Einfty$-structures from \Cref{Def:KO}, the canonical $\eta$-localization map $\KO \to \KW$ is then a map of $\Einfty$-rings. 
	Indeed, \cite[below Cor.\ 8.1.8]{CHN} provides an $\Einfty$-map $w_\R: \mathsf{KQ} \to \KW$. 
	We claim that the underlying map is giving by $\eta$-localization, 
	which can be described as the natural map to the colimit of the telescope for multiplication by $\eta$. 
	By construction, $w_\R$ is obtained by applying the lax symmetric monoidal, colimit-preserving 
	functor $\mathcal{M}^s_\R$ \cite[Lem.\ 7.4.10]{CHN} to the map of $\mathbb{GW}$-modules $\mathbb{GW} \to \mathbb{L}$ 
	between stabilized Grothendieck-Witt theory and stabilized L-theory. By \cite[\S 3]{Schlichting-2025}, the latter is
	also given by $\eta'$-localization, where $\eta'$ is the image of the unit in $GW^{[0]}_0(\Z)$ by the boundary map 
	from \cite[Rmk.\ 8.1.11]{CHN}. It is shown there that after applying the functor $\mathcal{M}^s_\R$, this boundary 
	map is the usual $\eta$, namely the motivic Hopf map. 
\end{Rmk}

\begin{Prop}\label{Prop:rRKO} There is an $\Einfty$-map of ring spectra $\gamma: \mathsf{L}(\R) \longrightarrow \rR\KW$,
    where $\mathsf{L}$ denotes the L-theory spectrum (see \cite[\S 2.2]{LNS}), inducing an equivalence of $\Einfty$-rings $\mathsf{L}(\R)[1/2] \simeq \rR\KW$.
	In particular, there is an equivalence of $\Einfty$-rings $\rR\KO \simeq \rR\KO[1/2] \simeq \rR\KW \simeq \KOtop[1/2]$.
\end{Prop}
\begin{proof}
		By \cite[Cor.\ 8.1.8]{CHN}, we have $\Gamma(\R,\KW) \simeq \mathsf{L}(\R)$, 
		where $\Gamma(\R,-)$ is the global sections functor (i.e., the right adjoint 
		to the unique colimit-preserving symmetric monoidal functor $\Sp \to \SH(\R)$).
		Note that the latter functor is lax symmetric monoidal and preserves filtered colimits
		(its left adjoint preserves compact objects). We have to argue that this equivalence is indeed one of $\Einfty$-rings. 
		Using the notation from \emph{loc.\ cit.,} we have $\KW := \mathcal{M}^s_\R(\mathbb{L}) \in \CAlg(\SH(\R))$,
		and this equivalence is obtained by applying the lax monoidal functor $\Gamma(-,\R)$ 
		to the map of $\Einfty$-algebras $L^s \to \mathbb{L}^s \to \mathsf{Loc}_\Affl \mathbb{L}^s$
		and is therefore an $\Einfty$-map.\\

We now claim that $\KW[1/\rho]\simeq \KW[1/2]$. Since $\eta$ is by definition invertible on $\KW$, the relation $\eta h = 0$ in Milnor--Witt K-theory becomes $0 = h = 2 + \rho\eta$ and thus $\rho\eta = -2$. Therefore, inverting $\rho$ on $\KW$ is equivalent to inverting $2$. Then, by \Cref{Prop:rRinvertingrho}, we have 
$$\rR\KW \simeq \Gamma(\R,\KW[1/\rho]) \simeq \Gamma(\R,\KW[1/2]) \simeq \Gamma(\R,\KW)[1/2] \simeq \mathsf{L}(\R)[1/2].$$
By \cite[Ex.\ 9.2]{LNS}, there is an $\Einfty$-map $\tau_\R: \kotop \to L(\mathbb{R})_{\geq 0}$, inducing an equivalence after inverting 2 (note that in \emph{loc.\ cit.,} the definition of the $\Einfty$ $L$-theory that is used is the same as the one used in \cite{CHN}, see \cite[\S 2.2]{LNS} and \cite[\S 2.3, p.\ 9]{CHN}). Moreover, on the eighth homotopy groups, this maps sends $\lambda_8$ (a generator for $\pi_8(\kotop) \cong \Z$) to $16b^2$ (where $b$ is a generator of $\pi_{4}(\mathsf{L}(\R)) \cong \Z$). Therefore, there is an equivalence 
$$\kotop[1/2][\lambda_8^{-1}] \simeq \mathsf{L}(\R)_{\geq 0}[1/2][(16b^2)^{-1}] \simeq \mathsf{L}(\R)_{\geq 0}[1/2][b^{-1}].$$

The left-hand side is equivalent to $\KOtop[1/2]$ since $\kotop[\lambda_8^{-1}] \simeq \KOtop$ by periodicity (see for example \cite[Cor.\ 5.1]{LN}), and similarly the right-hand side is $\mathsf{L}(\R)[1/2]$ because $\mathsf{L}(\R)_{\geq 0}[b^{-1}] \simeq \mathsf{L}(\R)$ (same reference). This proves the last equivalence in the statement. \\

To prove the first equivalence, consider the Wood cofiber sequence $\Sigma^{1,1}\KO \xrightarrow{\ \eta\ } \KO \longrightarrow \KGL$ \cite[Thm.\ 3.4]{RO}. In Milnor--Witt K-theory, we have $h\rho = 0$. Indeed, $h\rho = (2 + \rho\eta)\rho = 2\rho + \rho^2\eta = [(-1)(-1)] = [1] = 0$ (since $\rho = [-1]$, see \cite[Thm.\ 6.3.3 and 6.4.1]{Morel-A1intro}). Therefore, similarly as above, after inverting $\rho$, or equivalently, after real realization (\Cref{Prop:rRinvertingrho}), we have $0 = h = 2 + \rho\eta$. Then the real realization of $\eta$ is equivalent to multiplication by $-2$. But also, we have seen in \Cref{Prop:rRKGL} that $\rR\KGL\simeq 0$. This means that the cofiber of the multiplication by 2 map on $\rR\KO$ is zero, in other terms $\cdot 2$ is an equivalence, and thus $\rR\KO \simeq \rR(\KO[1/2])$ via the canonical localization map. Also, the above implies that the canonical localization maps $ \KO[1/2] \longrightarrow \KO[\eta^{-1},1/2] \simeq \KW[1/2] \longleftarrow \KW $ (see \Cref{Rmk:map-KO-KW}) are sent to equivalences under $\rho$-localization, namely under real realization.
Thus, $\rR(\KO[1/2]) \simeq \rR(\KW)$. This establishes the first and second equivalences in the statement.
\end{proof}    

Let $\kw := \KW_{\geq 0}$ be the connective cover of $\KW$ in the homotopy t-structure. It inherits the structure of a motivic $\Einfty$-ring (\Cref{Rmk:smstconheart}). Also let $\kotop := \KOtop_{\geq 0}$ be the connective cover of the real topological K-theory spectrum. 

\begin{Prop}\label{Prop:rRko1/2} The equivalences $\rR\KO[1/2] \simeq \rR\KW \simeq \KOtop[1/2]$ of \Cref{Prop:rRKO} induce equivalences of $\Einfty$-rings 
   $\kotop[1/2] \simeq \rR\kw \simeq \rR\ko[1/2]$.
\end{Prop}
\begin{proof}
By \cite[Lem.\ 6.9]{BH-eta}, the composition $\ko\to\KO\to\KO[\eta^{-1}] =: \KW$ of the canonical maps induces an equivalence ${\ko[\eta^{-1}] \simeq \KW_{\geq 0} =: \kw}$. Then, as in the proof of \Cref{Prop:rRKO}, localization away from $\eta$ and away from 2 induce equivalences $\rR\kw \simeq \rR\kw[1/2] \simeq \rR\ko[1/2]$. Note that the functor $\rR$ is $t$-exact: by \Cref{Def:homotopytstc}, $E\in\SHk$ belongs to $\SHk_{\geq 0}$ if and only if $\homsh{i}{\ast}{E}=0$ for all $i<0$ (respectively, $\SHk_{\leq 0}$ and $i>0$), and by \Cref{Prop:homotopyofrR} we have $\pi_i(\rR(E)) \cong \colim_n \homsh{i}{n}{E}(\R)$ for all $i\in\Z$. Therefore, $\rR\kw \simeq \rR(\KW)_{\geq 0} \simeq \KOtop[1/2]_{\geq 0} \simeq \kotop[1/2]$ by \Cref{Prop:rRKO}. 
\end{proof}  

\vspace{0.5em}
\section{The real realization of \texorpdfstring{$\ko$}{ko}}\label{Sect:rRko}

The motivic spectrum $\ko$ inherits an $\Einfty$-ring structure from that of $\KO$, because the very effective cover functor is lax symmetric monoidal by \Cref{Prop:smheart}. Thus, $\rR\ko$ is also an $\Einfty$-ring. In this section, we identify the underlying $\E{1}$-structure by an explicit 2-local fracture square. In \Cref{Subsect:pistarrRko}, we will compute the homotopy ring of $\rR\ko$, using the Wood cofiber sequence $\Sigma^{1,1}\ko \to \ko \to \kgl$ \cite[Prop.\ 11]{ARO} and the knowledge of the very effective slices of $\ko$ (computed in \cite{Bachmann-genslices}, recalled as \Cref{Prop:veffslicesko}). The description of these slices involves only motivic spectra whose real realizations we have already computed in \Cref{Sect:examplesrealization}. We will show that the homotopy ring $\pi_\ast\rR(\ko)$ is a polynomial ring over $\Z$ with a single generator in degree 4.\\

But this is not enough to identify $\rR\ko$ as an $\E{1}$-ring. Indeed, by \Cref{Prop:freeE1Ralgebra}, the free $\E{1}$-$\HZ$-algebra $\HZ[t^4]$ has the same homotopy ring, but they are not equivalent (see \Cref{Rmk:kinvariants}). Inverting 2 makes things easier; we have already identified $\rR(\ko)[1/2]$ in \Cref{Prop:rRko1/2}. This might be a good reason to consider a 2-local fracture square, expressing $\rR\ko$ as the pullback of $\rR(\ko)[1/2]$ and $\rR(\ko)_{(2)}$ over $\rR(\ko)_\Q$. Our approach is to identify these terms and the maps in the square more explicitly, taking into account the $\E{1}$-structures. This will use the description of free $\E{1}$-$\HZ_{(2)}$-algebras and their homotopy rings from \Cref{Prop:freeHA}; this is one reason for which we identify the $\E{1}$-structure only (and not the $\Einfty$-one). The resulting 2-local fracture square coincides with the one obtained in \cite{HLN} for the connective L-theory spectrum of $\R$. As a corollary, we get an equivalence of $\E{1}$-rings $\rR(\ko)\simeq\mathsf{L}(\R)_{\geq 0}$.
\vspace{0.2cm}

\subsection{The homotopy ring of \texorpdfstring{$\rR(\ko)$}{rR(ko)}}\label{Subsect:pistarrRko}\hfill\vspace{0.2cm}

\begin{Prop}\label{Prop:homgprRko} The homotopy ring of $\rR\ko$ is the polynomial ring $\pi_\ast(\rR\ko) \cong \Z[x]$, where $x$ is a generator in degree 4, mapping to a generator of $\pi_4(\rR\kgl)\cong \Zmod$ (see \Cref{Prop:rRkgl}) in the real realization of the Wood cofiber sequence $\Sigma^{1,1}\ko \to \ko \to \kgl$.
\end{Prop}

\begin{proof}
    We will proceed in four steps:
\begin{enumerate}
    \item We first show that these homotopy groups are finitely generated in every degree, using the description of the very effective slices (\Cref{Def:effcover}) of $\ko$.
    \item Using the Wood cofiber sequence, we show that these groups have no 2-torsion.
    \item We then show that they have no torsion at the other primes, and that they are actually free abelian groups of rank 1 in every non-negative degree divisible by 4, and 0 else. To do so, we reduce to a statement after inverting 2, and use the computation of $\rR\ko[1/2]$ in \Cref{Prop:rRko1/2}.
    \item Finally, we describe the ring structure, studying again the Wood cofiber sequence. 
\end{enumerate}
\vspace{0.2cm}
    \textbf{Step 1:} \emph{The homotopy groups of $\rR\ko$ are finitely generated abelian groups.} We claim that: $(\star)$ the real realizations of the very effective slices of $\ko$ have degreewise finitely generated homotopy groups. Then, we show by induction on $m$ the following statement: for all $n\in \N$, $\pi_{n+m}(\rR(\tilde{f}_n \KO))$ is finitely generated. This holds for $m\leq -1$ since real realization on $\SHR^\veff(n)$ takes values in $n$-connective spectra as seen in \Cref{Prop:rRconnective}. This implies our result because for $n=0$, we get that for all $m\in\N$, $\pi_m(\rR(\tilde{f}_0 \KO)) = \pi_m(\rR\ko)$ is finitely generated. Assume that the induction hypothesis holds for some fixed $m\geq -1$. Then for all $n\in \N$, the cofiber sequence $\tilde{f}_{n+1}\KO \to \tilde{f}_{n}\KO \to \tilde{s}_{n}\KO$ defining the very effective slices induces a long exact sequence
% https://q.uiver.app/#q=WzAsNyxbMSwwLCJcXHBpX3tuKzErbX0oXFxyUihcXHRpbGRle2Z9X3tuKzF9IFxcS08gKSkiXSxbMiwwLCJcXHBpX3tuK20rMX0oXFxyUihcXHRpbGRle2Z9X24gXFxLTykpIl0sWzMsMCwiXFxwaV97bittKzF9KFxcclIoXFx0aWxkZXtzfV9uIFxcS08pKSJdLFswLDAsIlxcZG90cyJdLFsyLDEsIlxccGlfe24rbX0oXFxyUihcXHRpbGRle2Z9X3tufSBcXEtPICkpIl0sWzEsMSwiXFxwaV97bittfShcXHJSKFxcdGlsZGV7Zn1fe24rMX0gXFxLTyApKSJdLFszLDEsIlxcZG90cyJdLFswLDFdLFsxLDJdLFsyLDUsImIiLDFdLFszLDAsImEiXSxbNSw0XSxbNCw2XV0=
\[\begin{tikzcd}[ampersand replacement=\&]
	\cdots \& {\pi_{n+1+m}(\rR(\tilde{f}_{n+1} \KO ))} \& {\pi_{n+m+1}(\rR(\tilde{f}_n \KO))} \& {\pi_{n+m+1}(\rR(\tilde{s}_n \KO))} \\
	\& {\pi_{n+m}(\rR(\tilde{f}_{n+1} \KO ))} \& {\pi_{n+m}(\rR(\tilde{f}_{n} \KO ))} \& \cdots
	\arrow["a", from=1-1, to=1-2]
	\arrow[from=1-2, to=1-3]
	\arrow[from=1-3, to=1-4]
	\arrow["b"{description}, from=1-4, to=2-2]
	\arrow[from=2-2, to=2-3]
	\arrow[from=2-3, to=2-4]
\end{tikzcd}\]
and therefore $\pi_{n+m+1}(\rR(\tilde{f}_{n}\KO))$ is an extension of $\mathsf{ker}(b)$ and $\mathsf{coker}(a)$. The former is finitely generated because it is a subgroup of $\pi_{n+m+1}(\rR(\tilde{s}_n \KO))$, which is finitely generated by $(\star)$. The latter is also finitely generated since it is a quotient of $\pi_{n+1+m}(\rR(\tilde{f}_{n+1} \KO) )$, which is finitely generated by the induction hypothesis. Therefore, $\pi_{n+m+1}(\rR(\tilde{f}_{n}\KO))$ is finitely generated, finishing our proof by induction. To finish Step 1, it remains to prove $(\star)$.

\begin{thm}[\protect{\cite[Thm.\ 16]{Bachmann-genslices}}]\label{Prop:veffslicesko}
The very effective slices of $\KO$ are given by $\tilde{s}_n\KO \simeq \Sigma^{2n,n}\tilde{S}_{n\, \mathsf{mod}\, 4}$ where
    $$\tilde{S}_{0} = \tilde{s}_0\KO,\quad\tilde{S}_{1} = \HZmod,\quad\tilde{S}_2 =\HZ, \quad\tilde{S}_3 = 0,$$
and there is a cofiber sequence $ \Sigma^{1,0}\HZmod \longrightarrow \tilde{s}_0\KO \longrightarrow \HZtilde$.
\end{thm} 

By the computation of the real realizations of $\HZ$, $\HZmod$ and $\HZtilde$ in \Cref{Subsect:rRHZHZmod,Subsect:rRHZtilde}, we know that their homotopy groups are finitely generated in every degree. Thus, that of the realizations of very effective slices of $\ko$ are too (for $\tilde{s}_0$, use the same inductive argument as we just did).\\

    \textbf{Step 2:} \emph{There is no 2-torsion in $\pi_\ast(\rR\ko)$.} The Wood cofiber sequence seen in the proof of \Cref{Prop:rRKO} refines to one on the very effective covers: by \cite[Prop.\ 11]{ARO}, there is a cofiber sequence $\Sigma^{1,1}\ko \xrightarrow{\eta} \ko \xrightarrow{c} \kgl$, where $\eta$ is induced by the Hopf map and $c$ is induced by the forgetful map from Hermitian K-theory to algebraic K-theory. As we computed in the case of $\KO$ (\Cref{Prop:rRKO}), the real realization of $\eta$ is the multiplication by -2 map on the sphere spectrum, and by \Cref{Prop:rRKGL}, the $\E{1}$-ring $\rR\kgl$ is the free $\E{1}$-$\HZmod$-algebra $\HZmod[t^4]$. Assume for a contradiction that $\pi_m(\rR\ko)$ has 2-torsion for some $m\geq 0$. Then, in the long exact sequence induced by the Wood cofiber sequence
% https://q.uiver.app/#q=WzAsNixbMSwwLCJcXHBpX3tuKzF9KFxcclJcXGtnbCApIl0sWzIsMCwiXFxwaV97bn0oXFxyUlxca28pIl0sWzMsMCwiXFxwaV97bn0oXFxyUlxca28pIl0sWzAsMCwiXFxkb3RzIl0sWzUsMCwiXFxkb3RzIl0sWzQsMCwiXFxwaV97bn0oXFxyUlxca2dsKSJdLFswLDEsImEiXSxbMiw1LCJjIl0sWzMsMF0sWzUsNF0sWzEsMiwiXFxjZG90IDIgPSBiIl1d
\[\begin{tikzcd}
	\cdots & {\pi_{m+1}(\rR\kgl )} & {\pi_{m}(\rR\ko)} & {\pi_{m}(\rR\ko)} & {\pi_{m}(\rR\kgl)} & \cdots,
	\arrow[from=1-1, to=1-2]
	\arrow["a", from=1-2, to=1-3]
	\arrow["{\cdot (-2) = b}", from=1-3, to=1-4]
	\arrow["c", from=1-4, to=1-5]
	\arrow[from=1-5, to=1-6]
\end{tikzcd}\]
the kernel of the map $\cdot (-2) = b$ is non-trivial, so $a$ is non-zero and $\pi_{m+1}(\rR\kgl) \neq 0$, i.e., $m\equiv 3 \mod 4$. Then $\pi_{m}(\rR\kgl) = 0$, meaning that multiplication by 2 is surjective on $\pi_{m}(\rR\ko)$. In Step 1, we saw that this group was finitely generated abelian. Using the classification of such groups, we see that this is impossible for non-zero groups with 2-torsion.\\
    
    \textbf{Step 3:} \emph{The groups $\pi_m(\rR\ko)$ are torsion-free of rank 1 in degrees $m\in 4\N$, and zero otherwise.} By \Cref{Prop:rRko1/2} and since real realization and homotopy groups commutes with localization away from $2$ (\Cref{Prop:rRcommuteswithloc}), we have isomorphisms of graded abelian groups
    \begin{align*}
        \pi_\ast(\rR\ko)[1/2] &\cong \pi_\ast((\rR\ko)[1/2]) \cong \pi_\ast(\rR(\ko[1/2])) \cong \pi_\ast(\kotop[1/2]) \\
        &\cong \left(\factor{\Z[\alpha_1,\beta_4,\lambda_8]}{(\alpha_1^3, 2\alpha_1,\alpha_1\beta_4, \beta_4^2 - 4\lambda_8)}\right)[1/2] \cong \Z[1/2][\beta_4]
    \end{align*}
    (where $|\alpha_1|=1$, $|\beta_4|=4$, and $|\lambda_8|=8$) by a classical result in topology, see for instance \cite[\S 1.1]{Strickland}. In particular, the homotopy groups of $\rR\ko$ cannot have any $p$-torsion for $p\neq 2$ either. Otherwise, their localizations away from 2 would still have torsion, but $\Z[1/2]$ does not. Therefore, $\pi_\ast(\rR\ko)$ is degreewise free. Comparing with the localization away from 2, we deduce that it consists of a copy of $\Z$ in every degree divisible by 4, and the trivial group else. Indeed, if $\Z^{\oplus m}[1/2] \cong \Z[1/2]$ for some $m\in\N$, then the quotients by the $\Z$-module map $\cdot 3$ on both sides are isomorphic as well. They are respectively $(\Z/3)^{\oplus m}$ and $\Z/3$, whence $m=1$ for cardinality reasons.\\
    
    \textbf{Step 4:} \emph{We describe the ring structure.} The Wood cofiber sequence induced a long exact sequence
 % https://q.uiver.app/#q=WzAsNSxbMSwwLCJcXHBpX3s0bX0oXFxyUlxca28gKSJdLFsyLDAsIlxccGlfezRtfShcXHJSXFxrZ2wpIFxcY29uZyBcXFptb2QiXSxbMywwLCJcXHBpX3s0bS0xfShcXHJSXFxrbyk9MCJdLFswLDAsIlxcZG90cyJdLFs0LDAsIlxcZG90cyJdLFswLDEsImNfezRtfSJdLFsyLDRdLFszLDBdLFsxLDJdXQ==
\[\begin{tikzcd}
	\cdots & {\pi_{4m}(\rR\ko )} & {\pi_{4m}(\rR\kgl) \cong \Zmod} & {\pi_{4m-1}(\rR\ko)=0} & \cdots.
	\arrow[from=1-1, to=1-2]
	\arrow["{c_{4m}}", from=1-2, to=1-3]
	\arrow[from=1-3, to=1-4]
	\arrow[from=1-4, to=1-5]
\end{tikzcd}\]
Let $x$ be a generator of $\pi_4(\rR\ko)\cong \Z$. Then since $c_4$ is surjective, $c_4(x)$ generates $\pi_{4}(\rR\kgl) \cong \Z/2$. The forgetful map $c: \ko \to \kgl$ in the Wood cofiber sequence is a map of $\Einfty$-rings, and thus induces a ring homomorphism in homotopy. Therefore, for $m\geq 0$, $c_{4m}(x^m) = c_4(x)^m$ generates $\pi_{4m}(\rR\kgl) \cong \Z/2$ as well (since $\pi_\ast(\rR\kgl)$ is a polynomial ring by \Cref{Prop:rRkgl}). We claim that $x^m$ generates $\pi_{4m}(\rR\ko)$. Indeed, write $x^m = a\cdot g_m$ with $a\in \Z$ and $g_m$ is a generator of $\pi_{4m}(\rR\ko)$. By \Cref{Prop:rRko1/2}, we have ring isomorphisms $\pi_\ast(\rR\ko) \otimes \Z[1/2] \cong \pi_\ast(\kotop[1/2]) \cong \Z[1/2][\beta_4]$. In degree 4, since $x$ generates $\pi_4(\rR\ko)$, under this isomorphism $x$ must be sent to a generator of $\Z[1/2]\{\beta_4\}$ as a $\Z[1/2]$-module, i.e., $c\beta_4$ with $c\in\Z[1/2]^\times$. In particular, $x^m$ corresponds to a generator $c^m\beta_4^m$ of the degree $4m$ part. Therefore, there exists $b\in\Z[1/2]$ with $g_m = b\cdot x^m = ba\cdot g_m$, so $a\in\Z$ is a unit in $\Z[1/2]$. Thus, if $a\neq \pm 1$, then it must be even. In the latter case, we would have $c_{4m}(x^m) = a \cdot c_{4m}(g_m) = 0 \in \Z/2$, which is a contradiction. We conclude that $a=\pm 1$, whence our claim follows.
\end{proof}

\subsection{A 2-local fracture square for the \texorpdfstring{$\E{1}$}{E1}-ring structure on \texorpdfstring{$\rR(\ko)$}{rR(ko)}}\label{Subsect:E1ringstc}

\subsubsection{Identifying the localization at \texorpdfstring{$(2)$}{(2)}}\label{SubSubsect:rRko(2)}

Using that $\pi_\ast(\rR\ko) \cong \Z[x]$ by \Cref{Prop:homgprRko}, we have ${\pi_\ast((\rR\ko)_{(2)}) \cong \pi_\ast(\rR\ko) \otimes_\Z \Z_{(2)} \cong \Z_{(2)}[x]}$. We want to use this to identify $(\rR\ko)_{(2)}$ with the free $\E{1}$-$\mathsf{H}\Z_{(2)}$-algebra on one generator in degree 4. To do so, we want to produce an $\E{2}$-map $\mathsf{H}\Z_{(2)} \to (\rR\ko)_{(2)}$. Inspired by the computation of a 2-local fracture square for $\mathsf{L}(\R)_{\geq 0}$ in \cite[p.\ 3]{HLN}, we use the Thom spectrum $\MSO \in \Sp$ and the motivic Thom spectrum $\MSL \in \SH(\R)$ as intermediate steps. To do so, we need to show that $\rR(\MSO) \simeq \MSL$ as $\E{2}$-rings. They are actually equivalent as $\Einfty$-rings, as we prove in \cite{rRThom-arXiv} by comparing the motivic and topological Thom spectrum functors under real realization.

\begin{thm}[\protect{\cite{rRThom-arXiv}}]\label{Prop:MSL} There is an equivalence of $\Einfty$-rings $\rR\MSL \simeq \MSO$
    between the real realization of the special linear motivic cobordism spectrum, and the oriented cobordism spectrum. 
    Here, the $\Einfty$-structure on $\MSL$ and $\MSO$ come respectively from their descriptions as a motivic 
    or topological Thom spectrum (see \cite[Def.\ 16.1 and Ex.\ 16.22]{BH-norms} and \cite[Thm.\ 1.6]{ABG}, \cite{rRThom-arXiv} respectively).
\end{thm}

\begin{Prop}\label{Prop:rRko(2)} There is an equivalence of $\E{1}$-rings $\HZ_{(2)}[t^4] \lsimeq{} (\rR\ko)_{(2)}$
sending $t^4\in\pi_4(\HZ_{(2)}[t^4])$ to (the image in the localization of) the generator $x \in \pi_4(\rR\ko)$.
\end{Prop}

\begin{proof}
    By \cite[Thm.\ 10.1]{HJNY}, there is an $\Einfty$-ring morphism $\mathsf{MSL} \to \ko$. Note that the reference uses the $\Einfty$-structure on $\MSL$ coming from the motivic Thom spectrum functor of \cite{BH-norms}, i.e., the same as in \Cref{Prop:MSL}. Indeed, it relies on \cite{EHKSY}, which in turn uses this $\Einfty$-structure, see \cite[\S 3.3, p.\ 24]{EHKSY}. As for the $\Einfty$-structure on $\KO$, it agrees with the one we chose to consider by \Cref{Rmk:Einfty-stc-KO}. 

    The real realization of this morphism is an $\Einfty$-map $\MSO \simeq \rR(\MSL) \to \rR\ko$ (\Cref{Prop:MSL}). Now, by \cite[Cor.\ 3.7]{HLN}, $\MSO$ admits after localization at $(2)$ a map of $\E{2}$-rings $\HZ_{(2)} \to \mathsf{MSO}_{(2)}$ (where the ring structure on $\mathsf{MSO}$ also comes from its description as a (topological) Thom spectrum, as in \Cref{Prop:MSL}). Therefore, by composition, we obtain an $\E{2}$-map $\HZ_{(2)} \to (\rR\ko)_{(2)}$. By \Cref{Prop:freeHA}$(iii)$, to obtain a map of $\E{1}$-rings $\HZ_{(2)}[t] \to (\rR\ko)_{(2)}$, it then suffices to specify an element of $\pi_4((\rR\ko)_{(2)})$. We choose the image in the localization of $x\in \pi_4(\rR\ko)$. Using that $\pi_\ast((\rR\ko)_{(2)}) \cong \pi_\ast(\rR\ko)\otimes_\Z \Z_{(2)} \cong \Z_{(2)}[x]$ and \Cref{Prop:freeHA}$(ii)$, we deduce that this map induces isomorphisms on the homotopy groups (note that $\Z_{(2)}$ is a module over itself in a unique way), whence the desired equivalence $\HZ_{(2)}[t^4]\lsimeq{} (\rR\ko)_{(2)}$.  
\end{proof}

\begin{Rmk}\label{Rmk:kinvariants}
    Localization at $(2)$ is crucial here, because $\mathsf{MSO}$ is not an $\HZ$-module, since it has non-trivial $k$-invariants. The $(n-1)$-th $k$-invariant of $E\in\Sp$ is the horizontal composition in the diagram
% https://q.uiver.app/#q=WzAsNyxbMSwwLCJFX3tcXGxlcSBuKzJ9Il0sWzEsMSwiRV97XFxsZXEgbisxfSJdLFsxLDIsIkVfe1xcbGVxIG59Il0sWzIsMiwiXFxTaWdtYV57bisyfVxcbWF0aHNme0h9KFxccGlfe24rMX1FKSIsWzExNSw3OCwzNiwxXV0sWzAsMSwiXFxTaWdtYV57bisxfVxcbWF0aHNme0h9KFxccGlfe24rMX1FKSIsWzExNSw3OCwzNiwxXV0sWzAsMCwiXFxTaWdtYV57bisyfVxcbWF0aHNme0h9KFxccGlfe24rMn1FKSIsWzI0MCw2MCw2MCwxXV0sWzIsMSwiXFxTaWdtYV57biszfVxcbWF0aHNme0h9KFxccGlfe24rMn1FKSIsWzI0MCw2MCw2MCwxXV0sWzEsMiwiIiwwLHsiY29sb3VyIjpbMTE1LDc4LDM2XX1dLFswLDEsIiIsMCx7ImNvbG91ciI6WzI0MCw2MCw2MF19XSxbMiwzLCIiLDAseyJjb2xvdXIiOlsxMTUsNzgsMzZdfV0sWzQsMSwiIiwwLHsiY29sb3VyIjpbMTE1LDc4LDM2XX1dLFs1LDAsIiIsMCx7ImNvbG91ciI6WzI0MCw2MCw2MF19XSxbMSw2LCIiLDAseyJjb2xvdXIiOlsyNDAsNjAsNjBdfV1d
\[\begin{tikzcd}[row sep = 0.85em]
	{\Sigma^{n+2}\mathsf{H}(\pi_{n+2}E)} & {E_{\leq n+2}} \\
	{\Sigma^{n+1}\mathsf{H}(\pi_{n+1}E)} & {E_{\leq n+1}} & {\Sigma^{n+3}\mathsf{H}(\pi_{n+2}E)} \\
	& {E_{\leq n}} & {\Sigma^{n+2}\mathsf{H}(\pi_{n+1}E),}
	\arrow[BrickRed, dashed, from=1-1, to=1-2]
	\arrow[BrickRed, dashed, from=1-2, to=2-2]
	\arrow[MyGreen, from=2-1, to=2-2]
	\arrow[BrickRed, dashed, from=2-2, to=2-3]
	\arrow[MyGreen,from=2-2, to=3-2]
	\arrow[MyGreen,from=3-2, to=3-3]
\end{tikzcd}\]
    where the right-down-right zigzags (one showed in dashed arrows, the other one in solid arrows) are distinguished triangles, or the maps defined in the same way using the Whitehead tower instead of the Postnikov tower. In particular, the strategy above cannot be used to identify $\rR\ko$ as an $\E{1}$-ring with $\HZ[t^4]$ directly. Actually, we can even prove that this is wrong. Indeed, this would imply that $\rR\ko$ and thus $(\rR\ko)[1/2] \simeq \kotop[1/2]$ are $\HZ$-modules. Since $\HZ$ is $K(1)$-acyclic (it is truncated), we deduce that $\kotop[1/2]\simeq \KOtop[1/2]_{\geq 0}$ is $K(1)$-acyclic (for $p=3$, say), and thus so is $\KOtop[1/2]$. The topological Wood cofiber sequence $\Sigma \KOtop \to \KOtop \to \KU$ (\cite{Wood}, see the version of \cite{Mathew-proof-of-Wood} cited in the bibliography for a modern proof) then induces a cofiber sequence
    $ \Sigma \KOtop[1/2] \Smash K(1) \to \KOtop[1/2] \Smash K(1) \to \KU[1/2] \Smash K(1)$
    whose first two terms vanish. Whence $0 \simeq \KU[1/2] \Smash K(1) \simeq (\KU \Smash \KU)/3$. The latter admits $\KU/3 \simeq K(1)$ as a retract, so $K(1) \simeq 0$, which is a contradiction.
\end{Rmk}
\vspace{0.2cm}

\subsubsection{Identifying the rationalization}\label{SubSubsect:rRkoQ}\hfill\vspace{0.2cm}

Using our knowledge of the localization at $(2)$ of $\rR\ko$, we can identify its rationalization.

\begin{Prop}\label{Prop:rRkoQ} There is an equivalence of $\E{1}$-rings $\HQ[t^4] \lsimeq{} (\rR\ko)_\Q$
sending $t^4\in\pi_4(\HQ[t^4])$ to (the image in the localization of) the generator $x \in \pi_4(\rR\ko)$.
\end{Prop}
\begin{proof} By \Cref{Prop:rRko(2)}, we have $(\rR\ko)_\Q \simeq (\rR\ko)_{(2)}[1/2] \cong \HZ_{(2)}[t^4][1/2]$. We claim that the latter is equivalent to the free $\E{1}$-$\HQ$-algebra $\HQ[t^4]$. Indeed, the unique ring map $\Z_{(2)} \to \Q$ (which induces an $\Einfty$-map $\HZ_{(2)} \to \HQ$ and thus a map of $\HZ_{(2)}$-modules $\HZ_{(2)} \to \HQ[t^4]$) and the element $t^4\in\pi_4(\HQ[t^4])$ induce a map of $\E{1}$-$\HZ_{(2)}$-algebras $\HZ_{(2)}[t^4] \to \HQ[t^4]$. Comparing the homotopy rings (\Cref{Prop:freeHA}$(ii)$), we see that after localization away from 2, this map induces isomorphisms in homotopy, as desired.
\end{proof}
\vspace{0.2cm}

\subsubsection{Main result: identifying the maps and assembling the fracture square}\label{SubSubsect:mapsinsquare}\hfill\vspace{0.2cm}

As the title indicates, we now combine the results of the previous subsection to identify the advertised explicit 2-local fracture square for $\rR\ko$. We first define the Chern character map; it will appear in our fracture square.

\begin{Def}\label{Def:Cherncharacter} Let $\HQ[t^4]$ and $\HQ[u^2]$ be the free $\E{1}$-$\HQ$-algebras on one generator in degrees 4 and 2 respectively. We call \emph{Chern character} any of the following:
\begin{itemize}
    \item the map $\mathsf{ku} := \mathsf{KU}_{\geq 0} \to \HQ[u^2]$ given by the (connective cover of the) classical Chern character as defined in \cite[\S 9.1]{BH-charclasses}. In homotopy, it induces the inclusion $\Z[u^2] \to \Q[u^2]$ where $u^2 \in \pi_2(\Z\times BU)\cong \Z$ is a generator.
    \item the map $\kotop \to \HQ[t^4]$ obtained by pre-composition of the previous map with the connective cover the complexification map $\KOtop \to \mathsf{KU}$ (induced by complexification of vector bundles) (their composition factors through the inclusion $\HQ[t^4] \xhookrightarrow{} \HQ[u^2] $ given by $t^4\mapsto (u^2)^2$). In homotopy, it induces the map $\Z[\alpha_1,\beta_4,\lambda_8]/ (\alpha_1^3, 2\alpha_1,\alpha_1\beta_4, \beta_4^2 - 4\lambda_8) \to \Z[u^2]$ sending $\alpha_1$ to 0, $\beta_4$ to $2(u^2)^2$, and thus $\lambda_8$ to $(u^2)^4$ (see for example \cite[\S 5.3]{Rognes}).
    \item the maps induced by the two previous ones after localization away from 2 of the source.
\end{itemize}
\end{Def}

\begin{Rmk}\label{Rmk:chQ} The Chern character $\mathsf{ch}:\kotop \to \HQ[t^4]$ factors through the rationalization, and the map $\mathsf{ch}_\Q: \kotop_\Q \to \HQ[t^4]$ obtained in this way is an equivalence. Indeed, $\pi_\ast(\kotop_\Q) \cong \pi_\ast(\kotop)\otimes_\Z \Q \cong \Q[\beta_4]$ and $\pi_\ast(\HQ[t^4]) \cong \Q[t^4]$ by \Cref{Prop:freeHA}$(ii)$, with $\mathsf{ch}_\Q$ sending $\beta_4$ to $2t^4$: it induces isomorphisms in homotopy. 
\end{Rmk}

\begin{thm}\label{Prop:rRko} There is a Cartesian square of $\mathcal{E}_1$-rings (and thus also of spectra)
    \[\begin{tikzcd}[row sep = 2em, column sep = 2em]
        {\rR\ko} & {\ko^\mathsf{top}[1/2]} \\
        {\HZ_{(2)}[t^4]} & {\mathsf{H}\Q[t^4].}
        \arrow["x\mapsto \beta_4/2", from=1-1, to=1-2]
        \arrow["x\mapsto t^4"', from=1-1, to=2-1]
        \arrow["\lrcorner"{anchor=center, pos=0.125}, draw=none, from=1-1, to=2-2]
        \arrow["{\mathsf{ch}}"', "\beta_4/2\mapsto t^4", from=1-2, to=2-2]
        \arrow["{t^4\mapsto t^4}"', from=2-1, to=2-2]
    \end{tikzcd}\]
    where $\HZ_{(2)}[t^4]$ is the free $\E{1}$-$\HZ_{(2)}$-algebra on a generator of degree 4 (\Cref{Def:freeE1HA}), and similarly for $\mathsf{H}\Q[t^4]$; and the map $\mathsf{ch}$ is the \emph{Chern character} from \Cref{Def:Cherncharacter}. The assignments labeling the arrows describe the maps induced on the fourth homotopy groups.
\end{thm}

\begin{proof}
    As mentioned in the introduction to this section, the 2-local fracture square of $\rR\ko$ reads as the pullback square of $\E{1}$-rings
    \[\begin{tikzcd}[row sep = 1.5em, column sep = 1.5em]
	{\rR\ko} & {(\rR\ko)[1/2]} \\
	{(\rR\ko)_{(2)}} & {(\rR\ko)_\Q.}
	\arrow[from=1-1, to=1-2]
	\arrow[from=1-1, to=2-1]
    \arrow["\lrcorner"{anchor=center, pos=0.125}, draw=none, from=1-1, to=2-2]
	\arrow[from=1-2, to=2-2]
	\arrow[from=2-1, to=2-2]
\end{tikzcd}\]
This holds in general for $\rR\ko$ replaced with any spectrum or $\E{n}$-ring for $1\leq n\leq \infty$, and 2 replaced by any prime $p$. 

Using the descriptions of the $\E{1}$-rings $(\rR\ko)[1/2]$, $(\rR\ko)_{(2)}$, and $(\rR\ko)_\Q$ in \Cref{Prop:rRko1/2,Prop:rRko(2),Prop:rRkoQ} respectively, we obtain a commutative diagram
% https://q.uiver.app/#q=WzAsOCxbMCwwLCJcXHJSXFxrbyJdLFsyLDAsIlxca29eXFxtYXRoc2Z7dG9wfVsxLzJdIl0sWzAsMiwiXFxIWl97KDIpfVt0XjRdIl0sWzAsMSwiKFxcclJcXGtvKV97KDIpfSJdLFsxLDAsIihcXHJSXFxrbylbMS8yXSJdLFsxLDEsIihcXHJSXFxrbylfXFxRIl0sWzIsMSwiKFxca29eXFxtYXRoc2Z7dG9wfVsxLzJdKV9cXFEiXSxbMywxLCJcXEhRW3TigbRdIl0sWzAsNF0sWzAsM10sWzMsNV0sWzQsNV0sWzIsMywiZiJdLFs0LDEsImciLDJdLFs1LDYsImdfXFxRIiwyXSxbMSw2XSxbNiw3LCJcXG1hdGhzZntjaH1fXFxRIiwyXSxbMSw3LCJcXG1hdGhzZntjaH0iLDFdLFsyLDcsImsiLDEseyJjdXJ2ZSI6Miwic3R5bGUiOnsiYm9keSI6eyJuYW1lIjoiZGFzaGVkIn19fV1d
\[\begin{tikzcd}
	{\rR\ko} & {(\rR\ko)[1/2]} & {\ko^\mathsf{top}[1/2]} \\
	{(\rR\ko)_{(2)}} & {(\rR\ko)_\Q} & {(\ko^\mathsf{top}[1/2])_\Q} & {\HQ[t^4]} \\
	{\HZ_{(2)}[t^4]}
	\arrow[from=1-1, to=1-2]
	\arrow[from=1-1, to=2-1]
	\arrow["g"', "\simeq", from=1-2, to=1-3]
	\arrow[from=1-2, to=2-2]
	\arrow[from=1-3, to=2-3]
	\arrow["{\mathsf{ch}}"{description}, from=1-3, to=2-4]
	\arrow[from=2-1, to=2-2]
	\arrow["{g_\Q}"',"\simeq", from=2-2, to=2-3]
	\arrow["{\mathsf{ch}_\Q}"',"\simeq", from=2-3, to=2-4]
	\arrow["f", "\simeq"', from=3-1, to=2-1]
	\arrow["\ell"{description}, curve={height=20pt}, dashed, from=3-1, to=2-4]
\end{tikzcd}\]
where $f$ is the equivalence of \Cref{Prop:rRko(2)}, $g$ is the equivalence of \Cref{Prop:rRko1/2}, $\mathsf{ch}_\Q$ is the map from \Cref{Rmk:chQ}, and $\ell$ is chosen so that the diagram commutes. In particular, the outer quadrilateral provides a pullback square
\[\begin{tikzcd}
        {\rR\ko} & {\ko^\mathsf{top}[1/2]} \\
        {\HZ_{(2)}[t^4]} & {\mathsf{H}\Q[t^4].}
        \arrow["x\mapsto g(x)", from=1-1, to=1-2]
        \arrow["x\mapsto t^4"', from=1-1, to=2-1]
        \arrow["\lrcorner"{anchor=center, pos=0.125}, draw=none, from=1-1, to=2-2]
        \arrow["{\mathsf{ch}}", from=1-2, to=2-2]
        \arrow["{t^4\mapsto \ell(t^4)}"', from=2-1, to=2-2]
    \end{tikzcd}\]
We thus have to compute $\ell(t^4) = \mathsf{ch}(g(x))$. Looking back at the proof of \Cref{Prop:rRko1/2}, $g$ is the inverse of the map $g'$ induced on the connective covers by the top row in the diagram
% https://q.uiver.app/#q=WzAsNixbMCwwLCJcXEtPdG9wWzEvMl0iXSxbMSwwLCJcXG1hdGhzZntMfShcXFIpWzEvMl0gXFxzaW1lcSBcXEdhbW1hKFxcUixcXEtXWzEvMl0pIl0sWzIsMCwiXFxHYW1tYShcXFIsXFxLV1sxL1xccmhvLDEvMl0pIFxcc2ltZXEgXFxyUlxcS09bMS8yXSJdLFsyLDEsIlxcR2FtbWEoXFxSLFxcS1dbMS9cXHJob10pIFxcc2ltZXEgXFxyUlxcS1ciXSxbMSwxLCJcXG1hdGhzZntMfShcXFIpIFxcc2ltZXEgXFxHYW1tYShcXFIsXFxLVykiXSxbMCwxLCJcXEtPdG9wIl0sWzAsMSwiXFx0YXVfXFxSIl0sWzEsMl0sWzMsMiwiXFx3ciIsMl0sWzQsM10sWzQsMV0sWzUsNCwiXFx0YXVfXFxSIl0sWzUsMF1d
\[\begin{tikzcd}[column sep = 3em]
	{\KOtop[1/2]} & {\mathsf{L}(\R)[1/2] \simeq \Gamma(\R,\KW[1/2])} & {\Gamma(\R,\KW[1/\rho,1/2]) \simeq \rR\KO[1/2]} \\
	\KOtop & {\mathsf{L}(\R) \simeq \Gamma(\R,\KW)} & {\Gamma(\R,\KW[1/\rho]) \simeq \rR\KW}
	\arrow["{\tau_\R}", from=1-1, to=1-2]
	\arrow[from=1-2, to=1-3]
	\arrow[from=2-1, to=1-1]
	\arrow["{\tau_\R}", from=2-1, to=2-2]
	\arrow[from=2-2, to=1-2]
	\arrow[from=2-2, to=2-3]
	\arrow["\simeq"', from=2-3, to=1-3]
\end{tikzcd}\]
where all maps in the square on the right-hand side are induced by the localizations (away from 2 or $\rho$), and the map $\tau_\R$ is defined in \cite[Thm.\ A and Ex.\ 9.2]{LNS}, where it is also proven that it sends $\beta_4 \in \pi_4(\KOtop)$ to $8b\in\pi_4(\mathsf{L}(\R)) \simeq \Z$. We claim that $b=\pm\eta^{-4}\beta_\KO \in \pi_4(\Gamma(\R,\KW)) \cong \pi_4(\mathsf{L}(\R))$. Indeed, $\KW = \KO[\eta^{-1}]$ is $\beta_\KO$- and $\eta$-periodic, and thus multiplication by $\eta^{-4}\beta_\KO$ is an equivalence $\Sigma^4\KW \lsimeq{} \KW$, and induces an isomorphism 
$$\Z \cong \pi_0(\Gamma(\R,\KW)) \cong [\calS^4,\Sigma^4\KW] \xrightarrow{\ \ \cong\ \ } [\calS^4,\KW] \cong \pi_4(\Gamma(\R,\KW)) \cong \Z\{b\}.$$
So $\eta^{-4}\beta_\KO = \pm b$. Using the fact that $\rR(\eta) = -2$ (see the proof of \Cref{Prop:rRKO}), it follows that $\rR(b) = \pm\rR(\beta_\KO)/16 \in \pi_4(\Gamma(\R,\KW[1/\rho]))$. We defer to \Cref{Prop:rRbetako} the proof that $\rR(\beta_\ko) = \pm 16x$. It follows that $g'(\beta_4) = \pm\rR(8b) = \pm 8x$, and thus $g(x) = \pm\beta_4/8$. Then $\ell(x) = \mathsf{ch}(g(x)) = \pm t^4/4$ (see \Cref{Rmk:chQ}).

In order to be able to compare our fracture square with the one in \cite[p.\ 3]{HLN} (\Cref{Prop:compwithLtheory}), let us replace $g: \rR\ko[1/2] \to \kotop[1/2]$ by $g'':= \psi^2 \circ g$, where $\psi^2: \kotop[1/2] \to \kotop[1/2]$ is the second Adams operation. By \cite[Section IV, 7.13, 7.19 and 7.25]{Karoubi-Ktheory}, the latter is an $\Einfty$-map, and it induces multiplication by 4 on the fourth homotopy groups. Then, repeating the same proof, we obtain ${g''(x) = \psi^2(\pm\beta_4/8) = \pm\beta_4/2}$ and $\ell(t^4) = \pm t^4$. We can further replace $g''$ by $-g''$ if needed to obtain $g''(x) = \beta_4/2$ and $\ell(t^4) = t^4$. Thus, the lower horizontal map $\ell$ in our pullback square is the map of $\E{1}$-$\HZ_{(2)}$-algebras $\HZ_{(2)}[t^4] \to \HQ[t^4]$ sending $t^4$ to $t^4$ in the fourth homotopy groups, as desired.
\end{proof}

\begin{Prop}\label{Prop:rRbetako} In the notation of \Cref{Prop:homgprRko}, we have $\rR(\beta_\ko) = \pm 16x \in \pi_4(\rR\ko)$.
\end{Prop}
\begin{proof}
    The proof is divided in several lemmas proven in \Cref{Subsubsect:usedinproof}. We explain here how they combine to prove the proposition.\\

    \textbf{Step 1:} \emph{Reducing the problem to computing the cardinality of a certain homotopy group.} We want to understand the action of $\rR\beta_\ko$ on the homotopy groups of $\rR\ko$, by considering a cofiber sequence where this map appears together with only a small number of easily computable groups. We will thus consider slices. By \Cref{Prop:aiscardinality} below, if we write $\rR(\beta_\ko) = ax$ with $a\in\Z$ (recall that $x$ generates $\pi_4(\rR\ko) \cong \Z$), then $a=\pm|\pi_4(\rR C)|$, where $C$ is the cofiber of the natural map $\tilde{f}_4\ko \to \ko$. This result holds because, as we will see in its proof, the natural map $\tilde{f}_4\ko \to \ko$ corresponds to $\beta_\ko$ under the identification $\tilde{f}_4\ko \simeq \tilde{f}_4\KO \simeq \tilde{f}_4(\Sigma^{8,4}\KO) \simeq T^{\Smash 4} \Smash \tilde{f}_0\KO \simeq T^{\Smash 4} \Smash \ko$, where the first equivalence is induced by $\beta_\KO$ (periodicity of $\KO$) and the second one is \cite[Lem.\ 8]{Bachmann-genslices}.\\
    
    \textbf{Step 2:} \emph{Setting up a spectral sequence that computes $|\pi_4(\rR C)|$.}  Under sufficiently nice assumptions, any (co)filtered object in a stable category has an associated strongly convergent spectral sequence computing the values of a fixed homological functor on this object. We recall the general theory in \Cref{Subsect:multss}. We will consider the very effective filtration on $C = \ko/\tilde{f}_4\ko$, and after applying real realization, this will give us a filtration on $\rR C$. The $E^1$-page of such a spectral sequence involves the homotopy groups of the successive cofibers of the maps in the filtered object. In our case, these will be the homotopy groups of the real realizations of the very effective slices of $C$, which are the first four very effective slices of $\ko$. These slices have already been computed in \cite[Thm.\ 16]{Bachmann-genslices} (see \Cref{Prop:veffslicesko}).\\
            
    \textbf{Step 3:} \emph{Making explicit computations with this spectral sequence.} We computed the realizations of all the motivic spectra appearing as slices of $\ko$ in \Cref{Sect:examplesrealization}, except for $\tilde{s}_0\ko$. The only grasp we have on the latter is given by its decomposition as a cofiber sequence in \Cref{Prop:veffslicesko}. To compute the homotopy groups of $\rR(\tilde{s}_0\ko)$, we will use another spectral sequence, this time with respect to the filtration given by the effective homotopy $t$-structure. This is the content of \Cref{Prop:pistilde0ko}. This allows us to prove $|\pi_4(\rR C)|=16$ in \Cref{Prop:pi4C}, using the first spectral sequence we mentioned.
\end{proof}
\vspace{0.2cm}

\subsubsection{Digression: spectral sequences with a multiplicative structure}\label{Subsect:multss}\hfill\vspace{0.2cm}

The proofs of \Cref{Prop:pistilde0ko,Prop:pi4C} needed for \Cref{Prop:rRbetako} will require us to use spectral sequences arguments. We therefore recall the necessary machinery.

\begin{thm}\label{Prop:spectralsequence} Let $X := \cdots \to X_{-1} \to X_0 \to X_1 \to \cdots$ be a tower in $\Sp$. Then, there is a spectral sequence $E^1_{p,q} = \pi_p(\mathsf{cof}(X_{-q-1} \to X_{-q}))$ converging conditionally to $\pi_p(\mathsf{cof}(\lim_n X_n \to \colim_n X_n))$. Moreover, if the spectral sequence is concentrated in the upper half-plane and each term is the source of only finitely many non-trivial differentials, then it converges strongly. Namely, for all $p,q\in\Z$, we have $E^\infty_{p,q} = \mathsf{coker}(\pi_p(X_{-q}) \to \pi_p(X_{-q+1}))$ and $\pi_p(\colim_n X_n) = \colim_q\, \mathsf{coker}(\pi_p(X_{-q}) \to \pi_p(\colim_n X_n))$.

We call the cofibers $\mathsf{cof}(X_{-q-1} \to X_{-q})$ the \emph{subquotients} associated with the cofiltered object $(X_n)_{n\in\Z}$.
\end{thm}
\begin{proof} Such a spectral sequence for cofiltered objects in a general stable category $\calC$ is defined in \cite[Construction 1.2.2.6]{Lurie-HA}. The conditional convergence is proven in \cite[Lem.\ 6.16]{BH-eta}. Then, the statement about strong convergence is \cite[Remark below Thm.\ 7.1]{Boardman}.
\end{proof}

If the tower of spectra we started with admits some kind of multiplicative structure, the spectral sequence also inherits a multiplicative structure, which can help to determine certain differentials in such spectral sequences. More precisely, assume that $(X_n)_{n\in\Z}$ admits a pairing taking the form of compatible maps $X_n \Smash X_m \to X_{n+m}$ for all $n,m\in \Z$. This descends to a pairing on the subquotients, given by the maps 
$\gamma_{n,m}$ induced between the cofibers of $\alpha$ and the zero map in the bottom face of the following diagram (where $S_n := \mathsf{cof}(X_{n-1}\to X_n)$ for all $n\in\Z$):

%https://q.uiver.app/#q=WzAsMTQsWzQsNiwiU19tIFxcd2VkZ2UgU19uIl0sWzIsNSwiWF9tIFxcd2VkZ2UgU19uIl0sWzAsNCwiWF97bS0xfSBcXHdlZGdlIFNfbiJdLFs4LDYsIlNfe20rbn0iXSxbNiw1LCJTX3ttK259Il0sWzQsNCwiU197bStuLTF9Il0sWzQsMiwiWF97bStuLTF9Il0sWzQsMCwiWF97bStuLTJ9Il0sWzYsMSwiWF97bStuLTF9Il0sWzYsMywiWF97bStufSJdLFsyLDMsIlhfbSBcXHdlZGdlIFhfbiJdLFswLDIsIlhfe20tMX0gXFx3ZWRnZSBYX24iXSxbMCwwLCJYX3ttLTF9IFxcd2VkZ2UgWF97bi0xfSJdLFsyLDEsIlhfe219IFxcd2VkZ2UgWF97bi0xfSJdLFsyLDUsIlxcbGFtYmRhX3ttLTEsbn0iLDAseyJsYWJlbF9wb3NpdGlvbiI6NzAsImNvbG91ciI6WzI0MCw2MCw2MF0sInN0eWxlIjp7ImJvZHkiOnsibmFtZSI6ImRhc2hlZCJ9fX0sWzI0MCw2MCw2MCwxXV0sWzEsNCwiXFxsYW1iZGFfe20sbn0iLDAseyJjb2xvdXIiOlsyNDAsNjAsNjBdfSxbMjQwLDYwLDYwLDFdXSxbMCwzLCJcXGdhbW1hX3ttLG59Il0sWzIsMSwiXFxhbHBoYSIsMSx7ImNvbG91ciI6WzI0MCw2MCw2MF19LFsyNDAsNjAsNjAsMV1dLFs1LDQsIjAiLDEseyJjb2xvdXIiOlsyNDAsNjAsNjBdLCJzdHlsZSI6eyJib2R5Ijp7Im5hbWUiOiJkYXNoZWQifX19LFsyNDAsNjAsNjAsMV1dLFs2LDksIiIsMSx7InN0eWxlIjp7ImJvZHkiOnsibmFtZSI6ImRhc2hlZCJ9fX1dLFs3LDhdLFs4LDldLFs5LDRdLFs3LDYsIiIsMSx7InN0eWxlIjp7ImJvZHkiOnsibmFtZSI6ImRhc2hlZCJ9fX1dLFs2LDUsIiIsMSx7InN0eWxlIjp7ImJvZHkiOnsibmFtZSI6ImRhc2hlZCJ9fX1dLFsxMiw3XSxbMTIsMTNdLFsxMiwxMV0sWzExLDIsIlxcYWxwaGEnIiwxXSxbMTAsMV0sWzExLDEwXSxbMTMsMTBdLFsxMSw2LCIiLDEseyJzdHlsZSI6eyJib2R5Ijp7Im5hbWUiOiJkYXNoZWQifX19XSxbMTAsOV0sWzEzLDhdLFsxLDBdLFs0LDNdXQ==
\begin{equation}\label{Diagram:pairing}
    \begin{tikzcd}[column sep = 0.8em, row sep = 0.6em]
	{X_{m-1} \wedge X_{n-1}} &&&& {X_{m+n-2}} \\
	&& {X_{m} \wedge X_{n-1}} &&&& {X_{m+n-1}} \\
	{X_{m-1} \wedge X_n} &&&& {X_{m+n-1}} \\
	&& {X_m \wedge X_n} &&&& {X_{m+n}} \\
	{X_{m-1} \wedge S_n} &&&& {S_{m+n-1}} \\
	&& {X_m \wedge S_n} &&&& {S_{m+n}} \\
	&&&& {S_m \wedge S_n} &&&& {S_{m+n}.}
	\arrow[from=1-1, to=1-5]
	\arrow[from=1-1, to=2-3]
	\arrow[from=1-1, to=3-1]
	\arrow[from=1-5, to=2-7]
	\arrow[dashed, from=1-5, to=3-5]
	\arrow[from=2-3, to=2-7]
	\arrow[from=2-3, to=4-3]
	\arrow[from=2-7, to=4-7]
	\arrow[dashed, from=3-1, to=3-5]
	\arrow[from=3-1, to=4-3]
	\arrow["{\alpha'}"{description}, from=3-1, to=5-1]
	\arrow[dashed, from=3-5, to=4-7]
	\arrow[dashed, from=3-5, to=5-5]
	\arrow[from=4-3, to=4-7]
	\arrow[from=4-3, to=6-3]
	\arrow[from=4-7, to=6-7]
	\arrow[MyGreen, "{\lambda_{m-1,n}}"{pos=0.7}, dashed, from=5-1, to=5-5]
	\arrow[MyGreen, "\alpha"{description}, from=5-1, to=6-3]
	\arrow[MyGreen, "0"{description}, dashed, from=5-5, to=6-7]
	\arrow[MyGreen, "{\lambda_{m,n}}", from=6-3, to=6-7]
	\arrow[from=6-3, to=7-5]
	\arrow[from=6-7, to=7-9]
	\arrow["{\gamma_{m,n}}", from=7-5, to=7-9]
\end{tikzcd}
\end{equation}
All vertical composites are cofiber sequences; the maps $\alpha$, $\lambda_{m-1,n}$, $\lambda_{m,n}$, and $0$ in the bottom face are the maps induced on the cofibers by the top two layers of horizontal maps.

\begin{Prop}[\protect{\cite[\S 6, Thm.\ 6.2]{Dugger}}]\label{Prop:multss} In the situation described above, the pairing on the subquotients itself descends to a pairing $E^r_{p,q} \otimes E^r_{s,t} \to E^r_{p+s,q+t}$ (for all $r\geq 1$ and $p,q,s,t\in\Z$) on the spectral sequence; and the differentials $d^r$ satisfy the (graded) Leibniz rule with respect to this pairing.
\end{Prop}
\vspace{0.2cm}

\subsubsection{Results used in the proof of \Cref{Prop:rRbetako}}\label{Subsubsect:usedinproof}\hfill\vspace{0.2cm}

\begin{Lemma}\label{Prop:aiscardinality} Let $a\in\Z$ with $\rR(\beta_\ko) = ax$ (where $x$ generates $\pi_4(\rR\ko)\cong\Z$). Then $a=\pm|\pi_4(\rR C)|$, where $C$ is the cofiber of the natural map $\tilde{f}_4\ko \to \ko$.
\end{Lemma}
\begin{proof}
    The cofiber sequence $\tilde{f}_4\ko \to \ko \to C$ induces a long exact sequence
% https://q.uiver.app/#q=WzAsNixbMSwwLCJcXHBpXzQoXFxyUlxcdGlsZGV7Zn1fNFxca28pIl0sWzIsMCwiXFxwaV80KFxcclJcXGtvKSJdLFszLDAsIlxccGlfNChcXHJSQykiXSxbNCwwLCJcXHBpXzMoXFxyUlxcdGlsZGV7Zn1fNFxca28pIl0sWzUsMCwiXFxkb3RzIl0sWzAsMCwiXFxkb3RzIl0sWzUsMF0sWzAsMV0sWzEsMl0sWzIsM10sWzMsNF1d
\[\begin{tikzcd}
	\cdots & {\pi_4(\rR\tilde{f}_4\ko)} & {\pi_4(\rR\ko)} & {\pi_4(\rR C)} & {\pi_3(\rR\tilde{f}_4\ko)} & \cdots.
	\arrow[from=1-1, to=1-2]
	\arrow[from=1-2, to=1-3]
	\arrow[from=1-3, to=1-4]
	\arrow[from=1-4, to=1-5]
	\arrow[from=1-5, to=1-6]
\end{tikzcd}\]
Since by definition $\Sigma^{8,4}\ko \in \SH(k)^\veff(4)$, there is a lift
% https://q.uiver.app/#q=WzAsMyxbMSwwLCJcXHRpbGRle2Z9XzQoXFxrbykiXSxbMSwxLCJcXGtvIl0sWzAsMSwiXFxTaWdtYV57OCw0fVxca28iXSxbMCwxXSxbMiwxLCJcXGJldGFfXFxrbyIsMl0sWzIsMCwiXFxhbHBoYSIsMCx7InN0eWxlIjp7ImJvZHkiOnsibmFtZSI6ImRhc2hlZCJ9fX1dXQ==
\[\begin{tikzcd}
	& {\tilde{f}_4\ko} \\
	{\Sigma^{8,4}\ko} & \ko
	\arrow[from=1-2, to=2-2]
	\arrow["\alpha", dashed, from=2-1, to=1-2]
	\arrow["{\beta_\ko}"', from=2-1, to=2-2]
\end{tikzcd}\]
where $\tilde{f}_4\ko \to \ko$ is the natural map. We claim that $\alpha$ is an equivalence. In particular, by \Cref{Prop:homgprRko}, applying real realization and taking the fourth homotopy groups, the map ${\pi_4(\rR\tilde{f}_4\ko)} \to {\pi_4(\rR\ko)}$ is a morphism $\Z\to \Z$ given by multiplication by $a$. Since $\rR\tilde{f}_4\ko$ is 4-connective by \Cref{Prop:rRconnective}, in the long exact sequence above we have ${\pi_3(\rR\tilde{f}_4\ko)}\cong 0$ and thus $\pi_4(\rR C) \cong \Z/a$. Therefore, to finish the proof we only have to show that $\alpha$ is an equivalence. In the following diagram
% https://q.uiver.app/#q=WzAsNyxbMSwxLCJcXFNpZ21hXns4LDR9XFxLTyJdLFswLDEsIlxcU2lnbWFeezgsNH1cXGtvIl0sWzIsMSwiXFxLTyJdLFsyLDAsIlxcdGlsZGV7Zn1fNChcXEtPKSJdLFsxLDAsIlxcdGlsZGV7Zn1fNChcXFNpZ21hXns4LDR9XFxLTykiXSxbMywxLCJcXGtvIl0sWzMsMCwiXFx0aWxkZXtmfV80KFxca28pIl0sWzEsMF0sWzAsMiwiXFxiZXRhX1xcS08iLDJdLFszLDJdLFs0LDMsIlxcdGlsZGV7Zn1fNChcXGJldGFfXFxLTykiLDJdLFsxLDQsIlxcYWxwaGEnIiwwLHsic3R5bGUiOnsiYm9keSI6eyJuYW1lIjoiZGFzaGVkIn19fV0sWzQsMF0sWzMsNiwiXFxzaW0iXSxbNiw1XSxbNSwyXSxbMSw2LCJcXGFscGhhIiwxLHsic3R5bGUiOnsiYm9keSI6eyJuYW1lIjoiZGFzaGVkIn19fV0sWzEsNSwiXFxiZXRhX1xca28iLDIseyJjdXJ2ZSI6M31dXQ==
\[\begin{tikzcd}[row sep = 3em, column sep = 3em]
	& {\tilde{f}_4(\Sigma^{8,4}\KO)} & {\tilde{f}_4(\KO)} & {\tilde{f}_4(\ko)} \\
	{\Sigma^{8,4}\ko} & {\Sigma^{8,4}\KO} & \KO & \ko
	\arrow["{\tilde{f}_4(\beta_\KO)}"',"\sim", from=1-2, to=1-3]
	\arrow[from=1-2, to=2-2]
	\arrow["\sim", from=1-3, to=1-4]
	\arrow[from=1-3, to=2-3]
	\arrow[from=1-4, to=2-4]
	\arrow["{\alpha'}", dashed, from=2-1, to=1-2]
	\arrow["\alpha"{description}, dashed, from=2-1, to=1-4]
	\arrow[from=2-1, to=2-2]
	\arrow["{\beta_\ko}"', curve={height=18pt}, from=2-1, to=2-4]
	\arrow["{\beta_\KO}"',"\sim", from=2-2, to=2-3]
	\arrow[from=2-4, to=2-3]
\end{tikzcd}\]
the map $\alpha$ factors through the lift $\alpha'$ also obtained by very 4-effectiveness of $\Sigma^{8,4}\ko$. But $\alpha'$ is exactly the equivalence $T^{\Smash 4} \Smash \tilde{f}_0 \KO \simeq \tilde{f}_4(T^{\Smash 4} \Smash \KO)$ of \cite[Lem.\ 8]{Bachmann-genslices}. From the commutativity of the diagram, we deduce that $\alpha'$ is a composition of three equivalences, and thus an equivalence, as needed. 
\end{proof}

\begin{Lemma}\label{Prop:pistilde0ko} The homotopy groups of $\rR(\tilde{s}_0\ko)$ are given for $m\geq 0$ by
$$\pi_m(\rR\tilde{s}_0\ko) = \begin{cases} \Z &\text{ if }m=0 \\ \Zmod &\text{ if }m\equiv 2\, \mathsf{mod}\, 4 \text{ or }m\equiv 3\, \mathsf{mod}\, 4 \\
0 &\text{ if } m\equiv 1\, \mathsf{mod}\, 4 \\
\end{cases}$$
and in positive degrees divisible by 4, $\pi_m(\rR\tilde{s}_0\ko)$ is an extension of two copies of $\Zmod$.
\end{Lemma}
\begin{Rmk}
A posteriori, using the proof of \Cref{Prop:pi4C}, we get that $\pi_4(\rR(\tilde{s}_0\ko))$ is at the same time a quotient of a subgroup of $\pi_4(\rR(C)) \cong \Z/16$ and an extension of two copies of $\Z/2$. Therefore, we have $\pi_4(\rR(\tilde{s}_0\ko)) \simeq \Z/4$.
\end{Rmk}
\begin{proof}
    We apply \Cref{Prop:spectralsequence} with $X_n =\rR((\tilde{s}_0\ko)_{\geq_e -n})$ for all $n\in\Z$. We will see in Step 1 below that the spectral sequence obtained satisfies the additional assumptions for the strong convergence. The colimit $\colim_n X_n$ is eventually constant with value $\rR(\tilde{s}_0\ko)$, whereas the limit $\lim_n X_n$ is trivial because $X_n$ is $(-n)$-connective for all $n\leq 0$. Indeed, $$\pi_r(X_n) = \pi_r(\rR((\tilde{s}_0\ko)_{\geq_e -n})) = \pi_r(\Sigma^{-n}\rR((\Sigma^n\tilde{s}_0\ko)_{\geq_e 0})) = 0$$ for all $r<-n$ (because $(\Sigma^r\tilde{s}_0\ko)_{\geq_e 0} \in \SH(\R)^\eff_{\geq 0} = \SH(k)^\veff$ by \Cref{Prop:propertiesofSHkeff}, and very effective spectra real realize to connective ones by \Cref{Prop:rRconnective}). Since the limit of a tower of $k$-connective spectra is $(k-1)$-connective \cite[Prop.\ VI.2.15]{GJ}, $\lim_n X_n$ is $\infty$-connective and thus trivial. Thus, the spectral sequence converges strongly to $\pi_p(\rR(\tilde{s}_0\ko))$.\\
    
    \textbf{Step 1:} \emph{In the $E^1$-page, only the zeroth and first line have non-trivial terms, i.e., only the zeroth and first subquotients are non-trivial.} Indeed, we have by definition $\tilde{s}_0\ko \in \SHR^\veff = \SHR^\eff_{\geq 0}$ (by \Cref{Prop:propertiesofSHkeff}$(i)$), so that $(\tilde{s}_0\ko)_{\geq_e 0} = \tilde{s}_0\ko$. By Claim (1) in the proof of \cite[Lem.\ 11]{Bachmann-genslices}, we have $(\tilde{s}_0\ko)_{\geq_e 1} \simeq s_0(\ko_{\geq 1})$, and the latter is the first term in the cofiber sequence of \Cref{Prop:veffslicesko}, namely $\Sigma^{1,0}\HZmod$. The natural map $(\tilde{s}_0\ko)_{\geq_e 1} \to \tilde{s}_0\ko$ is exactly the map $\Sigma^{1,0}\HZmod \to \tilde{s}_0$ appearing in this cofiber sequence. In particular, the $0$-th subquotient is nothing but the third term of this cofiber sequence, $\HZtilde$. Let us denote $s'_1\ko := (\tilde{s}_0\ko)_{\geq_e 1} \simeq \Sigma^{1,0}\HZmod$ and $s'_0\ko := (\tilde{s}_0\ko)/((\tilde{s}_0\ko)_{\geq_e 1}) \simeq \HZtilde$. The first subquotient is actually $s'_1\ko$ since $(\tilde{s}_0\ko)_{\geq_e k} \simeq 0$ for all $k\geq 2$. Indeed, in the cofiber sequence
$$ (\Sigma^{1,0}\HZmod)_{\geq_e k} \to (\tilde{s}_0\ko)_{\geq_e k} \to (\HZtilde)_{\geq_e k},$$
the two outer terms are zero, because by \Cref{Prop:HZHZmodveryeff} we have $\HZmod, \HZtilde \in \SHR^{\eff, \heartsuit}$, in particular $(\HZtilde)_{\geq_e k} \simeq 0$ for all $k\geq 1$ and $(\Sigma^{1,0}\HZmod)_{\geq_e k} \simeq \Sigma^{1,0}(\HZmod_{\geq_e k-1}) \simeq 0$ for all $k\geq 2$.\\

\textbf{Step 2:} \emph{We write down the $E^1$-page.} By \Cref{Prop:rRHZHZmod,Prop:rRHZtilde}, we know the homotopy groups of $\rR(s'_0\ko)$ and $\rR(s'_1\ko)$. Therefore, the $E^1$-page of our spectral sequence for $\pi_\ast(\rR\tilde{s}_0\ko)$ reads as follows:

\begin{figure}[h]
\centering
\begin{tikzpicture}
  \matrix (m) [matrix of math nodes,
    nodes in empty cells,nodes={minimum width=5ex,
    minimum height=5ex,outer sep=-5pt},
    column sep=3.5ex,row sep=3ex]{
          {\color{gray} 1}     &  0 &  \Zmod  & \Zmod & \Zmod & \Zmod &  \Zmod  & \Zmod & \Zmod & \Zmod \\
          {\color{gray} 0}    &  \Z  & 0 &  \Zmod  & 0& \Zmod  & 0 &  \Zmod  & 0 & \Zmod \\
    \quad\strut &  {\color{gray} 0}  &  {\color{gray} 1}  &  {\color{gray} 2}  & {\color{gray} 3} & {\color{gray} 4}  &  {\color{gray} 5}  &  {\color{gray} 6}  & {\color{gray} 7} & {\color{gray} 8} \\};
    \draw [draw= BrickRed, thick, -stealth] (m-2-4.north west) -- (m-1-3.south east) node[midway,inner sep = 1pt,above right] {\footnotesize \color{BrickRed} $g_2$} node[midway,inner sep = 1pt,below left] {\footnotesize \color{BrickRed} $\cong$};
    \draw [draw= BrickRed, thick, dashed, -stealth] (m-2-6.north west) -- (m-1-5.south east) node[midway,inner sep = 1pt,above right] {\footnotesize \color{BrickRed} $g_4$};
    \draw [draw= BrickRed, thick, -stealth] (m-2-8.north west) -- (m-1-7.south east) node[midway,inner sep = 1pt,above right] {\footnotesize \color{BrickRed} $g_6$} node[midway,inner sep = 1pt,below left] {\footnotesize \color{BrickRed} $\cong$};
    \draw [draw= BrickRed, thick, dashed, -stealth] (m-2-10.north west) -- (m-1-9.south east) node[midway,inner sep = 1pt,above right] {\footnotesize \color{BrickRed} $g_8$};
    \draw ([yshift=-5pt]m-1-3.south) node{\footnotesize \color{Cerulean} $s$};
    \draw ([yshift=-5pt]m-1-4.south) node{\footnotesize \color{Cerulean} $s'$};
    \draw ([yshift=-5pt]m-1-5.south) node{\footnotesize \color{Cerulean} $t\smash\cdot\smash s$};
    \draw ([yshift=-5pt]m-1-7.south) node{\footnotesize \color{Cerulean} $t^2\smash\cdot \smash s$};
    \draw ([yshift=-5pt]m-1-9.south) node{\footnotesize \color{Cerulean} $t^3\smash\cdot \smash s$};
    \draw ([yshift=-5pt]m-2-4.south) node{\footnotesize \color{Cerulean} $t$};
    \draw ([yshift=-5pt]m-2-6.south) node{\footnotesize \color{Cerulean} $t^2$};
    \draw ([yshift=-5pt]m-2-8.south) node{\footnotesize \color{Cerulean} $t^3$};
    \draw ([yshift=-5pt]m-2-10.south) node{\footnotesize \color{Cerulean} $t^4$};
\draw[thick, -stealth] ([xshift=5pt]m-3-1.east) -- ([xshift=5pt,yshift=20pt]m-1-1.east) node [left] {$q$} ;
\draw[thick, -stealth] ([yshift=5pt]m-3-1.north) -- ([xshift=20pt,yshift=5pt]m-3-10.north) node [below] {$p$} ;
\end{tikzpicture}
\caption{\footnotesize $E^1$-page for the spectral sequence for $\rR\tilde{s}_0\ko$ with respect to the effective homotopy $t$-structure.}\label{Fig:smallsseq}
\end{figure}
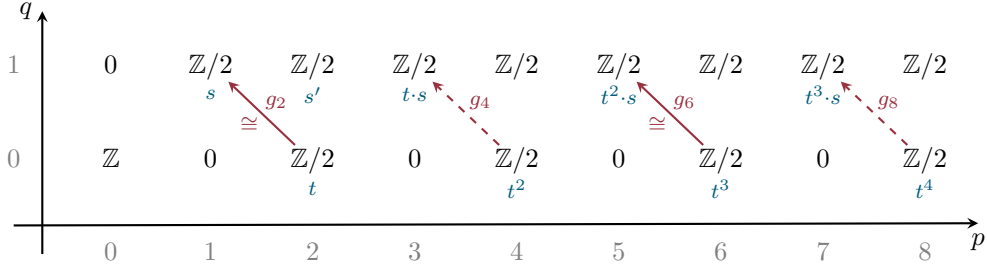

The arrows are some $d^1$ differentials, we will explain how to determine them just below. Dotted arrows mean that the differential is trivial. The elements written under some groups in the $E^1$-page are generators of these groups; we will see in Step 4 why it is indeed the case.\\

\textbf{Step 3:} \emph{We identify the differential $g_2$.} To begin with, note that we must have $\pi_1(\rR\tilde{s}_0\ko) \cong 0$. Indeed, this group is the $E^1_{1,0}$-term in the spectral sequence given by \Cref{Prop:spectralsequence} for $X'_n = \rR(\tilde{f}_{-n}\ko)$, where we set $\tilde{f}_{n}\ko = \ko$ for $n<0$. The assumptions of the theorem are still verified in this case by the same arguments as above. We obtain a spectral sequence concentrated in the first quadrant with $E^1_{p,q} = \pi_p(\tilde{s}_q\ko)$ for all $p,q\in\N$, converging strongly to $\pi_p(\rR\ko)$. In particular, the $E^1_{1,0}$-term will never receive a non-trivial differential. Moreover, since the first column of this same spectral contains only zeroes except in position $(0,0)$ (we have $\pi_0(\rR\tilde{s}_i\ko) \cong 0$ for all $i\geq 1$), it will never be the source of a non-trivial differential either, and will survive to the $E^\infty$-page. But since $\pi_1(\rR\ko) \cong 0$ by \Cref{Prop:homgprRko}, nothing in the first column can survive. This means that $\pi_1(\rR\tilde{s}_0\ko) \cong 0$. In particular, in the spectral sequence for $\rR(\tilde{s}_0\ko)$ and the effective homotopy t-structure, $E^1_{1,1} \cong \Zmod$ must eventually die. Since it cannot be the source of any non-trivial differential, the only possibility is that it receives a non-trivial differential $g_2$ from $E^1_{2,0} \cong \Zmod$, which must then be an isomorphism.\\

\textbf{Step 4:} \emph{We use the multiplicative structure on this spectral sequence to compute the next differentials.} In the situation of \Cref{Prop:multss}, let $E = \tilde{s}_0\ko$ and $X_n := \rR(E_{\geq_e -n})$ for all $n\in\Z$. The pairings $X_{-n} \Smash X_{-m} \to X_{-n-m}$ are given by the real realizations of the compositions $ E_{\geq_e n} \Smash E_{\geq_e m} \to E\Smash E \to E$ of the natural maps with the multiplication on $E$ (the latter is itself induced by the pairing in the tower for the very effective slice filtration on $\ko$, as we will see below), which lifts to $E_{\geq_e n+m}$. Indeed, we have
$$\SHR^\eff_{\geq n} \Smash \SHR^\eff_{\geq m} = \Sigma^{n,0}\SHR_{\geq 0}^\eff \Smash \Sigma^{m,0}\SHR_{\geq 0}^\eff \subseteq \Sigma^{m+n,0}\SHR_{\geq 0}^\eff = \SHR^\eff_{\geq m+n}$$ because the effective homotopy t-structure is compatible with the symmetric monoidal structure (in the sense of \cite[Def. A.10]{AN}). In particular, we have pairings 
\begin{align*}\HZtilde \Smash \HZtilde &\simeq s'_0\ko \Smash s'_0\ko \to s'_0\ko = \HZtilde,\\
\HZtilde \Smash \Sigma^{1,0}\HZmod &\simeq s'_0\ko \Smash s'_1\ko \to s'_1\ko = \Sigma^{1,0}\HZmod.
\end{align*}
We defer to \Cref{Prop:pairings} below the proof that these are exactly the pairings induced, respectively, by the ring structure on $\HZtilde$ and the composition $\HZtilde \Smash \HZmod \to \HZ \Smash \HZmod \to \HZmod$ induced by the natural map from Milnor--Witt K-theory to Milnor K-theory and the $\HZ$-algebra structure on $\HZmod$. 

Then, the pairing $\rR(s'_0\ko) \Smash \rR(s'_0\ko) \to \rR(s'_0\ko)$ on the subquotients is given in homotopy by the ring structure of $\pi_\ast(\rR\HZtilde) = \Z[t^2]/(2t^2)$ (\Cref{Prop:rRHZtilde}), with $\HZtilde \to \HZmod$ killing $2$. In a similar way, the pairing $\rR(s'_0\ko) \Smash \rR(s'_1\ko) \to \rR(s'_1\ko)$ is given in homotopy, after desuspension, by 
$\Z[t^2]/(2t^2) \otimes \Z/2[t] \to \Z/2[t^2] \otimes \Z/2[t] \to \Z/2[t]$,
 where the first map is the quotient and the second one is multiplication (viewing $\Z/2[t^2]\subseteq \Z/2[t]$ by \Cref{Prop:rRHZHZmod}). This explains why the generators are the ones displayed in the spectral sequence in Figure \ref{Fig:smallsseq}. Since $g_2$ is an isomorphism, we have $d^1(t)=s$. It follows that $g_{2k}$ is zero for $k$ even and an isomorphism for $k$ odd. Indeed, by the Leibniz rule $g_{2k}(t^k) = d^1(t^k)=kt^{k-1}d^1(t) \in \Zmod$ is 0 if $k$ is even and equals $t^{k-1}\cdot s$ otherwise (i.e., a generator in the target). \\

\textbf{Step 5:} \emph{We conclude.} The previous steps show that the $E^1$-page of our spectral sequence has the form displayed in Figure \ref{Fig:smallsseq}. Therefore, the $E^2$-page reads as follows:
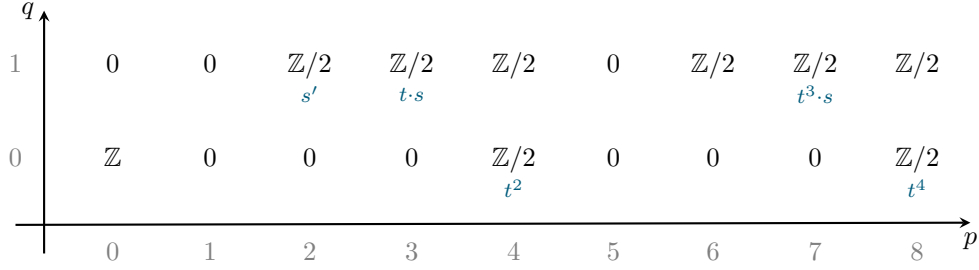
\[ \begin{tikzpicture}
      \matrix (m) [matrix of math nodes,
        nodes in empty cells,nodes={minimum width=5ex,
        minimum height=5ex,outer sep=-5pt},
        column sep=3.5ex,row sep=3ex]{
              {\color{gray} 1}     &  0 &  0  & \Zmod & \Zmod & \Zmod &  0  & \Zmod & \Zmod & \Zmod \\
              {\color{gray} 0}    &  \Z  & 0 &  0  & 0& \Zmod  & 0 &  0  & 0 & \Zmod \\
        \quad\strut &  {\color{gray} 0}  &  {\color{gray} 1}  &  {\color{gray} 2}  & {\color{gray} 3} & {\color{gray} 4}  &  {\color{gray} 5}  &  {\color{gray} 6}  & {\color{gray} 7} & {\color{gray} 8} \\};
        \draw ([yshift=-5pt]m-1-4.south) node{\footnotesize \color{Cerulean} $s'$};
        \draw ([yshift=-5pt]m-1-5.south) node{\footnotesize \color{Cerulean} $t\smash\cdot\smash s$};
        \draw ([yshift=-5pt]m-1-9.south) node{\footnotesize \color{Cerulean} $t^3\smash\cdot \smash s$};
        \draw ([yshift=-5pt]m-2-6.south) node{\footnotesize \color{Cerulean} $t^2$};
        \draw ([yshift=-5pt]m-2-10.south) node{\footnotesize \color{Cerulean} $t^4$};
    \draw[thick, -stealth] ([xshift=5pt]m-3-1.east) -- ([xshift=5pt,yshift=20pt]m-1-1.east) node [left] {$q$} ;
    \draw[thick, -stealth] ([yshift=5pt]m-3-1.north) -- ([xshift=20pt,yshift=5pt]m-3-10.north) node [below] {$p$} ;
    \end{tikzpicture} \]
    and for degree reasons, no non-trivial differential can exist (on the $E^2$-page, differentials have degree (-1,2)). Thus, the spectral sequence collapses, and this finishes the proof.
\end{proof}

\begin{Lemma}\label{Prop:pairings} For any $E\in \SHR^{\eff,\heartsuit}$, there is an equivalence 
    $$\map_{\SH(\R)}(\HZtilde \Smash E, E) \lsimeq{} \map_{\SH(\R)}(E, E)$$
    induced by precomposition by the map $E \simeq \unit \Smash E \to \HZtilde \Smash E$ (given by the unit of the $\Einfty$-ring $\HZtilde$).
    
    In particular, the pairings $\HZtilde \Smash \HZtilde  \to  \HZtilde$ and $\HZtilde \Smash \Sigma^{1,0}\HZmod \to \Sigma^{1,0}\HZmod$ for the multiplicative spectral sequence of Figure \ref{Fig:smallsseq} (before taking real realization and homotopy group) are given, respectively, by the multiplication of $\HZtilde$ and the composition $\HZtilde \Smash \HZmod \to \HZ \Smash \HZmod \to \HZmod$ induced by the natural map from Milnor--Witt K-theory to Milnor K-theory and the $\HZ$-algebra structure on $\HZmod$. 
\end{Lemma}
\begin{proof} First, we claim that, for any $F,G \in \SHR^\eff_{\geq 0}$, we have $(F\Smash G)_{\leq_e 0} \simeq (F_{\leq_e 0} \Smash G_{\leq_e 0})_{\leq_e 0}$ via the natural maps $F\to F_{\leq_e 0}$ and $G\to G_{\leq_e 0}$. Indeed, almost by definition of a t-structure, we have in $\SH(\R)^\eff_{\geq 0}$ a cofiber sequence $F_{\geq_e 1} \Smash G \to F \Smash G  \to F_{\leq_e 0} \Smash G $. The associated long exact sequence then gives an equivalence $(F \Smash G)_{\leq_e 0} \lsimeq{} (F_{\leq_e 0} \Smash G)_{\leq_e 0}$ (since $F_{\geq_e 1} \Smash G \in \SH(\R)^\eff_{\geq_e 1} \Smash \SH(\R)^\eff_{\geq_e 0} \subseteq \SH(\R)^\eff_{\geq_e 1}$). Applying the same argument while exchanging the roles of $F$ and $G$, we obtain our claim.

    Let $E\in \SH(\R)^{\eff,\heartsuit}$. Using the previous claim and $\HZtilde = \unit_{\leq_e 0}$ \cite[Lem.\ 12]{Bachmann-genslices}, we have
\begin{align*}
    \map_{\SHR}(\HZtilde \Smash E, E) &\simeq \map_{\SHR}((\HZtilde \Smash E)_{\leq_e 0}, E) \simeq \map_{\SHR}((\unit_{\leq_e 0} \Smash E_{\leq_e 0})_{\leq_e 0}, E)\\
    &\simeq \map_{\SHR}((\unit \Smash E)_{\leq_e 0}, E) \simeq \map_{\SHR}(\unit \Smash E, E).
\end{align*}
This equivalence of mapping spaces is induced by precomposition with $\unit \to \unit_{\leq_e 0} \simeq \HZtilde$ (tensored with $E$). Indeed, there is a commutative diagram
%  https://q.uiver.app/#q=WzAsNSxbMCwwLCJcXHVuaXQgXFx3ZWRnZSBFIl0sWzEsMCwiXFx1bml0X3tcXGxlcV9lIDB9IFxcd2VkZ2UgRV97XFxsZXFfZSAwfSBcXHNpbWVxXFxIWnRpbGRlIFxcd2VkZ2UgRSJdLFsyLDAsIkVfe1xcbGVxX2UgMH0gXFxzaW1lcSBFIl0sWzAsMSwiKFxcdW5pdCBcXHdlZGdlIEUpX3tcXGxlcV9lIDB9Il0sWzEsMSwiKFxcdW5pdF97XFxsZXFfZSAwfSBcXHdlZGdlIEVfe1xcbGVxX2UgMH0pX3tcXGxlcV9lIDB9Il0sWzEsMiwiXFx2YXJ0aGV0YSJdLFswLDEsIlxcdGF1XzEiXSxbMSw0XSxbMCwzLCJcXHRhdV8yIiwyXSxbMyw0LCJcXHRhdV8zIl0sWzQsMiwiXFx2YXJ0aGV0YSciLDIseyJzdHlsZSI6eyJib2R5Ijp7Im5hbWUiOiJkYXNoZWQifX19XV0=
\[\begin{tikzcd}
	{\unit \wedge E} & {\unit_{\leq_e 0} \wedge E_{\leq_e 0} \simeq\HZtilde \wedge E} & {E_{\leq_e 0} \simeq E} \\
	{(\unit \wedge E)_{\leq_e 0}} & {(\unit_{\leq_e 0} \wedge E_{\leq_e 0})_{\leq_e 0}.}
	\arrow["{\tau_1}", from=1-1, to=1-2]
	\arrow["{\tau_2}"', from=1-1, to=2-1]
	\arrow["\vartheta", from=1-2, to=1-3]
	\arrow[from=1-2, to=2-2]
	\arrow["{\tau_3}", from=2-1, to=2-2]
	\arrow["{\vartheta'}"', dashed, from=2-2, to=1-3]
\end{tikzcd}\]
where all maps in the square are induced by the natural maps in the truncation. The equivalence above sends a map $\vartheta$ to the composite $\vartheta'\circ\tau_3 \circ\tau_2 \simeq \vartheta \circ\tau_1$. 

We now apply this result to $E = \HZtilde$, or $E = \HZmod$ (recall that $\HZtilde,\HZmod \in\SHR^{\eff,\heartsuit}$ by \Cref{Prop:HZHZmodveryeff}). Since $\SHR(\HZtilde \Smash \Sigma^{1,0}\HZmod,\Sigma^{1,0}\HZmod) \simeq \SHR(\HZtilde \Smash \HZmod,\HZmod)$, we only have to show that the pairings induced on the subquotients of the tower correspond to the unit after precomposition by the natural map $\unit \to \unit_{\leq_e 0} \simeq \HZtilde$.
This can be proven by chasing through the construction of the pairing induced on the subquotients; which is ultimately obtained from the multiplication on $\ko$ (and the latter is unital). More precisely, the unit $\unit \to \unit_{\leq_e 0} = \HZtilde$ is exactly the one induced by 
$$\unit \longrightarrow \ko \longrightarrow \tilde{s}_0\ko = f_0(\tilde{s}_0\ko) \longrightarrow f_0((\tilde{s}_0\ko)_{\leq 0}) = f_0(\underline{\pi}_0(\tilde{s}_0\ko)) = s'_0\ko.$$
To check that $\unit \Smash s'_0\ko \to s'_0\ko\Smash s'_0\ko \to s'_0\ko$ is the unit, for example, using Diagram (\ref{Diagram:pairing}), we want to show that it holds true for $\unit \Smash (\tilde{s}_0\ko)_{\geq_e i} \to (\tilde{s}_0\ko)_{\geq_e i}$ for $i=0,1$. Since these pairings are induced by one we had on $\tilde{s}_0$, which is induced from that on the very effective filtration, we want to show it for the maps $\unit \Smash \tilde{f}_i\ko \to \tilde{f}_i\ko$ for $i=0,1$, and these where ultimately induced by $\unit \Smash \ko \to \ko\Smash \ko \to \ko$, whence the claim follows.
\end{proof}

\begin{Lemma}\label{Prop:pi4C} With the same notation as in \Cref{Prop:aiscardinality}, we have $|\pi_4(\rR C)|=16$.
\end{Lemma}
\begin{proof} 
Applying \Cref{Prop:spectralsequence} again, this time with $X_n = \rR(\tilde{f}_{-n}\ko)$ for all $n \in \Z$ (by convention $\tilde{f}_{n}\ko = \ko$ for $n<0$), we obtain a spectral sequence concentrated in the first quadrant, with $E^1_{p,q} = \pi_p(\tilde{s}_q\ko)$ for all $p,q\in\N$. As we will see in Figure \ref{Fig:bigsseq}, the assumptions for strong convergence are satisfied. We have $\colim_n X_n \simeq \rR\ko$, and $\lim_n X_n \simeq 0$ by the same argument as in the beginning of the proof of \Cref{Prop:pistilde0ko}, since a very $n$-effective spectra is in particular $n$-connective in the effective homotopy t-structure. Indeed, we have $\SH(k)^\veff(n) := T^{\Smash n} \Smash \SH(k)^\veff \subseteq \Sigma^{n}\SH(k)^\veff = \SH(k)^\eff_{\geq n}$ by \Cref{Prop:propertiesofSHkeff}$(i)$. Therefore, our spectral sequence converges strongly to $\pi_p(\rR\ko)$. We know the latter groups by \Cref{Prop:homgprRko}, and we will use this to compute several differentials in the $E^1$-page. Then, since the same spectral sequence for $\ko$ replaced with $C$ is given by the truncation above the third line, this will give us information about the homotopy groups of $\rR(C)$.\\

Using the explicit description of the very effective slices of $\ko$ (\Cref{Prop:veffslicesko}) and the computation of the real realizations of these slices (\Cref{Prop:rRHZHZmod,Prop:rRHZtilde,Prop:pistilde0ko}), the $E^1$-page of our spectral sequence reads as in Figure \ref{Fig:bigsseq} below. The notation is as in Figure \ref{Fig:smallsseq}. We will explain just below how to determine the differentials, but let us first see how this implies the result. The spectral sequence for $\rR C=\rR(\ko/\tilde{f}_4\ko)$ is the truncation of that for $\rR\ko$ above the third horizontal line (indeed, $\tilde{s}_i C = \tilde{s}_i\ko$ for $0\leq i \leq 3$, and 0 otherwise). In particular, the differentials must be the same and in the fourth column, all terms on the $E^1$-page survive to the $E^\infty$-page. Therefore, $\pi_4(\rR C)$ is an extension of four copies of $\Z/2$; it has cardinality 16 (we can even identify it as $\Z/16$ because we saw in the proof of \Cref{Prop:aiscardinality} that it was cyclic), as desired.

\begin{figure}[h]
\centering
\begin{tikzpicture}
\matrix (m) [matrix of math nodes,
    nodes in empty cells,nodes={minimum width=5ex,
    minimum height=5ex,outer sep=-5pt},
    column sep=3.5ex,row sep=3ex]{
           {\color{gray} 6}     &  0 &  0  & 0 & 0 & 0 &  0  & \Zmod \\
           {\color{gray} 5}     &  0 &  0  & 0 & 0 & 0 &  \Zmod  & \Zmod \\
           {\color{gray} 4}     &  0 &  0  & 0 & 0 & \Z &  0  & \Zmod \\
           {\color{gray} 3}     &  0 &  0  & 0 & 0 & 0 &  0  & 0 \\
           {\color{gray} 2}     &  0 &  0  & \Zmod & 0 & \Zmod &  0  & \Zmod \\
           {\color{gray} 1}     &  0 &  \Zmod  & \Zmod & \Zmod & \Zmod &  \Zmod  & \Zmod \\
          {\color{gray} 0}    &  \Z  & 0 &  \Zmod  & \Zmod & \mathclap{\substack{\text{Extension}\\\text{of two } \Zmod\text{'s}}}  & 0 &  \Zmod  \\
    \quad\strut &  {\color{gray} 0}  &  {\color{gray} 1}  &  {\color{gray} 2}  & {\color{gray} 3} & {\color{gray} 4}  &  {\color{gray} 5}  &  {\color{gray} 6}  \\
    \quad \strut &  {\color{MyGreen}\Z}  &  {\color{MyGreen}0}  &  {\color{MyGreen}0}  & {\color{MyGreen}0} & {\color{MyGreen}\Z} &  {\color{MyGreen}0}  &  {\color{MyGreen}0}  \\};
    \draw [draw= BrickRed, thick, -stealth] (m-7-4.north west) -- (m-6-3.south east) node[midway,inner sep = 1pt,above right] {\footnotesize \color{BrickRed} $f_1$} node[midway,inner sep = 1pt,below left] {\footnotesize \color{BrickRed} $\cong$};
    \draw [draw= BrickRed, thick, -stealth] (m-7-5.north west) -- (m-6-4.south east) node[midway,inner sep = 1pt,above right] {\footnotesize \color{BrickRed} $f_2$} node[midway,inner sep = 1pt,below left] {\footnotesize \color{BrickRed} $\cong$};
    \draw [draw= BrickRed, thick, -stealth] (m-7-8.north west) -- (m-6-7.south east) node[midway,inner sep = 1pt,above right] {\footnotesize \color{BrickRed} $f_5$} node[midway,inner sep = 1pt,below left] {\footnotesize \color{BrickRed} $\cong$};
    \draw [draw= BrickRed, thick, -stealth] (m-6-5.north west) -- (m-5-4.south east) node[midway,inner sep = 1pt,above right] {\footnotesize \color{BrickRed} $f_3$}node[midway,inner sep = 1pt,below left] {\footnotesize \color{BrickRed} $\cong$};
    \draw [draw= BrickRed, dashed, thick, -stealth] ([yshift=2pt]m-7-6.north west) -- (m-6-5.south east) node[midway,inner sep = 1pt,above right] {\footnotesize \color{BrickRed} $f_4$};
    \draw [draw= BrickRed, dashed,thick, -stealth] (m-6-7.north west) -- (m-5-6.south east) node[midway,inner sep = 1pt,above right] {\footnotesize \color{BrickRed} $f_6$};
    \draw ([yshift=-5pt]m-7-4.south) node{\footnotesize \color{Cerulean} $s'$};
    \draw ([yshift=-5pt]m-6-3.south) node{\footnotesize \color{Cerulean} $a$};
    \draw ([yshift=-5pt]m-7-8.south) node{\footnotesize \color{Cerulean} $t^2\smash\cdot\smash s'$};
    \draw ([yshift=-5pt]m-6-7.south) node{\footnotesize \color{Cerulean} $t^2\smash\cdot\smash a$};
    \draw ([xshift=-15pt]m-9-1.west) node{\color{MyGreen} Converges to:};
\draw[thick, -stealth] ([xshift=5pt,yshift=-35pt]m-7-1.east) -- ([xshift=5pt,yshift=20pt]m-1-1.east) node [left] {$q$} ;
\draw[thick, -stealth] ([xshift=-45pt,yshift=5pt]m-8-2.north) -- ([xshift=20pt,yshift=5pt]m-8-8.north) node [below] {$p$} ;
\draw[dashed] ($([xshift=-45pt]m-4-2)! 0.5!([xshift=-45pt]m-3-2)$) -- ($([xshift=60pt]m-4-7)! 0.5!([xshift=60pt]m-3-7)$);
\draw [thick, decorate,decoration={brace,amplitude=4pt},xshift=0.3cm,yshift=0pt]
      ([xshift=-50pt,yshift=5pt]m-8-2.north) -- ($([xshift=-50pt]m-4-2)! 0.5!([xshift=-50pt]m-3-2)$) node [midway,left,xshift=-5pt] {$\rR C$};
\end{tikzpicture}
\caption{\footnotesize $E^1$-page of the spectral sequence for $\rR\ko$ with respect to the very effective slice filtration.}\label{Fig:bigsseq}
\end{figure}
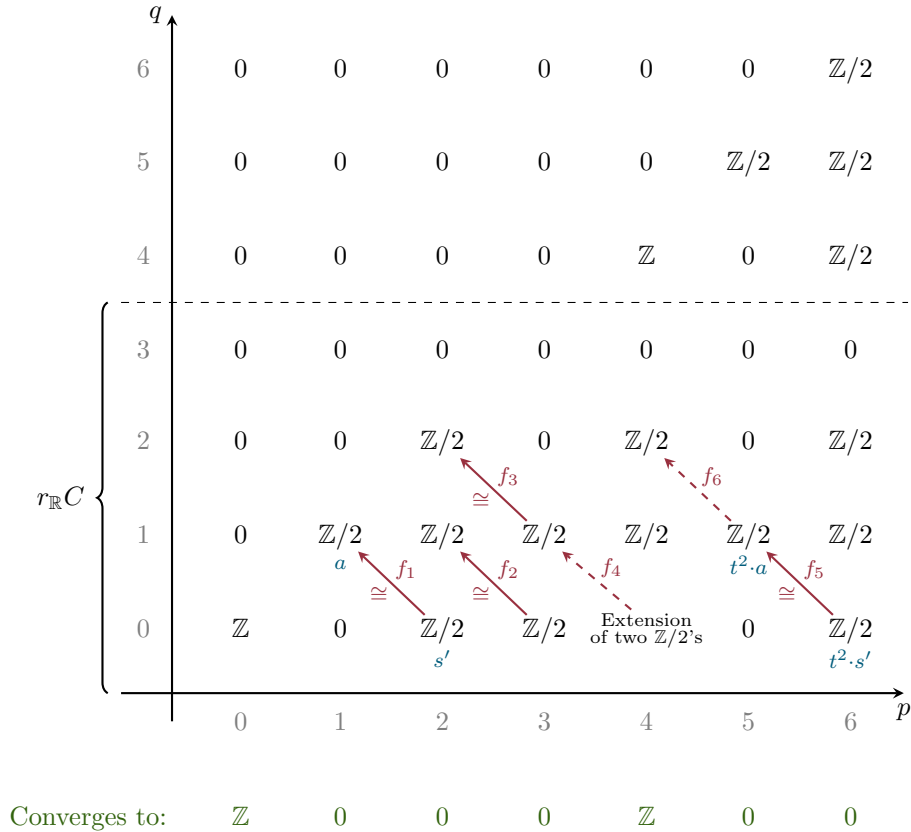

To conclude the proof, we only have to explain why the differentials are the ones displayed in Figure \ref{Fig:bigsseq}. First, note that by \Cref{Prop:homgprRko}, the spectral sequence for $\rR\ko$ converges to the groups displayed in the last line in Figure \ref{Fig:bigsseq}, i.e., to $\Z$ in every degree divisible by $4$, and $0$ otherwise.
\begin{itemize}
    \item No non-zero differential with source $E^1_{1,1} = \Z/2$ can exist. So it must receive a non-trivial differential on some page; the only possibility is that $d^1: E^1_{2,0} \to E^1_{1,1}$, namely $f_1$, is an isomorphism. 
    \item The same argument shows that $f_2$ is an isomorphism.
    \item The $E^1_{2,2}$-term could a priori also receive a non-trivial differential $d^2$ from $E^2_{3,0}$ on the $E^2$-page; but since $f_2$ is an isomorphism, $E^2_{3,0} = 0$. Thus, $f_3$ must be an isomorphism. 
    \item Since $d^1\circ d^1=0$, we have $f_3\circ f_4 = 0$, so by the previous bullet point $f_4=0$. The $E^1_{4,0}$-term then survives to the $E_\infty$-page. So does the $E^1_{4,1}$-term (it cannot be reached by any non-trivial differential). 
\end{itemize}
\vspace{0.2cm}

The multiplicative structure on the spectral sequence provides the missing information to determine $f_5$ and $f_6$. The multiplicative structure comes this time from the pairing $\tilde{f}_n\KO \Smash \tilde{f}_m \KO \to \tilde{f}_{n+m}\KO$ obtained as follows: for any $n,m\geq 0$, the composition 
$\tilde{f}_n\KO \Smash \tilde{f}_m\KO \to \KO\Smash \KO \to \KO$ given by the natural maps, followed by multiplication, lifts to $\tilde{f}_{n+m}\ko$, because 
\begin{align*}\SHR^\veff(n) \Smash \SHR^\veff(m) &= (T^{\Smash n} \Smash \SHR^\veff) \Smash (T^{\Smash m} \Smash \SHR^\veff) \\
&\subseteq (T^{\Smash (n+m)} \Smash \SHR^\veff) = \SHR^\veff(n+m).
\end{align*}
The compatibility between these pairings come from the $\Einfty$-ring axioms for $\KO$. Then the pairing descends to the slices as Diagram (\ref{Diagram:pairing}). We claim that the pairing $\tilde{s}_0\ko \Smash \tilde{s}_1\ko \to \tilde{s}_1\ko$ factors through $(\tilde{s}_0\ko)/(s'_1\ko)\  \Smash\  \tilde{s}_1\ko \simeq s'_0\ko \Smash \tilde{s}_1\ko$. Indeed, the latter is the cofiber of $s'_1\ko \Smash \tilde{s}_1\ko \to \tilde{s}_0\ko \Smash \tilde{s}_1\ko$, whose post composition with the pairing is zero. This holds because $\Sigma_T^{-1}(s'_1\ko \Smash \tilde{s}_1\ko \to \tilde{s}_1\ko)$ is a map $\Sigma^{1,0}\HZmod \wedge \HZmod \to \HZmod$, so it must be zero by the axioms of a t-structure (the source belongs to $\SHR^\eff_{\geq 1}$ whereas the right-hand side belongs to  $\SHR^\eff_{\leq 0}$ (\Cref{Prop:HZHZmodveryeff})). Moreover, the map $\tilde{s}_0\ko \Smash s'_1\ko \to \tilde{s}_0\ko \Smash \tilde{s}_0\ko \to \tilde{s}_0\ko$ induced by the natural map and the pairing lifts to $s'_1\ko$ on the target, and then by the same argument as above this lift factors through $s'_0\ko \Smash s'_1\ko$. Then, since $\tilde{s}_1\ko \simeq \Sigma^{2,1}\HZmod$, we can use the exact same strategy as we did for $s'_1\ko \simeq \Sigma^{1,0}\HZmod$ to show that $s'_0\ko \Smash \tilde{s}_1\ko \to \tilde{s}_1\ko$ is the $(2,1)$ suspension of the pairing $\HZtilde \Smash \HZmod \to \HZmod$ previously considered. In particular, we deduce the following:
\begin{itemize}
    \item Let us also denote by $s'$ a generator of $\pi_2(\rR\tilde{s}_0\ko) \cong \Zmod$ (it comes from the generator $s'$ of $\pi_2(\rR s'_1\ko)$ in Figure \ref{Fig:smallsseq}). Then, a preimage in $\pi_4(\rR\tilde{s}_0\ko)$ for $t^2 \in \pi_4(\rR s'_0\ko)$ acts upon $s'$ by sending it to the image of $t^2\cdot s' \in \pi_6(\rR s'_1\ko)$ in $\pi_6(\rR \tilde{s}_0\ko) \cong \Zmod$, which is a generator, as we can see from the spectral sequence for $\rR\tilde{s}_0\ko$ (Figure \ref{Fig:smallsseq}).
    \item Let $a$ be a generator of $\pi_1(\rR\tilde{s}_1\ko) \cong \Zmod$. Then a preimage in $\pi_4(\rR\tilde{s}_0\ko)$ for $t^2 \in \pi_4(\rR s'_0\ko)$ acts upon $a$ by sending it to a generator in $\pi_5(\rR\tilde{s}_1\ko) \cong \Zmod$ (this is the action of $\rR\HZtilde$ on $\rR\Sigma^{2,1}\HZmod$ as we have seen it before, for the spectral sequence in Figure \ref{Fig:smallsseq}). 
    \item Therefore, since by the Leibniz rule for $d^1$, we have 
    $$f_5(t^2\cdot s') = d^1(t^2\cdot s') = t^2 \cdot d^1(s') + d^1(t^2)\cdot s'= t^2 \cdot a \in E^1_{5,1}\cong\Zmod $$
    (since $f_1$ is an isomorphism and $f_4$ is trivial), $f_5$ maps a generator to a generator and is an isomorphism. In particular, since $d^1\circ d^1=0$, we also get $f_6=0$, and the $E^1_{4,2}$-term survives to the $E_\infty$-page.
\end{itemize}
This finishes the proof.
\end{proof}
\vspace{0.2cm}
    
\subsection{Comparison with the L-theory of \texorpdfstring{$\R$}{the real numbers}}\label{Subsect:compwithL(R)}\hfill\vspace{0.2cm}

\begin{thm}\label{Prop:compwithLtheory} There is an equivalence of $\E{1}$-rings
    $$\rR\ko \simeq \mathsf{L}(\R)_{\geq 0}$$ between the real realization of the very effective Hermitian K-theory spectrum and the connective cover of the L-theory spectrum of $\R$.
\end{thm}
\begin{proof} It suffices to prove that the 2-local fracture square for $\mathsf{L}(\R)_{\geq 0}$ is the same as in \Cref{Prop:rRko}. Such a fracture square appears in \cite[p.\ 3]{HLN}. We compute the fracture square for $\mathsf{L}(\R)_{\geq 0}$ following the same strategy as for $\rR\ko$ (\Cref{Prop:rRko}).
    
To identify the localization away from 2, we consider the same map $\kotop \to \mathsf{L}(\R)_{\geq 0}$ as in the proof of \Cref{Prop:rRko}, from \cite{LNS}; recall that this map is an equivalence after localizing at 2, and induces multiplication by 8 on the fourth homotopy groups before localization. We post-compose it with the automorphism of $\kotop[1/2]$ given by the (connective cover of the) second Adams operation $\psi^2$ as in the proof of \Cref{Prop:rRko}, and obtain $\mathsf{L}(\R)_{\geq 0} \to \kotop[1/2]$ sending $b$ to $\beta_4/2$ in the fourth homotopy groups.

To identify the localization at $(2)$, we consider the map of $\Einfty$-rings $\mathsf{MSO} \to \mathsf{L}(\R)$ from \cite[Thm.\ 3.4 and Rmk.\ 3.5]{HLN}. By \cite[Cor.\ 3.7]{HLN}, $\mathsf{MSO}_{(2)}$ receives an $\E{2}$-map from $\HZ_{(2)}$. Therefore, $\mathsf{L}(\R)_{(2)}$ receives an $\E{2}$-map from $\HZ_{(2)}$. By \Cref{Prop:freeE1Ralgebra}$(iii)$, it then also receives a map of $\E{1}$-rings from the free $\E{1}$-$\HZ_{(2)}$-algebra $\HZ_{(2)}[t^4]$, sending $t^4$ to $b$ in $\pi_4$. Since the domain is connective, this map lifts to $\mathsf{L}(\R)_{\geq 0}$. The ring structure on the homotopy groups on both sides implies that this map is an equivalence (indeed, by \cite[Prop.\ 4.1]{HLN}, $\pi_\ast(\mathsf{L}(\R)) \cong \Z[b]$ with $|b|=4$, and by \Cref{Prop:freeHA}$(ii)$, $\pi_\ast(\HZ_{(2)}[t^4]) \cong \Z[t^4]$ is the same ring).  

For the rationalization, by \Cref{Rmk:chQ} we have $\kotop[1/2]_\Q \simeq \HQ[t^4]$ via the Chern character map.

Finally, we have to determine what the map $\HZ_{(2)}[t^4] \to \HQ[t^4]$ in the pullback square is, corresponding to our choices of identifications. Studying the actions of these maps on the fourth homotopy groups, we see that the composite $\mathsf{L}(\R)_{\geq 0} \to \kotop[1/2] \to \HQ[t^4]$ maps $b$ to $t^4$. Since the map $\mathsf{L}(\R)_{\geq 0} \to \HZ_{(2)}[t^4]$ we have chosen sends $b$ to $t^4$, the bottom horizontal map must send $t^4$ to $t^4$.
\end{proof}

\begin{Question} One may ask whether this equivalence can be refined to one of $\mathbb{E}_\infty$-rings. The author does not know whether this is the case.
\end{Question}

\newpage
\section*{Index of notation}\label{Sect:indexofnotation}

\vspace{-1em}
\begin{table}[H]
  \centerline{%
    \begin{tabular}[h]{lll}
 $\Affl$ & Affine line over some base scheme & \\
 $\Alg(\calC^\otimes)$, $\CAlg(\calC^\otimes)$ & Categories of $\E{1}$-, resp.\ $\Einfty$-algebras in a symmetric monoidal category $\calC^\otimes$ & \S \ref{Subsect:freeE1HA} \\
 $\Cat$ & Category of small ($\infty$-)categories \\
 $\eta$ & Motivic Hopf map & \S \ref{Subsect:rRKGLkgl}\\
 $\E{n}$ & $\infty$-operad of little $n$-cubes & \S \ref{Subsect:freeE1HA}\\
 $f_n$, $\tilde{f}_n$ & $n$-effective, respectively very $n$-effective cover functors & \ref{Def:effcover}\\
 $\Gm$ & Group scheme $S\times_{\Spec(\Z)} \Spec(\Z[t,t^{-1}])$ over a base scheme $S$ & \\
  $\HA[t^n]$ & Free $\E{1}$-$\HA$-algebra on one generator in degree $n$ & \ref{Def:freeE1HA}\\
 $\HZ, \HZmod$ & Motivic or topological Eilenberg-Mac Lane spectra & \ref{Def:HZ}\\
 $\HZtilde$ & Effective cover of the Milnor-Witt K-theory sheaf: $f_0\underline{K}^{MW}_\ast$ & \S \ref{Subsect:rRHZtilde} \\
 $\underline{K}^{MW}_\ast$, $\underline{K}^{M}_\ast$, $\underline{K}^{W}_\ast$ & Milnor-Witt, Milnor, and Witt K-theory sheaves & \ref{Def:HZ}\\
 $\KGL, \kgl$ & Algebraic K-theory motivic spectrum and its very effective cover & \ref{Subsect:rRKGLkgl}\\
 $\KO, \ko$ & Hermitian K-theory motivic spectrum and its very effective cover & \ref{Def:KO}\\
 $\KOtop, \kotop$ & Real topological K-theory spectrum and its connective cover & \ref{Subsect:rRKO}\\
 $\KU$ & Complex topological K-theory spectrum & \ref{Def:Cherncharacter} \\
 $\KW, \mathsf{kw}$ & Balmer--Witt K-theory motivic spectrum and its connective cover & \ref{Def:KO}\\
 $\mathsf{L}(-)$ & L-theory spectrum functor & \ref{Prop:rRKO}\\
 $\mathsf{Map}_\calC$ & Mapping space in a category $\calC$ & \\
 $\MSL$ & Special linear cobordism motivic Thom spectrum & \S \ref{SubSubsect:rRko(2)}\\
 $\MSO$ & Oriented cobordism Thom spectrum & \S \ref{SubSubsect:rRko(2)}\\
 $\Omega^\infty, \Omega^\infty_T$ & Infinite loop space functors for classical, resp.\ motivic ($\mathbb{P}^1$-)spectra & \\
 $\homsh{\ast}{\ast}{-}$, $\underline{\pi}_{\ast,\ast}(-)$ & Homotopy sheaves of a motivic space or spectrum & \\
 $\Pre(\calC)$ & Category of presheaves of spaces on a category $\calC$ & \\
 $\rho$ & Inclusion $\calS^0 \to \Gm$ on the points corresponding to 1 and $-1$ & \ref{Prop:rRinvertingrho}\\
 $\rR$, $\rC$ & Real and complex Betti realization functors & \S \ref{Subsect:rR}\\
 $\spaces$ & Category of spaces & \\
 $\calS$ & Motivic sphere spectrum & \\
 $\calS^n$ & (Simplicial) $n$-sphere in the category of motivic spaces or motivic spectra & \\
 $\calS^{p,q}$ & Motivic bigraded spheres, $\calS^{p,q} = \calS^{p-q} \Smash \Gm^{\Smash q}$ & \\
 $s_n$, $\tilde{s}_n$ & $n$-effective, respectively very $n$-effective slices & \ref{Def:effcover}\\
 $\Sph$ & Topological sphere spectrum & \\
 $\Sph^n$ & $n$-sphere in the category of spaces or spectra & \\
 $\SH(S)$ & Category of motivic ($\mathbb{P}^1$-)spectra over $S$ & \\
 $\SH(S)^\eff$, $\SH(S)^\veff$ & Subcategories of effective, respectively very effective spectra over $S$ & \ref{Def:SHkeff}\\
 $\SH(S)^\eff(n)$, $\SH(S)^\veff(n)$ & Subcategories of $n$-effective, respectively very $n$-effective spectra over $S$ & \ref{Def:SHkeff}\\
 $\Sigma_T$ & Functor $ T \Smash - : \SH(S) \to \SH(S)$ & \\
 $\Sigma^\infty$ & Infinite suspension functors for topological or motivic spectra & \\
 $\Sm_S$ & 1-category of smooth schemes of finite type over the base scheme $S$ & \\
 $\Sp$ & Category of (topological) spectra & \\
 $\Spc(S)$ & Category of motivic spaces over the base scheme $S$ & \\
 $T$ & Pointed motivic space or spectrum corresponding to $(\mathbb{P}^1,\infty)$ & \\
 $\tau_{\geq n}, \tau_{\leq n}$ & Truncation functors for a $t$-structure\\
 $(-)_+$ & Disjoint base point functor (free functor from unpointed to pointed objects) & \\
 \end{tabular}%
  }
\end{table}

\paragraph{\textbf{Acknowledgments:}} The research presented in this article was conducted for my Master's thesis, under the supervision of Prof.\ Tom Bachmann, in JGU Mainz. I would like to express my deepest gratitude to him, and to my second supervisor at EPFL, Prof.\ Jérôme Scherer. Many thanks are also due to Anton Engelmann, Klaus Mattis, Luca Passolunghi, and professors Marc Hoyois and Markus Land for very helpful discussions. Moreover, I wish to thank all the members of the AGTZ group for welcoming me so warmly in Mainz, and for insightful mathematical discussions. I am also grateful to two anonymous referees, and Qi Zhu, for useful comments about an earlier version of this article. Finally, I acknowledge financial by the EPFL Student Support Program, and by the Deutsche Forschungsgemeinschaft (DFG, German Research Foundation) through the Collaborative Research Center TRR 326 Geometry and Arithmetic of Uniformized Structures, project number 444845124.\\

\paragraph{\textbf{Competing interests:}} The author declares none.

\bibliographystyle{abbrv}  
\bibliography{biblio.bib}\label{biblio}
\vspace{1cm}
\end{document}